\documentclass[12pt]{amsart}
\usepackage{amsmath}
\usepackage{amssymb}
\usepackage{graphicx}
\usepackage{pinlabel}
\usepackage[all]{xy}

\usepackage{amsmath,amssymb,amsfonts,amsthm,graphicx}

\usepackage{pinlabel}
\newtheorem{dummy}{dummy}[section]
\newtheorem{lemma}[dummy]{Lemma}
\newtheorem{theorem}[dummy]{Theorem}

\newtheorem{corollary}[dummy]{Corollary}
\newtheorem{proposition}[dummy]{Proposition}

\theoremstyle{definition}
\newtheorem{definition}[dummy]{Definition}

\newtheorem{notation}[dummy]{Notation}
\newtheorem{example}[dummy]{Example}
\newtheorem{remark}[dummy]{Remark}

\newtheorem{observation}[dummy]{Observation}

\newtheorem{property}[dummy]{Property}

\numberwithin{equation}{section}

\newcommand{\R}{\mathbb{R}}

\newcommand{\Z}{\mathbb{Z}}

\newcommand{\FF}{\mathcal{F}}

\newcommand{\bfe}{\mathbf{e}}

\newcommand{\ShS}{\mathbf{Sh}^\bullet_{\mathcal{S}}(M\times \R, \mathbb{K})}
\newcommand{\ShLambda}{\mathbf{Sh}^\bullet_{\Lambda}(M\times \R, \mathbb{K})}

\begin{document}
\title[Sheaves via augmentations of Legendrian surfaces]{Sheaves via augmentations of Legendrian surfaces}

\author{Dan Rutherford}
\address{Ball State Unversity}
%\email{Dan's email here}

\author{Michael Sullivan}
\address{University of Massachusetts Amherst}
%\email{Mike's email here}

\begin{abstract}  
Given an augmentation for a Legendrian surface in a $1$-jet space, $\Lambda \subset J^1(M)$, we explicitly construct an object, $\mathcal{F} \in \ShLambda$,  of the (derived) category from \cite{STZ} of constructible sheaves on $M\times \R$ with singular support determined
 by $\Lambda$.  In the construction, we introduce a simplicial Legendrian DGA (differential graded algebra) for Legendrian submanifolds in $1$-jet spaces that, based on \cite{RuSu1,RuSu2,RuSu25}, is equivalent to the Legendrian contact homology DGA in the case of Legendrian surfaces.  
In addition, we extend the approach of \cite{STZ} for $1$-dimensional Legendrian knots to obtain a combinatorial model for sheaves in $\ShLambda$ in the $2$-dimensional case.  
\end{abstract}

\maketitle

{\small \tableofcontents}

\section{Introduction}
\label{sec:Introduction}

In recent years, microlocal sheaf theory has been shown to be a potent tool for establishing global results in symplectic and contact topology.  This is especially true in the setting of cotangent bundles and $1$-jet spaces where the method seems to have capability comparable to that of $J$-holomorphic curve theory, cf. \cite{Nadler, NadlerZ, Gu19}.  The main results of this article provide an explicit construction of sheaves from $J$-holomorphic curve invariants in the setting of  Legendrian surfaces in $1$-jet spaces.

Let $\Lambda \subset J^1(M) = T^*M \times \R$ be a closed Legendrian submanifold in the $1$-jet space of a manifold $M$.  A 
$J$-holomorphic curve based invariant of $\Lambda$ is its Legendrian contact homology DGA (differential graded algebra), $(\mathcal{A}, \partial)$, constructed in \cite{EES07}. It is an associative algebra generated by the Reeb chords of $\Lambda,$ with a count of holomorphic disks defining the differential.   An invariant of $\Lambda$ arising from microlocal sheaf theory \cite{GKS, STZ} is the (dg)-derived category, $\ShLambda$, of constructible sheaves on $M\times \R$ with singular support contained in the union of the zero section of $T^*(M\times \R)$ and the Lagrangian cylinder over $\Lambda \subset J^1M \cong S^-T^*(M\times \R) \subset ST^*(M\times \R)$.  That these objects should be related was already conjectured in \cite{STZ}, and indeed in \cite{NRSSZ} for $1$-dimensional Legendrians in $J^1\R$ an equivalence was established between a sub-category $\mathcal{C}_1(\Lambda; \mathbb{K}) \subset \mathbf{Sh}^\bullet_{\Lambda}(\R^2, \mathbb{K})$ and an $A_\infty$-category, $\mathit{Aug}_+(\Lambda; \mathbb{K})$, defined from $(\mathcal{A}, \partial)$.  The objects of $\mathit{Aug}_+(\Lambda; \mathbb{K})$ are {\it augmentations}, i.e. DGA homomorphisms $(\mathcal{A},\partial) \rightarrow (\mathbb{K}, 0)$.  

In this article we give a construction of sheaves\footnote{Here and elsewhere in the article we refer to objects of $\ShLambda$ as ``sheaves'' even though they are in fact cochain complexes of sheaves.} in $\ShLambda$ from augmentations in the case that $\Lambda \subset J^1M$ is a Legendrian surface with mild front singularities, i.e. generic singularities and no swallow tail points.  By the work of Entov \cite{Entov}, when $M$ and $\Lambda$ are orientable, after a Legendrian isotopy $\Lambda$ can be assumed to have mild front singularities, so little generality is lost by this hypothesis.  Central to our approach is the use of the cellular DGA \cite{RuSu1}, which is a DGA associated in a local, formulaic manner to a Legendrian surface equipped with a choice of compatible polygonal decomposition of the base space $M$.  Working over $\Z/2$, in \cite{RuSu2, RuSu25} the cellular DGA was shown to be stable tame isomorphic to the Legendrian contact homology DGA.   In this paper,  we introduce and work with a simplicial DGA using $\Z$-coefficients that agrees with the cellular DGA associated to a simplicial  decomposition  $\mathcal{E}$ of $M$ 
when coefficients are reduced mod $2$.

\begin{theorem}  \label{thm:main}
Let $\Lambda \subset J^1M$ be a Legendrian surface with mild front singularities equipped with a $\Z$-valued Maslov potential, $\mu$, and a compatible simplicial decomposition, $\mathcal{E}$.  The construction of this article provides a map 
\[
\Phi: \mathit{Aug}(\Lambda; \mathbb{K}) \rightarrow \ShLambda
\]
from the set of $\mathbb{K}$-valued augmentations of the simplicial DGA of $(\Lambda, \mu, \mathcal{E})$ to  $\ShLambda$.
\end{theorem}  

In particular, we have:

\begin{corollary}  \label{cor:main}
Let $\Lambda \subset J^1M$ be a Legendrian surface with $\Lambda$ and $M$ orientable.  
If the Legendrian contact homology DGA of $\Lambda$ has an augmentation to $\Z/2$, then there exists a non-constant sheaf $\mathcal{F} \in \mathbf{Sh}^\bullet_{\Lambda}(M\times \R, \mathbb{Z}/2)$.
\end{corollary}
  
For a discussion of other works relating augmentations and sheaves, see \ref{sec:works} below.  An interesting feature of our approach is that in the construction of cochain complexes of sheaves from augmentations, the underlying sheaves of $\mathbb{K}$-modules that constitute the terms in the complexes  are determined by $(\Lambda, \mu, \mathcal{E})$; only the differentials depend on the particular augmentation.  Moreover, the construction has a very concrete nature with both sides, augmentations and sheaves, reduced to explicit  combinatorial structures.

\subsection{Overview and organization}

The construction of sheaves from augmentations of the simplicial DGA is summarized in the following schematic:

\[
 \framebox{$\begin{array}{c} \mbox{Augmentations of} \\ \mathcal{A}(\Lambda, \mathcal{E}) \end{array}$} =   
 \framebox{$\begin{array}{c} \mbox{Chain homotopy} \\ \mbox{diagrams for} \\ (\Lambda, \mathcal{E}) \end{array}$} \rightarrow 
\framebox{$\begin{array}{c} \mbox{Combinatorial sheaves in} \\ \mathbf{Fun}^\bullet_\Lambda(\mathcal{S}, \mathbb{K}) \end{array}$} \cong
\framebox{$\begin{array}{c} \mbox{Sheaves in} \\ \ShLambda \end{array}$} 
\]

In Section \ref{sec:SimplicialAlgebra}, we introduce the simplicial DGA, $\mathcal{A}(\Lambda, \mathcal{E})$, for Legendrians of any dimension having mild front singularites.  
The exposition emphasizes similarities with the cobar DGA of the usual simplicial chain complex of $\mathcal{E}$ which become useful later in the article.  From this perspective, it is no more difficult to define $\mathcal{A}(\Lambda, \mathcal{E})$ for Legendrians of any dimension, and the definition is made somewhat more efficient by treating all dimensions  simultaneously; that the result is a DGA is established in Theorem \ref{thm:dg}.  
With coefficients reduced mod $2$, when $\dim(\Lambda) = 2$ the simplicial DGA becomes a special case of the cellular DGA from \cite{RuSu1}, and hence is stable tame isomorphic to the Legendrian contact homology DGA.  It is an open question if the simplicial DGA agrees with the usual LCH DGA when  $\dim(\Lambda) >2.$

Section \ref{sec:ConstructibleSheaves} establishes, in the case $\dim \Lambda =2$,  a combinatorial model for $\ShLambda$, building on the approach of \cite{STZ} used for $1$-dimensional Legendrian knots.  A {\it combinatorial sheaf} for $\Lambda$ is a functor  
   from the poset category of a certain stratification $\mathcal{S}$ of the front space $J^0(M)=M \times \R,$ to cochain complexes of $\mathbb{K}$-modules, subject to the restrictions of Definition \ref{def:fun}.
 Theorem \ref{thm:FunSh} translates these restrictions into singular support requirements for sheaves in $\ShLambda$. 
 The quasi-equivalence of categories  
  $\mathbf{Fun}^\bullet_\Lambda(\mathcal{S}, \mathbb{K}) \cong \ShLambda$ results as Corollary \ref{cor:Equiv}.

In Section \ref{sec:SDandGMC}, we review a way of packaging complexes, chain maps, and (higher) homotopies into an $n$-simplex, which we call a {\it simplex diagram.}  
  The second half of the section establishes properties of a generalized mapping cylinder construction that will later be used to produce fully commutative diagrams from simplex diagrams.  

In Section \ref{sec:CHD}, for Legendrians of arbitrary dimension, we establish a bijection between augmentations of the simplicial DGA of $(\Lambda, \mathcal{E})$ and {\it chain homotopy diagrams} (abbrv. {\it CHDs}) which are collections of maps associated to the simplices of $\mathcal{E}$ that assemble appropriately to form simplex diagrams whose underlying vector spaces are spanned by the sheets of $\Lambda$.  The case of surfaces with $\Z/2$-coefficients was established in \cite{RuSu3} using a definition of CHD in the $2$-dimensional case that is equivalent to but slightly different in appearance from the one used here.

Finally, in Section \ref{sec:SheafConstruction}, we construct a combinatorial sheaf from a CHD for a Legendrian surface $\Lambda$.  The section concludes by discussing possible modifications of the construction for the case of higher dimensional DGA representations and sheaves with micro-local rank larger than $1$.

\subsection{Related works}  \label{sec:works}  As mentioned above, for $1$-dimensional Legendrian knots in $\mathbb{R}^3$ a categorical equivalence is established in \cite{NRSSZ} between augmentations and sheaves of micro-local rank 1.  We leave it as a possible direction for future research as to whether the methods of the current article can be extended beyond the object level to obtain a map of categories.  When the Legendrian knot is the rainbow closure of a positive two-braid, Chantraine, Ng and Sivek \cite{CNS} generalize \cite{NRSSZ} to an equivalence (at the level of cohomology) between $n$-dimensional representation categories, defined as in a construction of \cite{CDGG}, and sub-categories of $\ShLambda$ consisting of sheaves with microlocal rank $n$.

Correspondences between augmentations and sheaves have been previously obtained for some interesting, special classes of Legendrian surfaces.  In \cite{Gao}, Honghao Gao provides an elegant construction of augmentations from micro-local rank 1 sheaves for conormal tori of smooth knots and links in $\R^3$, which form a class of Legendrian tori in $J^1S^2$, and proves that it gives a bijection between suitably defined moduli spaces of objects.   
The synergy between microlocal sheaf invariants and $J$-holomorphic curves for conormal tori is also nicely illustrated in the exciting recent result, due to Vivek Shende, that knots in $\R^3$ are determined up to isotopy by the Legendrian isotopy type of their conormal tori \cite{Shende}.  Shende's first proof of this result made use of microlocal sheaf theory and was then followed with a second proof by Ekholm, Ng, and Shende using methods of $J$-holomorphic curve theory \cite{ENS}.  

In \cite{TZ}, Treumann and Zaslow introduce a $2$-sheeted Legendrian surface, $\Lambda_G$, associated to any tri-valent graph, $G \subset S^2$, and compute the moduli space of sheaves for $\Lambda_G$ in terms of colorings of the dual graph $\widehat{G}$.  Complementary results on the augmentation side are provided by Casals, Murphy, and Sackel in \cite{CasalsMurphy} where a DGA inspired by the Legendrian contact homology DGA of $\Lambda_G$ is combinatorially associated to such a trivalent graph $G$ with its moduli space of augmentations shown in the appendix to agree with the sheaf moduli space from \cite{TZ}. 

A quite general approach to relating augmentations and sheaves, by way of wrapped Fukaya categories, is available through the combination of works \cite{GPS, EL, BEE, Ekh}.  In \cite[Section 6.4]{GPS}, for $\Lambda \subset J^1\R^n$ an $A_\infty$-equivalence is established between the sub-category $\ShLambda_0 \subset \ShLambda$ of sheaves with acyclic stalks at $\infty$ and a category of $A_\infty$-modules of the endomorphism algebra, $\mathcal{A}_\Lambda = CW^*(D,D)_{T^*\mathbb{R}^{n+1}, \Lambda}$, of a Lagrangian linking disk for $\Lambda$ in a partially wrapped Fukaya category associated to $\Lambda$.  Moreover, \cite[Conjecture 3, Appendix B]{EL} defines an $A_\infty$-algebra map $\Psi$ from $\mathcal{A}_\Lambda$ to an enhanced version of the Legendrian contact homology DGA that uses chains on the loop space of $\Lambda$ to incorporate information from higher dimensional moduli spaces of holomorphic disks, and using the methods of \cite{BEE, Ekh} the map is a quasi-isomorphism.  

In addition to the above cited references, there is much current research activity related to the slogan from \cite{NRSSZ} that ``augmentations are sheaves''.  In particular, An, Bae, and Su have announced a categorical equivalence in the case of $1$-dimensional Legendrian graphs, cf. \cite{ABS}.  We mention also the work-in-progress of Bourgeois and Viterbo extending the sheaf quantization of \cite{Viterbo} to augmented Legendrians of arbitrary dimension.  

\subsection{Acknowledgments}
The first author is partially supported by grant 429536 from the Simons Foundation.  The second author is partially supported by grant 317469 from the Simons Foundation. He thanks the Centre de Recherches Mathematiques for hosting him while some of this work was done. The authors also thank Baptiste Chantraine, Honghao Gao, Stephane Guillermou, Vivek Shende, David Treumann, Eric Zaslow, and Mahmoud Zeinalian for educational conversations and e-mail correspondence.

\section{A simplicial Legendrian DGA}
\label{sec:SimplicialAlgebra}

In Section \ref{sec:Cobar}, we review the simplicial cobar construction.
In Section \ref{sec:Mild}, after reviewing Legendrian front singularities with emphasis on $2$-dimensional Legendrians,
we introduce a class of compatible simplicial decompositions (of the base space) for Legendrians with mild front singularities.  
  In Section \ref{sec:SimplicialReformulation}, we define %use the first cobar construction to describe 
 a simplicial DGA over $\Z$ for Legendrians submanifolds of any dimension having mild front singularities.
 When the Legendrian has dimension 1 or 2, using $\Z/2$ coefficients, this DGA is stable tame isomorphic to the Legendrian contact DGA defined with pseudo-holomorphic curves.

\subsection{Simplicial algebra}
\label{sec:Cobar}

Let $\mathcal{E} = \{e_I \}$ be a simplicial decomposition of a topological space $M$.  We use the following notations: The vertices of $\mathcal{E}$ are $e_0, \ldots, e_N$.
  The cells $e_I$ are the interiors of simplices (with closed simplices denoted $\overline{e_I}$).  Moreover, the index $I$ is the tuple $I =(i_0, \ldots,  i_m)$ with $i_0 < \ldots < i_m$ such that $\overline{e_I}$ has vertices $e_{i_0},\ldots, e_{i_m}$.  We sometimes write $e_I = e_{i_0\ldots i_m}.$   There is a partial order defined on $\mathcal{E}$  by 
\[
e_I \leq e_J \quad \Leftrightarrow \quad e_I \subset \overline{e_J} \quad \Leftrightarrow  \quad I \subset J.
\]
We will make use of gradings
\[
|e_I| = \dim e_I = m,  \quad  \quad   \left\|e_I\right\| = |e_I|-1
\]   
where $|e_I|$ is the degree of $e_I$ in the simplicial chain complex $C_*(M)$ and $\|e_I\|$ is the degree of $e_I$ in the cobar complex, $\Omega C_*(M)$, that will be reviewed momentarily. 

Let $\partial_s:C_*(M) \rightarrow C_*(M)$ be the simplicial boundary operator and $\Delta_s:C_*(M) \rightarrow C_*(M) \otimes C_*(M)$ be the Alexander-Whitney coproduct from simplicial homology,
  \[
\partial_s(e_{i_0\cdots i_m}) = \sum_{k = 0}^{m} (-1)^k e_{i_0 \cdots \widehat{i_k} \cdots \cdots i_m},
\quad
\Delta_s(e_{i_0\cdots i_m}) = \sum_{k = 0}^{m} e_{i_0\cdots i_k} \otimes e_{i_k\cdots i_m}.
\]
Let $\partial^{red}_s$ be the ``reduced" boundary operator defined as in $\partial_s$ but without the first and last terms in the above sum.  

The cobar complex of $(C_*(X), \partial_s, \Delta_s)$ is the DGA (differential graded algebra) $(\Omega C_*(X), D_s)$.  The underlying graded algebra $\Omega C_*(X)$ is the unital tensor algebra of $C_*(X)[1]$ where $[1]$ denotes the degree shift so that $C_*(X)[1]_m = C_{m+1}(X)$.  As a unital associative algebra, $\Omega C_*(X)$ is freely generated by simplices $e_I$ with grading $\|e_I\|$.  
  The differential satisfies 
\begin{eqnarray*}
D_s(e_{i_0\cdots i_m}) = -\sum_{k = 0}^{m} (-1)^k e_{i_0 \cdots \widehat{i_k} \cdots \cdots i_m} +  \sum_{k = 0}^{m} (-1)^ke_{i_0\cdots i_k}  \otimes e_{i_k\cdots i_m},
\end{eqnarray*}
and the graded Liebniz rule $D_s(xy) = D_s(x) y + (-1)^{\left\|x \right\|} x D_s(y)$.  A ``reduced" version $D_s^{red}$ is again defined by omitting the first and last term in the first sum.  
Observe the relation
\begin{equation} \label{eq:Dreduced}
D_s^{red} = \varphi \circ D_s \circ \varphi^{-1}
\end{equation}
where $\varphi:\Omega C_*(X) \rightarrow \Omega C_*(X)$ is the (grading preserving) algebra homomorphism that satisfies
\begin{equation} \label{eq:varphimap}
\varphi(e_I) =\left\{\begin{array}{cr} e_I -1, & \dim e_I = 1 \\ e_I, & \dim e_I \neq 1. \end{array}\right.
\end{equation} 

\begin{remark}
 These are well-known constructions. For example,  $D_s^{red}$ was used in \cite{RiveraZeinalian} to generalize Adams' result that the cobar of the singular chains of a space is quasi-isomorphic
  to the chains of its loop space to the non-simply connected case.
 \end{remark}

\subsection{Legendrians with mild front singularities}
\label{sec:Mild}

In a $1$-jet space $J^1M =T^*M\times \R$, standard coordinates $(x,y,z)$ with $(x,y) \in T^*M$ and $z \in \R$ arise from a choice of local coordinate $x$ on $M$.  We define the {\bf front and base projections} 
$$\pi_{xz}: J^1M \rightarrow J^0M := M \times \R, \quad \pi_{x}: J^1M \rightarrow M.$$
Let $\Lambda \subset J^1M$ be a closed Legendrian submanifold of any dimension $n= \dim M$, and assume $\Lambda$ has {\bf mild front singularities.}  
 This means that  the front projection $\pi_{xz}\big|_{\Lambda}$ is an immersion outside a codimension $1$ submanifold $\Lambda_{\mathit{cusp}} \subset \Lambda$ along which $\pi_{xz}\big|_{\Lambda}$ has a standard cusp singularity.  More precisely, near any point of $\Lambda_{\mathit{cusp}}$ there exists a choice of coordinate on $\Lambda$ and $J^0M$ such that $\pi_{xz}(t_1, \ldots, t_n) = (t_1^3,t_1^2,t_2, \ldots, t_n)$ and $\Lambda_{\mathit{cusp}}$ is the subset $\{t_1=0\}$.  Any point in $\Lambda \setminus \Lambda_{\mathit{cusp}}$, has a neighborhood $U \subset \Lambda$ that is the $1$-jet of a function on $\pi_x(U)$, in particular, the front projection $\pi_{xy}(U)$ is the graph of some local defining function $f:\pi_x(U) \rightarrow \R$.   A {\bf $\Z/m$-valued Maslov potential} on $\Lambda$ is a locally constant function
\[
\mu: \Lambda \setminus \Lambda_{\mathit{cusp}} \rightarrow \Z/m
\]
whose value increases by $1$ mod $m$ when crossing $\Lambda_{\mathit{cusp}}$ from the lower sheet of the cusp edge to the upper sheet (in terms of the $z$-coordinate).  A $\Z/m$-valued Maslov potential exists on $\Lambda$ if and only if $m$ divides the Maslov number $m(\Lambda)$, cf.
 \cite{EES05a}.  To simplify considerations involving grading, we will restrict attention primarily to the case where $m(\Lambda) =0$, so that $\Lambda$ can be equipped with a $\Z$-valued Maslov potential. 
 
\begin{remark}
Entov proves that for $\Lambda \subset J^1(M)$ with $\dim(\Lambda) = 2$ and $\Lambda, M$ both orientable, $\Lambda$ can be assumed to satisfy the mild front singularities condition after a Legendrian isotopy 
\cite[Section 1.8]{Entov}.  
  Alvarez-Gavela generalizes a higher-dimensional result in \cite{Entov} to prove that the mildness condition can be achieved with a small isotopy in all dimensions, providing that there exists a formal deformation that makes it mild \cite[Theorem 1.17]{AlvarezGavela}.
 \end{remark} 
 
Consider now a Legendrian surface, i.e. $\Lambda \subset J^1M$ with $\dim(\Lambda) = \dim(M) =2$, with mild front singularities.    
  After possibly applying a small Legendrian isotopy, we can assume that $\Lambda$ is self transverse in both the front and base projections.  Then,   
the {\bf{Legendrian front}}  $\pi_{xz}(\Lambda)$ is stratified as 
\begin{equation} \label{eq:xzStrat}
\pi_{xz}(\Lambda) =   \Lambda^F_0 \sqcup \Lambda^F_1 \sqcup \Lambda^F_2
\end{equation}
 where $\Lambda^F_k$ is the union of all points belonging to
  codimension $k$ singularities in $\pi_{xz}(\Lambda)$: $\Lambda^F_0$ consists of smooth (non-cusp) points of $\pi_{xz}$ where $\pi_{xz}:\Lambda \rightarrow \pi_{xz}(\Lambda)$ is one-to-one;  
   $\Lambda^F_1$ consists of points either in the image of a single cusp point of $\Lambda_{cusp}$ or at a crossing of two smooth sheets of $\Lambda$;   
 $\Lambda^F_2$ consists of {\bf cusp-sheet intersection points}, where a cusp edge intersects a single smooth sheet of $\Lambda$ transversally, and {\bf triple points} at the intersection of three smooth sheets.  See Figure \ref{fig:FrontSing}. (Swallow tail points are not present due to the assumption of mild front singularities.)

\begin{figure}
\labellist
\small
\pinlabel $x_1$ [l] at 38 0
\pinlabel $x_2$ [l] at 20 28
\pinlabel $z$ [b] at 0 38
\endlabellist
\centerline{\includegraphics[scale=.6]{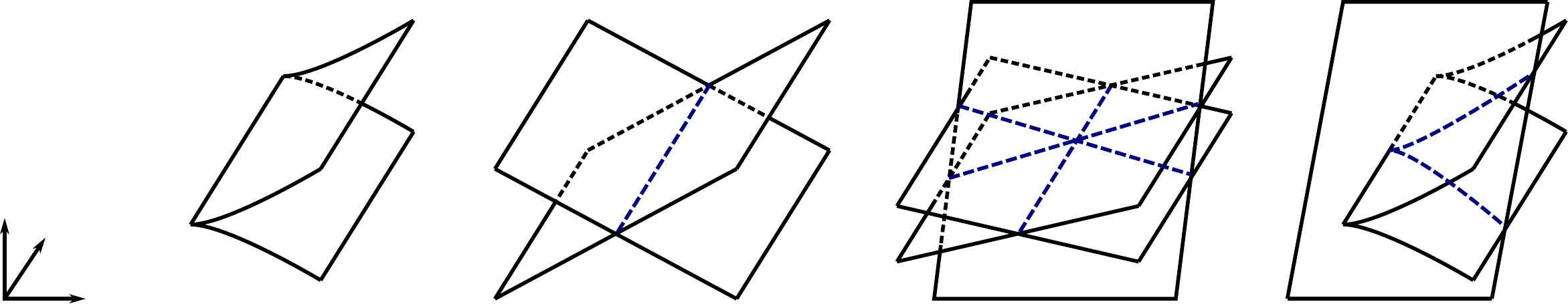}}

\caption{ Local appearance  near $\Lambda^F_1$ and $\Lambda^F_2$ of a Legendrian surface with mild front singularities. }
\label{fig:FrontSing}
\end{figure}

A similar stratification 
\begin{equation} \label{eq:baseStrat}
\pi_x(\Lambda)= \Lambda^B_0 \sqcup \Lambda^B_1 \sqcup \Lambda^B_2
\end{equation}
 exists for the base projection of $\Lambda$.  In detail, $\Lambda^B_k$ consists of points $x \in \pi_x(\Lambda)$ such that the pre-image of $x$ in $\pi_{xz}(\Lambda)$ consists of points $(x,z_1), \ldots, (x,z_l)$ with $(x,z_i)\in \Lambda^F_{k_i}$ with $k_1+\cdots +k_l = k$.   Points in $\Lambda^B_1$ belong to base projection of a single crossing or cusp arc;  points in $\Lambda^B_2$ belong either to the projection of $\Lambda^F_2$ or to the transverse intersection of the projection of a pair of (crossing or cusp) arcs in $\Lambda^F_1$.

\begin{remark}  A similar decomposition $\pi_{xz}(\Lambda) = \Lambda^F_0 \sqcup \Lambda^F_1 \sqcup \cdots \sqcup \Lambda^F_n$ exists when $\Lambda$ is $n$-dimensional with mild front singularities.  Then, $\Lambda^F_k$ consists of points whose preimage in $\Lambda$ consists of $p$ points in $\Lambda \setminus \Lambda_{\mathit{cusp}}$ and $q$ points in $\Lambda_{\mathit{cusp}}$ where $p+2q-1 = k$.
\end{remark}

\begin{definition}  
\label{def:compatible}
For a Legendrian surface $\Lambda \subset J^1M$ with $\dim(\Lambda)=2$ having mild front singularities, we say that a simplicial decomposition $\mathcal{E} = \{e_I \}$ of $M$ is {\bf compatible with $\Lambda$} if for all $0 \leq k \leq 2$, $\Lambda^B_k$  
is contained in the $(2-k)$-skeleton.   
\end{definition}

\begin{remark}
\begin{enumerate}
\item Any Legendrian surface has many compatible simplicial decompositions.
\item For the constructions of this article, the form of the triangulation outside of $\pi_x(\Lambda)$ is not particular important, and working instead with a simplicial decomposition of $\pi_x(\Lambda)$ (having only a finite number of simplices) would suffice.  However, to simplify exposition we consider decompositions of all of $M$.
\end{enumerate}
\end{remark}

We observe here some properties of $\Lambda$ above individual cells $e_I \in \mathcal{E}$ in the case that $\mathcal{E}$ is compatible with $\Lambda$ and that ${\dim(\Lambda) = 2}.$ 
\begin{property}  \label{ob:cusp}
For any $e_I \in \mathcal{E}$, the inverse image $\pi_x^{-1}(e_I) \subset \Lambda$ is a union of components that are projected homeomorphically to $e_I$ by $\pi_x$, and each component of $\pi_x^{-1}(e_I)$ is either contained entirely in $\Lambda \setminus \Lambda_{\mathit{cusp}}$ or contained entirely in $\Lambda_{\mathit{cusp}}$.
\end{property}

For each $e_I \in \mathcal{E}$, let $\Lambda(e_I)$ denote the set of {\bf sheets of $\Lambda$ above $e_I$} which by definition are those components of $\pi_{x}^{-1}(e_I) \subset \Lambda$ not belonging to $\Lambda_{\mathit{cusp}}$.
Note that for any $S_i \in \Lambda(e_I)$, the front projection $\pi_{xz}(S_i) \subset e_I \times \R$ coincides with the graph of a local defining function $z_i:e_I \rightarrow \R$; we will sometimes write $z_i = z(S_i)$.  Components of $\pi_x^{-1}(e_I)$ belonging to $\Lambda_{\mathit{cusp}}$ are called {\bf cusp components above $e_I$}.

\begin{property}  \label{ob:crossing}
For any $S_i, S_j \in \Lambda(e_I)$, their front projections are either disjoint with one entirely above the other, or they coincide completely.  That is, their defining functions satisfy either $z(S_i) > z(S_j)$, $z(S_i) > z(S_j)$, or $z(S_i) = z(S_j)$. 
\end{property}

\begin{property} \label{ob:closure}
 When $e_J \subset \overline{e_I}$, every sheet in $\Lambda(e_J)$ belongs to the closure (in $\Lambda$, not $\pi_{xz}(\Lambda)$) of a unique sheet in $\Lambda(e_I)$.  The remaining sheets of $\Lambda(e_I)$ meet in pairs at cusp components above $e_{J}$.  
\end{property}

When $\dim \Lambda \neq 2$, we take the above properties as the definition of compatible simplicial decompositions.  That is, we define a simplicial decomposition $\mathcal{E}$ of $M$ to be {\bf compatible with $\Lambda$} if Properties \ref{ob:cusp}-\ref{ob:closure} hold.  We do not address the question of whether compatible simplicial decompositions always exist.  Higher dimensional Legendrians will only be considered in the remainder of this section and in Section \ref{sec:CHD}.

\subsection{A simplicial Legendrian DGA}
\label{sec:SimplicialReformulation}

We will now associate a DGA $(\mathcal{A}, \partial)$ to a triple $(\Lambda, \mu, \mathcal{E})$ consisting of a Legendrian $\Lambda \subset J^1M$ with mild front singularities equipped with a $\Z$-valued Maslov potential, $\mu: \Lambda \setminus \Lambda_{\mathit{cusp}} \rightarrow \Z$, and a compatible simplicial decomposition, $\mathcal{E}$.  For each $e_I$ choose an ordering of the sheets in $\Lambda(e_I)$ as $S_1, \ldots, S_{N_I}$ so that the $z$-coordinates satisfy $z(S_{i-1}) \geq z(S_{i})$ for all $1 < i \leq N_I$.  Let $\mathcal{A}$ be the free unital associative algebra  
  over $\Z$ with generators $m_{i,j}^I$ for all $I, i,$ and $j$ such that $S_i,S_j \in \Lambda(e_I)$ satisfy $z(S_i) > z(S_j)$.  The grading 
 is  given
by 
$$ |m_{i,j}^I| = \mu(S_i) - \mu(S_j) + \|e_I\|.$$

In order to define the differential $\partial :\mathcal{A} \rightarrow \mathcal{A}$ we introduce some matrix notations.  For each $e_I$, let 
\[
M(e_I)= (m_{i,j}^I)
\]
 be the 
strictly upper triangular $N_I \times N_I$ matrix with entries $m_{i,j}^I$ when $m_{i,j}^I$ exists, and $0$ if $m_{i,j}^I$ does not exist, eg. the $(k,l)$-entry will be $0$ if $S_k, S_{l} \in \Lambda(e_I)$ cross one another above $e_I$ in $\pi_{xz}(\Lambda)$.  
  More generally, for each $e_J$ with $e_J \subset \overline{e_I}$ form an $N_I \times N_I$ matrix $M_I(e_J)$ by placing generators associated to $e_J$ in positions specified by the ordering of sheets {\it above $e_I$}.  That is, we place $m_{i,j}^J$ into the $(i',j')$-position where $S^J_{i} \subset \overline{S^I_{i'}}$ and $S^J_j \subset \overline{S^I_{j'}}$.  In addition, if $e_J$ is a $0$-cell, for each  pair of sheets $S^I_k$ and $S^I_l$ with $k<l$ that meet one another at a cusp over $e_J$ we place a $1$ at the $(k,l)$-entry of $M_I(e_J)$.  All other entries of the $M_I(e_J)$ are $0$.  Note that $M(e_I) = M_I(e_I)$.  Put another way, $M_I(e_J)$ is obtained from $M(e_J)$ by first taking the direct sum (block sum) with some number of $2\times 2$-matrices of the form $\begin{pmatrix} 0 & 1 \\ 0 & 0 \end{pmatrix}$ (resp.,  $\begin{pmatrix} 0 & 0 \\ 0 & 0 \end{pmatrix}$) if $e_J$ is a 0-cell (resp., a $(>0)$-cell), one for each pair of sheets of $\Lambda(e_I)$ that meet at a cusp edge above $e_J$, and then conjugating by a permutation matrix so that the orderings of rows and columns matches the ordering of sheets in $\Lambda(e_I)$.  In addition, for each $e_I$ define a diagonal matrix of signs, 
\[
\Theta_I = \mathit{diag}((-1)^{\mu(S_1)}, \ldots, (-1)^{\mu(S_{N_I})})
\]
 where $S_1, \ldots, S_{N_I} \in \Lambda(e_I)$. 

We define the differential of generators, $\partial m_{i,j}^I$, so that when $\partial$ is applied entry-by-entry the following matrix equation holds where $I = (i_0, \ldots, i_m)$,
\begin{equation} \label{eq:Dmatrix}
\partial  M(e_I)= \Theta_I \cdot \left(-\sum_{k = 0}^{m} (-1)^k M_I(e_{i_0 \cdots \widehat{i_k} \cdots \cdots i_m}) +  \sum_{k = 0}^{m} (-1)^k M_I(e_{i_0\cdots i_k}) \cdot M_I(e_{i_k\cdots i_m})\right).
\end{equation}
The equation (\ref{eq:Dmatrix}) can be stated more succinctly in terms of the cobar construction as 
\begin{equation} \label{eq:Dmatrix2}
\partial \circ M_I(e_I) = \Theta_I \cdot (M_I \circ D_s)(e_I)
\end{equation}
where $M_I$ is extended to a (non-grading preserving) unital algebra homomorphism $M_I: \Omega C_*(\overline{e_I}) \rightarrow \mathit{Mat}(n, \mathcal{A})$.  

\begin{theorem}  \label{thm:dg}
\begin{enumerate}
\item  There is a unique degree $-1$ derivation $\partial:\mathcal{A} \rightarrow \mathcal{A}$  (a linear map satisfying $\partial(xy) = \partial(x) y +(-1)^{|x|}x \partial(y)$) such that (\ref{eq:Dmatrix}) holds for all $e_I$.
\item $\partial^2 =0$.
\end{enumerate}
\end{theorem}

\begin{proof}
As each generator $m^I_{i,j}$ appears exactly once as an entry of $M(e_I)$, to prove (1) we just need to argue that whenever the $(i,j)$-entry of $M(e_I)$ is $0$, the $(i,j)$-entry of the RHS of (\ref{eq:Dmatrix}) is also $0$.  (That $\partial$ has degree $-1$ on generators is verified in a straightforward manner.)  Since all matrices that appear are upper-triangular, it suffices to consider the case where $i<j$.  Then, the $(i,j)$-entry of $M(e_I)$ can only be $0$ if $z(S_i) = z(S_j)$ for $S_i,S_j \in \Lambda(e_{I})$.  If this occurs, then for all $i \leq k \leq j$ and all $e_J \subset \overline{e_I}$ the sheets of $\Lambda(e_J)$ contained in the closure of $S_i$, $S_k$, and $S_j$ 
must also coincide in the front projection.  Moreover, none of $S_i, S_k,$ and $S_j$ can meet one another at a cusp above $e_J$, and it follows that the $(i,k)$ and $(k,j)$ entries of the matrices $M_I(e_J)$ also vanish.  Consequently, the $(i,j)$-entries of all of the matrices $M_I(e_{i_0 \cdots \widehat{i_k} \cdots \cdots i_m})$ and $M_I(e_{i_0\cdots i_k}) \cdot M_I(e_{i_k\cdots i_m})$ appearing on the right hand side are $0$.

As a preliminary to proving (2), we extend $\partial$ entry-by-entry to a derivation on $\mathit{Mat}(n, \mathcal{A})$, and note that for any $X,Y \in \mathit{Mat}(n,\mathcal{A})$ with $X =(x_{i,j})$ having entries of homogeneous degree the Liebniz rule translates to 
\begin{equation} \label{eq:XY}
\partial(XY) = \partial(X) \cdot Y + \sigma(X) \cdot \partial(Y) 
\end{equation}
where $\sigma(X) = ( (-1)^{|x_{i,j}|} x_{i,j} )$.  Note in particular that for any $e_J \subset \overline{e_I}$, 
\begin{equation}  \label{eq:sigma}
\sigma(M_I(e_J)) =  (-1)^{\|e_J\|} \Theta_I M_I(e_J) \Theta_I.
\end{equation} 
[Here, it is important to observe that when $\dim e_J = 0$ and the $(k,l)$-entry of $\sigma(M_I(e_J))$ is $1$ due to sheets $S_k, S_l \in \Lambda(e_I)$ meeting one another at a cusp above $e_J$, 
 the $(k,l)$ entry of $(-1)^{\|e_J\|} \Theta_I M_I(e_J) \Theta_I$ is 
\[
(-1)^{\|e_J\|} (-1)^{\mu(S_k)} \cdot 1 \cdot (-1)^{\mu(S_l)} = (-1)^{-1+\mu(S_k) + \mu(S_l)} = 1
\]
since $\mu(S_k)$ and $\mu(S_l)$ have opposite parity.]  In addition, since derivations annihilate scalars, for any scalar matrices $P,Q \in \mathit{Mat}(n, \Z)$ we have
\begin{equation} \label{eq:PYQ}
\partial(P Y Q) = P \partial(Y) Q.  
\end{equation}

\medskip

\noindent{\bf Claim:}  For any $x \in  \Omega C_*(\overline{e_I})$, 
\begin{equation} \label{eq:Claim1}
\partial \circ M_I(x) = \Theta_I \cdot (M_I \circ D_s)(x).  
\end{equation}

\medskip

Using the claim and (\ref{eq:PYQ}) we can compute 
\begin{align*}
\partial^2 \circ M(e_I) & = \partial \left( \Theta_I \cdot \left( M_I \circ D_s( e_I) \right) \right) = \Theta_I \cdot \left( (\partial \circ M_I)(D_s(e_I))\right)  = \Theta_I \cdot \left( \Theta_I \cdot (M_I \circ D_s^2)(e_I)\right) = 0.  
\end{align*}
This shows that $\partial^2$ vanishes on the generators of $\mathcal{A}$, and it follows that $\partial^2 =0$.  

We turn to proving the claim.  Since both sides of (\ref{eq:Claim1}) are $\Z$-linear in $x$, it suffices to verify the claim when $x$ is a product of simplices in $C_*(\overline{e_I})$.  This is done by induction on the word length of $x$.  

\medskip

\noindent {\bf Basis step:}   $x =e_J$.   We note from the construction of $M_I$ and $M_J$ that there is a permutation matrix $P_{I,J}$ such that for any $e_{K} \subset \overline{e_J} \subset \overline{e_I}$,
\begin{equation}
M_I(e_K) = P_{I,J}\left( M_J(e_K) \oplus N(e_K) \oplus \cdots \oplus N(e_K) \right)P_{I,J}^{-1} 
\end{equation}
where $N(e_K)$ is $\begin{pmatrix} 0 & 1 \\ 0 & 0 \end{pmatrix}$ (resp.  $\begin{pmatrix} 0 & 0 \\ 0 & 0 \end{pmatrix}$) when $\dim e_K = 0$ (resp. $\dim e_K > 0$) and the number of $N(e_K)$ terms in the block sum is the number of pairs of sheets in $\Lambda(e_I)$ that meet at cusps above $e_J$.  Extending $N$ to a unital algebra homomorphism $N:\Omega C_*(\overline{e_J}) \rightarrow \mathit{Mat}(2, \Z)$ we see that 
\begin{equation} \label{eq:M1}
M_I(y) = P_{I,J}\left( M_J(y) \oplus N(y) \oplus \cdots \oplus N(y) \right)P_{I,J}^{-1},  \quad \forall y \in \Omega C_*(\overline{e_J})
\end{equation}
(since both sides are unital algebra homomorphisms and they agree on a generating set), and
\begin{equation} \label{eq:N1}
N \circ D_s = 0 = \partial \circ N.
\end{equation}
[To verify the first equality, since $D_s$ is a derivation, it is enough to check that  $N \circ D_s (e_K) =0$ 
 for any simplex $e_K \in \Omega C_*(\overline{e_J})$.
 Note that $N$ vanishes on the monomials of length $2$, and the length $1$ terms in $D_se_K$ are $-\partial_s e_K$  (with $\partial_s$ the simplicial boundary operator).  Moreover, $N$ vanishes on all simplices with dimension $>0$, so it suffices to compute that when $\dim e_K =1$
\[
N \circ D_s(e_K) = N ( - \partial_s e_K) = -\begin{pmatrix} 0 & 1 \\ 0 & 0 \end{pmatrix} + \begin{pmatrix} 0 & 1 \\ 0 & 0 \end{pmatrix} = 0. \quad ]
\]
Now, using (\ref{eq:Dmatrix2}), (\ref{eq:PYQ}), (\ref{eq:M1}), and (\ref{eq:N1})  we compute
\begin{align*}
\partial M_I(e_J) & = P_{I,J}\left( \partial M_J(e_J) \oplus \partial N(e_K) \oplus \cdots \oplus \partial N(e_K) \right)P_{I,J}^{-1} \\
 & = P_{I,J}\left( \Theta_J M_J(D_se_J) \oplus 0 \oplus \cdots \oplus 0 \right)P_{I,J}^{-1} \\
 & = P_{I,J} (\Theta_J \oplus \Theta') \cdot \left( M_J(D_se_J) \oplus 0 \oplus \cdots \oplus 0 \right)P_{I,J}^{-1} \\ 
 & = P_{I,J} (\Theta_J \oplus \Theta') P_{I,J}^{-1}\cdot P_{I,J} \left( M_J(D_se_J) \oplus N(D_s e_J) \oplus \cdots \oplus N(D_s e_J) \right)P_{I,J}^{-1} \\
 & = \Theta_I \cdot (M_I \circ D_s)( e_J)
\end{align*}
where $\Theta'$ is the unique diagonal matrix of $\pm1$'s such that $P_{I,J} (\Theta_J \oplus \Theta') P_{I,J}^{-1} = \Theta_I$.  [This equation is satisfied on $\Theta_J$ entries since the Maslov potential has $\mu(S_i) = \mu(S_{i'})$ when $S_i \in \Lambda(e_J)$ and $S_{i'} \in \Lambda(e_I)$ have $S_i \subset \overline{S_{i'}}$.]

\medskip

\noindent {\bf Inductive Step:}   
 Suppose that $x = e_J x'$ and (\ref{eq:Claim1}) holds for $x'$.  Using (\ref{eq:XY}) and (\ref{eq:sigma}) we can compute
\begin{align*}
\partial \circ M_I(e_J x') & = \partial \circ M_I(e_J) \cdot M_I(x') + \sigma(M_I(e_J)) \cdot \partial \circ M_I(x')  \\
 & = \Theta_I \cdot(M_I\circ D_s(e_J)) \cdot M_I(x') + (-1)^{\|e_J\|}\Theta_I M_I(e_J) \Theta_I \cdot \Theta_I \cdot (M_I\circ D_s(x')) \\ 
 & = \Theta_I \cdot M_I\left( D_s(e_J) x' + (-1)^{\|e_J\|} e_J D_s(x') \right) = \Theta_I \cdot M_I \circ D_s( e_J x'). 
\end{align*}

\end{proof}

As we have seen, $\partial^2=0$ follows from the combinatorics of front projections and  
 that $D_s^2=0$ in the cobar complex.  What is not obvious is whether or not the homotopy type of $(\mathcal{A}, \partial)$ depends on the simplicial decomposition and is invariant under Legendrian isotopy.
When $\dim(\Lambda) = 2,$ however, with $\Z/2$-coefficients such invariance follows from the below theorem.

\begin{theorem} \label{thm:simplicial}
Suppose $\dim(\Lambda) = 2$.  Then, the simplicial DGA $(\mathcal{A}, \partial)$ with coefficients in $\Z/2$ is stable tame isomorphic to the Legendrian contact homology DGA of $\Lambda$ from \cite{Ch, EES07}.
\end{theorem}

\begin{proof}
This is a re-interpretation of \cite[Theorem 1.1]{RuSu1} when all cells are simplices.  
The simplicial DGA agrees precisely with the cellular DGA from \cite{RuSu1} when 
the additional choices needed to define the cellular DGA  are made as follows:   Each $1$-cell $e_{i_0i_1}$ with $i_0 < i_1$ is oriented from $e_{i_1}$ to $e_{i_0}$.  Each $2$-cell $e_{i_0i_1i_2}$ with $i_0< i_1 <i_2$ is assigned the initial vertex $e_{i_2}$ and the terminal vertex $e_{i_0}$.  
\end{proof}

We conjecture that in $1$-jet spaces of any dimension the simplicial DGA is stable tame isomorphic to the Legendrian contact homology DGA for Legendrians   having mild front singularities.

\begin{remark} When $\dim \Lambda =2$, the construction of the cellular DGA in \cite{RuSu1} uses polygonal decompositions with $2$-cells having any number of edges, rather than just triangles.  In the $2$-dimensional case, with little effort, the later constructions of this article can be generalized to the case of polygonal decompositions allowing $2$-cells with any number of edges.
\end{remark}

\section{Combinatorial constructible sheaves}
\label{sec:ConstructibleSheaves}

Throughout this section, we assume that $\dim(\Lambda) = 2.$  
In Section \ref{sec:handle}, we introduce a specific stratification $\mathcal{S}$ of $J^0M$ arising from a choice of compatible simplicial decomposition for $\Lambda$.   
In Section \ref{sec:CategoryDefinition}, after a brief review of constructible sheaves and their singular support following \cite{STZ},    we associate a category of combinatorial sheaves to $\Lambda$.  This is a $2$-dimensional analog of the combinatorial sheaf category for $1$-dimensional Legendrians from \cite{STZ};  objects are functors that assign chain complexes to the strata of $\mathcal{S}$, subject to local restrictions dictated by $\Lambda$.    
 Finally, in Section \ref{sec:CategoryEquivalence} we prove that this category of combinatorial  sheaves is quasi-equivalent to the category of constructible sheaves with singular support determined by $\Lambda$.

\subsection{Stratifications and polygonal decompositions} 

\label{sec:handle}
Let $\Lambda$ be a Legendrian surface in $J^1M$ with mild front singularities together with a simplicial decomposition $\mathcal{E}$ of $M$ compatible as in Definition \ref{def:compatible}.  Moreover, in this section and Section \ref{sec:SheafConstruction} we add the additional requirement that the closed star of any simplex $e_I$, i.e. the union of all closed simplices that contain $e_I$ as a face, is a closed disk.  This holds after possibly applying  barycentric subdivision.   
 We extend the stratification of $\pi_{xz}(\Lambda) \subset M \times \R$ from (\ref{eq:xzStrat}) to one of  $M \times \R$
\begin{equation}   \label{eq:xzstrat}
M \times \R = \Lambda^F_{-1} \sqcup \Lambda^F_0 \sqcup \Lambda^F_1 \sqcup \Lambda^F_2
\end{equation}
by setting $\Lambda^F_{-1} = (M \times \R) \setminus \pi_{xz}(\Lambda)$.  We refer to this decomposition as the {\bf $\Lambda^F$-stratification} of $M \times \R$.

\begin{figure}
\labellist
\small
\endlabellist
\centerline{\includegraphics[scale=1]{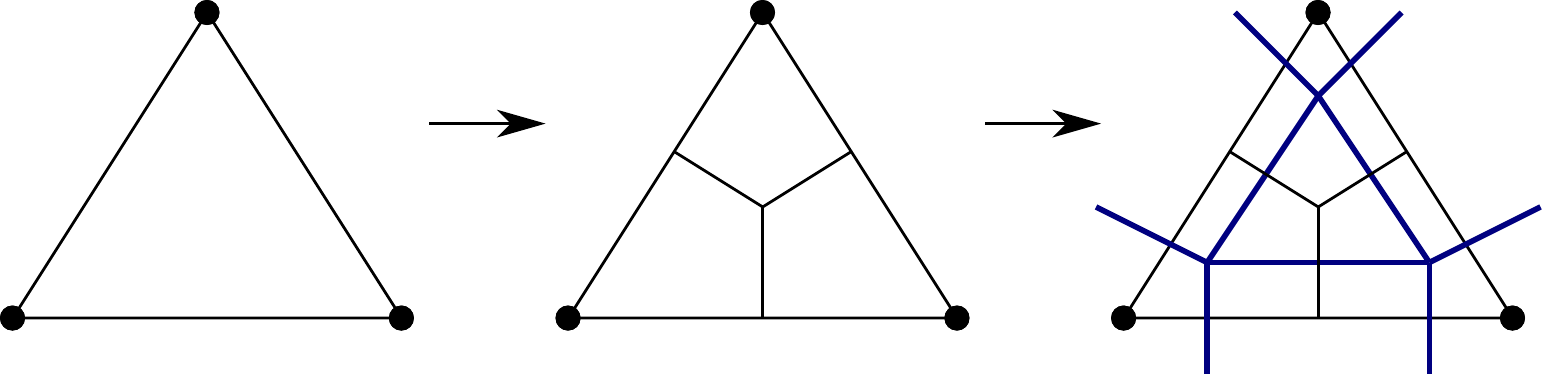}}

\quad

\caption{From left to right, a $2$-simplex in $\mathcal{E}$, its pair subdivision, and the associated handle decomposition $H$ (with edges pictured in blue).}
\label{fig:PairSub}
\end{figure}

Given $\mathcal{E}$, we construct an {\bf associated handle decomposition},  $H$. Viewed as a polygonal decomposition, $H$ is formed via the following two step process pictured in Figure \ref{fig:PairSub}:
\begin{enumerate}
\item First, take the {\bf pair-subdivision} of $\mathcal{E}$ by subdividing each $2$-simplex, $e_I \in \mathcal{E}$, into $3$ squares having vertices at the barycenters of the faces of $e_I$.  Note that the cells of the pair-subdivision are in bijection with ordered pairs to $\bfe:=(e_I, e_J)$ of simplices in $\mathcal{E}$ with $e_I \leq e_J$. Explicitly, $\bfe$ is the convex hull of barycenters of simplices $e_K$ with $e_I \leq e_K \leq e_J$.
\item Take $H$ to be a polygonal decomposition of $M$ dual to the pair-subdivision of $\mathcal{E}$.  The cells of $H$ can be taken to be smooth.
\end{enumerate}
  If $e_I \le e_J,$ then denote by $h(\bfe)$ the cell of $H$ dual to $\bfe:=(e_I, e_J)$.    
Note that $\dim h(\bfe) = \dim(\Lambda) - (\dim e_J -\dim e_I).$ See Figure \ref{fig:HLabel}.  
When viewing $H$ as a handle decomposition, 
index $k$ handles are in bijection with the $k$-simplices of $\mathcal{E}$; the handle associated to a simplex $e_I$  is the closure of $h(e_I,e_I)$.

\begin{figure}
\labellist
\small
\pinlabel $e_0$ [tr] at 46 52
\pinlabel $e_1$ [tr] at 200 52
\pinlabel $e_2$ [br] at 122 180
\pinlabel $e_{01}$ [t] at 128 54 
\pinlabel $e_{02}$ [br] at 88 118
\pinlabel $e_{12}$ [bl] at 164 122
\pinlabel $e_{012}$  at 128 100
\pinlabel $(e_0,e_0)$ [t] at 330 10
\pinlabel $(e_0,e_{01})$ [t] at 406 10
\pinlabel $(e_1,e_{012})$ [t] at 510 -2
\pinlabel $(e_{02},e_{012})$ [b] at 326 148
\pinlabel $(e_{12},e_{12})$ [b] at 492 162
\pinlabel $(e_{012},e_{012})$ [b] at 460 200
\endlabellist
\centerline{\includegraphics[scale=1]{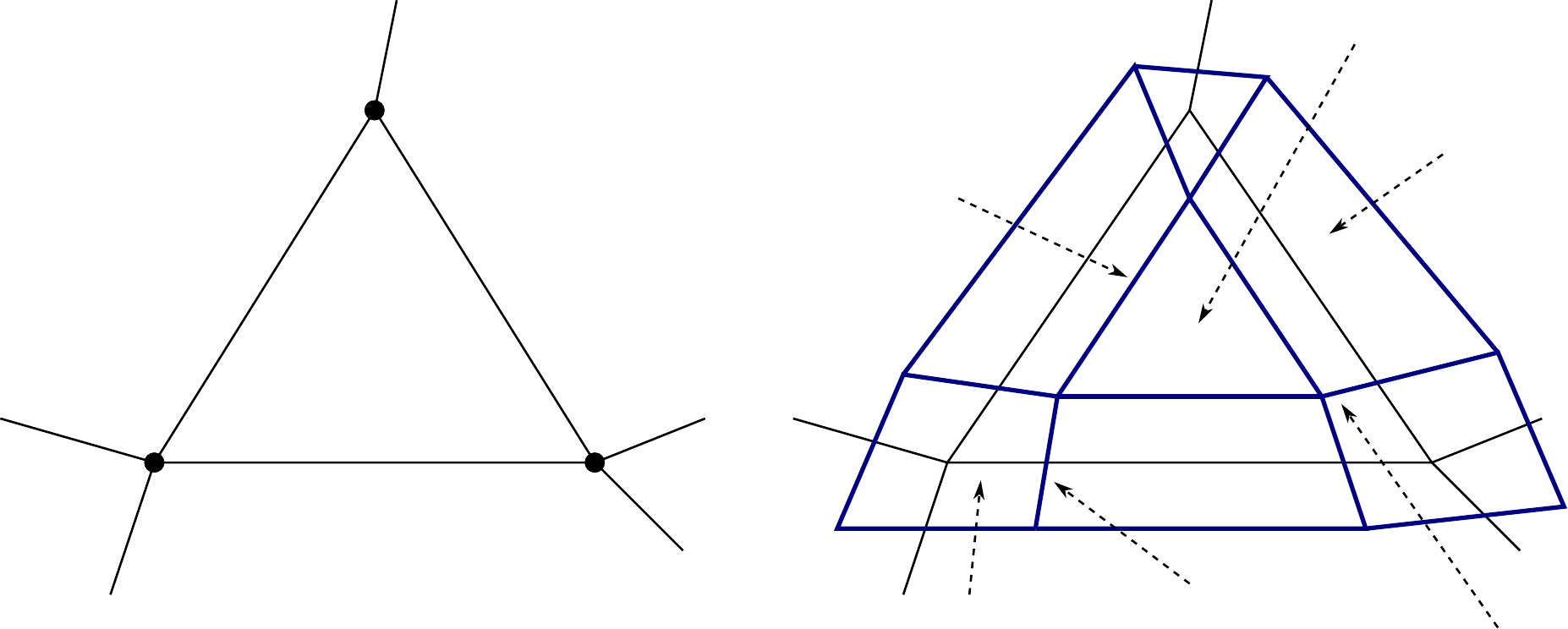}}

\quad

\caption{The labeling of cells $h(\bfe)$ by ordered pairs of simplices $\bfe =(e_I,e_J)$ is indicated (right) near a $2$-simplex $e_{012}$ of $\mathcal{E}$ with vertices $e_0, e_1, e_2$ (left).}
\label{fig:HLabel}
\end{figure}
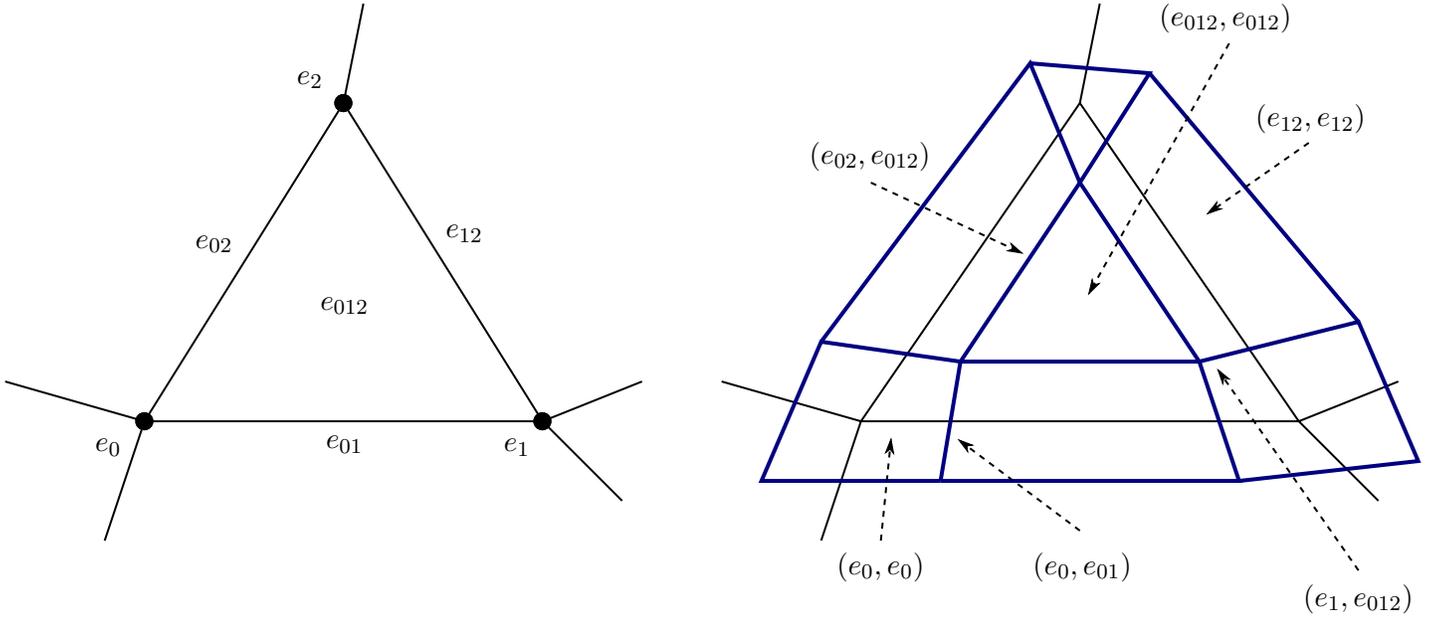

\begin{definition} 
\label{def:AFS}
Given $\mathcal{E}$ the {\bf associated front stratification} $\mathcal{S}$ is the stratification of $J^0M = M \times \R$ whose strata are the connected components of the intersections $\Lambda^F_{k} \cap (h(\mathbf{e})\times \R)$,  for $-1 \leq k \leq 2$. See (\ref{eq:xzstrat}). 
\end{definition}    
In other words, $\mathcal{S}$ is obtained from lifting the cells of $H$ to products $h(\mathbf{e}) \times \R$ and then subdividing to incorporate the natural stratification of $\pi_{xz}(\Lambda)$.

The stratification $\mathcal{S}$ has an {\bf associated poset category}, $P(\mathcal{S})$, whose objects are the strata of $\mathcal{S}.$ The partial ordering has $s_1 \le s_2$ when  $s_1 \subset \overline{s_2}$, and  in $P(\mathcal{S})$ there is a unique morphism from $s_1$ to $s_2$ if and only if $s_1  \le {s_2}$.  Non-identity morphisms in $P(\mathcal{S})$ are sometimes referred to as {\bf generization maps}.  

\begin{definition} \label{def:downward} We say that a generization map of strata, $s_1 \rightarrow s_2$, is {\bf downward} if there exists $x_0 \in M$ such that $s_1\cap(\{x_0\} \times \R)$ is a single point that is bordered below by an interval from $s_2 \cap(\{x_0\} \times \R)$.  
\end{definition}

\subsection{The dg-categories $\mathbf{Fun}^\bullet_\Lambda(\mathcal{S}, \mathbb{K})$ and $\ShLambda$}
\label{sec:CategoryDefinition}

We begin by recalling two quasi-equivalent dg-categories $\ShS$ and  $\mathbf{Fun}^\bullet(\mathcal{S},\mathbb{K})$ associated to the stratification $\mathcal{S}$ of $M \times \R$ as in \cite[Sections 3.1 and  3.3]{STZ}.  Both categories are instances of a dg-derived category construction that associates to an abelian category $\mathcal{A}$ a dg-category $\mathcal{D}_{dg}(\mathcal{A})$ having cochain complexes in $\mathcal{A}$ as objects and so that the cohomology category (where morphism spaces are $H^0$ of the morphism spaces of $\mathcal{D}_{dg}(\mathcal{A})$) coincides with the usual derived category $\mathcal{D}(\mathcal{A})$.  Assuming $\mathcal{A}$ has enough injectives, any left exact functor $F:\mathcal{A} \rightarrow \mathcal{B}$ can be right derived to a dg-quasi-functor $RF: \mathcal{D}_{dg}(\mathcal{A}) \rightarrow \mathcal{D}_{dg}(\mathcal{B})$ by applying $F$ naively to the dg-category of chain complexes $\mathit{Ch}_{dg}(\mathcal{I})$ where $\mathcal{I} \subset \mathcal{A}$ is the subcategory of injective objects.     
  See \cite{Drinfeld} as well as \cite[Section 2]{Nadler} for a concise discussion.  In this section, we will follow the convention of \cite{STZ}: All left exact functors, $F$, are right derived when applied to $\mathcal{D}_{dg}(\mathcal{A})$, and are denoted with the same notation $F$ rather than by $RF$.

\begin{definition} Let $\mathbb{K}$ be a field.  Let $\ShS$ be the dg-derived category of  constructible (with respect to $\mathcal{S}$ as in Definition \ref{def:AFS}) sheaves of $\mathbb{K}$-modules on  $M\times\R$ with bounded cohomology.
\end{definition}
  We will often refer to objects in $\ShS$ simply as sheaves even though they are in fact complexes of sheaves on $M\times \R$.

\begin{definition} Let  $\mathbf{Fun}^\bullet(\mathcal{S},\mathbb{K})$ be the dg-derived category associated to the abelian category of functors from $P(\mathcal{S})$ to $\mathbb{K}$-mod.
\end{definition}

  We will view objects of $\mathbf{Fun}^\bullet(\mathcal{S},\mathbb{K})$ as functors from $P(\mathcal{S})$ to $\mathit{Ch}(\mbox{$\mathbb{K}$-mod})$, the  category of cochain
 complexes of $\mathbb{K}$-modules with cochain maps.
Given $F \in \mathbf{Fun}^\bullet(\mathcal{S}, \mathbb{K})$, 
we write $F(s_1 \rightarrow s_2):F(s_1) \rightarrow F(s_2)$ for the cochain map which is the image of the generization map $s_1 \rightarrow s_2$.

For a stratum $s \in \mathcal{S},$ let the {\bf star} of $s$ be the union of strata in $\mathcal{S}$ that contain $s$ in their closure. We denote this by $s^*.$
\begin{proposition}  The functor
$$\Gamma_{\mathcal{S}}:\ShS \rightarrow \mathbf{Fun}^\bullet(\mathcal{S}, \mathbb{K}), \quad \mathcal{F} \mapsto [s \mapsto \Gamma(s^*; \mathcal{F})] 
$$
is a quasi-equivalence.
\end{proposition}
\begin{proof}
This is a special case of Proposition 3.9 of \cite{STZ} since $\mathcal{S}$ is a {\it regular cell-complex} in the terminology of \cite{STZ}, i.e. every stratum is contractible and the star of each stratum is contractible.  [This is due to the extra assumption on $\mathcal{E}$ that the star of each simplex is a disk.]
\end{proof}

\subsubsection{Singular support and the category $\ShLambda$}

A subcategory of $\ShLambda \subset \ShS$ is defined by imposing singular support conditions coming from the Legendrian $\Lambda$.  We review the relevant definitions following \cite{STZ}.

In the stratified Morse theory  \cite{Massey} for a smooth function $f:M \times \R \rightarrow \R,$ a non-degenerate critical point $q$ of $f$ (in the stratified sense) is a non-degenerate critical point of $f|_{s_0}$ (in the usual Morse sense) for some stratum $s_0 \in \mathcal{S},$
such that for all strata
$s_0 \le s_\alpha$ there exists 
$v \in T_q(M \times \R) \cap \overline{T s_\alpha}$ with $df_q(v) \ne 0.$
Fix a metric $g$ on $J^0M$ and let $B_\delta(x) \subset J^0M$ denote the open ball of radius $\delta$ centered at $x.$
For $\mathcal{F} \in \ShS$ and sufficiently small $0 < \epsilon \ll \delta$  we can define the {\bf{Morse group}} $Mo_{q,f}(\mathcal{F}),$ independent of $\epsilon, \delta,$ to be the cone of 
\begin{equation}  \label{eq:MorseGroup}
\Gamma(f^{-1}(-\infty, f(q)+\epsilon) \cap B_\delta(q) ; \mathcal{F}) \rightarrow \Gamma(f^{-1}(-\infty, f(q)-\epsilon) \cap B_\delta(q); \mathcal{F}).
\end{equation}

A cotangent vector $\xi \in T^*_q(M\times \R)$ is called a {\bf characteristic vector} if there exists a stratified Morse function $f$ with $df_q = \xi$ such that $Mo_{q,f}(\mathcal{F}) \neq 0.$  In fact, up to grading shift,  the Morse group depends only on $\xi$, so that this condition is independent of the choice of stratified Morse function $f$, cf. \cite[Section 3.1.2]{STZ}.  
The {\bf singular support}, $SS(\FF) \subset T^*(M\times \R)$, is the closure of the set of characteristic vectors.
Note that $SS(\FF)$ depends only on $\FF$ and not on the choice of stratification $\mathcal{S}$ with respect to which  $\FF$ is constructible, while the collection of characteristic vectors can depend on $\mathcal{S}$.  [Eg., when $\xi$ annihilates $\overline{Ts_\alpha}$ for some $s_0 \leq s_\alpha$ there are no stratified Morse functions with $df_q = \xi$.]  As stated in \cite[Item (1), p.1050]{STZ}, any $\mathcal{F} \in \ShS$ has its singular support contained in the union of conormals to the strata of $\mathcal{S}$, i.e.
\[
SS(\mathcal{F}) \subset \bigcup_{s_0 \in \mathcal{S}} (Ts_0)^\perp,  \quad \mbox{where } (Ts_0)^\perp = \{(q,\xi) \in T^*(M\times \R) \,|\, q \in s_0,  \, \xi|_{T_qs_0} = 0\}.
\]

The following Microlocal Morse Lemma \cite[Section 5.4]{KS} will be used later to compute Morse groups and hence singular support.
\begin{lemma}
\label{lem:Microlocal}
Suppose $f \in C^1(M\times \R)$ is proper on the support of $\FF,$ and for all $q \in f^{-1}([a,b)),$
 $df_q \notin SS(\FF).$  Then the following restriction map is a quasi-isomorphism
 \[
 \Gamma(f^{-1}(-\infty, b); \FF) \rightarrow  \Gamma(f^{-1}(-\infty, a); \FF).
 \]
\end{lemma}

Consider the contact embedding
  of $J^1M$ into $ST^*(M \times \R) =  T^*(M \times \R)/\R_{>0}$  given (locally in $M$) by $(x,y,z) \mapsto [(x,z; y, -1)].$  Let $\Lambda^-$ denote the image of the Legendrian $\Lambda \subset J^1M$ under this embedding.  It consists of the downward conormal directions to $\pi_{xz}(\Lambda)$ in $J^0M$.   

\begin{definition} Let $\ShLambda$ denote the full subcategory of $\ShS$ whose sheaves have singular support in the union of the Lagrangian
 cylinder 
 $\R_{>0} \Lambda^- \subset T^*(M\times \R)$ over $\Lambda^-$ and the zero-section of $T^*(M\times \R).$ 
\end{definition} 
 
In \cite[Theorem 4.1]{STZ},
the results of \cite{GKS} are applied to establish that up to
 quasi-equivalence $\ShLambda$ is a Legendrian isotopy invariant of $\Lambda$.

\subsubsection{The combinatorial sheaf category $\mathbf{Fun}^\bullet_\Lambda(\mathcal{S}, \mathbb{K})$}

In \cite[Definition 3.11]{STZ}, the authors describe a sub-category of $\mathbf{Fun}^\bullet(\mathcal{S}, \mathbb{K})$ when $\mathcal{S}$ is a stratification of $\R^2$ induced from the front of a one-dimensional Legendrian, with possibly extra 1-dimensional strata avoiding crossing points.  
We build off of this. 
Recall the $\Lambda^F$-stratification from  (\ref{eq:xzstrat}).

\begin{definition}
\label{def:fun} Given a stratification $\mathcal{S}$ as in Definition \ref{def:AFS},
let $\mathbf{Fun}_{\Lambda}^\bullet(\mathcal{S}, \mathbb{K})$ denote the full sub-category of $\mathbf{Fun}^\bullet(\mathcal{S}, \mathbb{K})$ consisting of those $F$ which satisfy the below three conditions.
\begin{enumerate}
\item  All downward generization maps are quasi-isomorphisms.
\item  All generization maps between strata of $\mathcal{S}$ belonging to the same $\Lambda^F$-stratum,  $\Lambda^F_k \subset M \times \R$,  are quasi-isomorphisms.
\item  Consider a labeled $1$-stratum $O \in \Lambda_1,$ $2$-strata $NE, NW \in \Lambda_{0} $ and $3$-stratum $N  \in \Lambda_{-1}$ as in Figure \ref{fig:Cross}.
  Then the total complex of the following diagram is acyclic.
\begin{equation} \label{eq:NW}
\xymatrix{ 
 & F(N) & \\
F(NW) \ar[ru] & & F(NE) \ar[lu] \\
& F(O) \ar[ru]  \ar[lu] &
} 
\end{equation}
\end{enumerate}
\end{definition}

\begin{figure}
\labellist
\scriptsize
\pinlabel $O$ [tl] at 146 100
\pinlabel $NW$  at 102 118
\pinlabel $NE$ [tr] at 192 126
\pinlabel $N$  at 148 152
\scriptsize
\pinlabel $x_2$ [l] at 32 16
\pinlabel $x_{1}$  [t] at 0 -2
\pinlabel $z$ [b] at 8 42
\endlabellist
\centerline{\includegraphics[scale=.8]{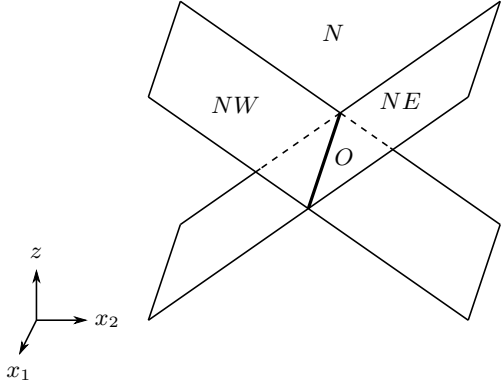}}

\caption{The four strata $O, NW, NE, N$ are contained in the closure of the upper region (with respect to the $z$-direction) at a $1$-stratum, $O$, that is part of the crossing locus of $\pi_{xz}(\Lambda)$.
}  
\label{fig:Cross}
\end{figure}

We will sometimes refer to an object $F \in \mathbf{Fun}^\bullet_\Lambda(\mathcal{S}, \mathbb{K})$ as a {\bf combinatorial sheaf}.

\subsection{Equivalence of categories}
\label{sec:CategoryEquivalence}

In this subsection, we prove the 2-dimensional analog of \cite[Theorem 3.12]{STZ}.

\begin{theorem}  
\label{thm:FunSh}
Let $\mathcal{F} \in \ShS$.  Then, $\Gamma_\mathcal{S}(\mathcal{F}) \in \mathbf{Fun}_{\Lambda}^\bullet(\mathcal{S},k)$ if and only if $\mathcal{F} \in \ShLambda$.
\end{theorem}

\begin{corollary} \label{cor:Equiv}
The dg-functor $\Gamma_{\mathcal{S}}$ restricts to a quasi-equivalence from $\ShLambda$ to $\mathbf{Fun}^\bullet_\Lambda(\mathcal{S}, \mathbb{K})$.
\end{corollary}

\begin{proof}
The reverse implication of the theorem implies that $\Gamma_\mathcal{S}$ restricts to a dg-functor from $\ShLambda$ to $\mathbf{Fun}^\bullet_\Lambda(\mathcal{S}, \mathbb{K})$ (that remains fully faithful on cohomology), and the forward implication shows that this restriction is essentially surjective. 
\end{proof}

Theorem \ref{thm:FunSh} is established in the remainder of this section.

\subsubsection{Possible strata for $2$-dimensional Legendrians}
\label{sec:list}
In preparation for the proof of Theorem \ref{thm:FunSh} we now give a detailed list of possible strata of $\mathcal{S}$.
 Recall that in the $\Lambda^F$-stratification the codimension $1$ singularities belonging to $\Lambda^F_{1}$ consist of crossing arcs and cusp edges; the codimension $2$ singularities in $\Lambda^F_2$ are either triple points or cusp-sheet intersections.   
See Figure \ref{fig:FrontSing}.     
Note from the construction of $H$ that strata
 above $0$-cells of $H$ only belong to $\Lambda^F_{-1}$ and $\Lambda^F_{0}$.  Above $1$-cells of $H$ there is at most one point in $\Lambda^F_{1}$.
 Codimension $2$ singularities only occur above $2$-cells of the form $h(\bfe)$ with $\bfe= (e_{i_0},e_{i_0})$, i.e. above the interiors of $0$-handles when $H$ is viewed as a handle decomposition.

\begin{remark}  \label{rem:orthog}
All of the $0$-cells of $H= \{h(\bfe)\}$ are $4$-valent, eg. a vertex $h(\bfe)$ with $\bfe = (e_{i_0}, e_{i_0i_1i_2})$ sitting near the $e_{i_0}$ corner of the $2$-simplex $e_{i_0i_1i_2}$ is an end point for the $4$ edges $h(e_{i_0}, e_{i_0i_1})$, $h(e_{i_0},e_{i_0i_2})$, $h(e_{i_0i_1}, e_{i_0i_1i_2})$, and $h(e_{i_0i_2}, e_{i_0i_1i_2})$.  To ease considerations, 
 we assume that at vertices, adjacent edges meet one another orthogonally (in local coordinates) and opposite edges piece together to form smooth paths through the vertex.  
\end{remark}
 
Organized by dimension, the strata of $\mathcal{S}$ are of the following types.

\medskip

\noindent{\bf $3$-strata:}  The 
 3-strata of $\mathcal{S}$ lie in the complement of the front $\pi_{xz}(\Lambda)$ above $2$-cells of $H$.

\medskip

\noindent{\bf $2$-strata:}  When the 2-strata are contained in $\Lambda^F_0$ we call them {\bf Legendrian 2-strata}.  
  There are also 2-dimensional {\bf vertical strata} contained in $\Lambda^F_{-1}$ that project to $1$-cells of $H$.  (We call them ``vertical'' since they are parallel with the $z$-axis.)  In the following, we sometimes refer to subsets of $M\times \R$ made up of  closures of Legendrian 2-strata (resp. vertical 2-strata) that meet one another smoothly as {\bf Legendrian front planes} (resp. {\bf vertical planes}).

\medskip

\noindent {\bf $1$-strata:}
We itemize the four types of 1-strata 
 for easier future reference.
\begin{itemize}
\item[($F^2$)] A crossing arc in $\Lambda^F_1$ is the intersection of the closure of four Legendrian 2-strata, which 
 in pairs form two transversely intersecting Legendrian front planes.
\item[($V^2$)] A vertical 1-stratum in $\Lambda^F_{-1}$ is the intersection of the closure of four vertical 2-strata above a vertex of $H$.
  This can also be viewed as two transversely intersecting vertical planes.  (As in Remark \ref{rem:orthog}, closures of vertical strata above opposite edges at a vertex of $H$  
 piece together into smooth planes.)
 \item[($FV$)] A non-vertical 1-stratum in $\Lambda^F_0$ results from the intersection of a Legendrian front plane and a vertical plane.
\item[($Cu$)] 
A cusp arc in $\Lambda^F_1$.  These occur at a cuspidal intersection of two Legendrian front planes.
\end{itemize}

\medskip

{\bf $0$-strata:} There are five types of 0-strata. 
\begin{itemize}
\item[($F^3$)] This is a triple point intersection in $\Lambda^F_2$ of 3 Legendrian front planes.
\item[($F^2V$)] This is a point in $\Lambda^F_1$ at the intersection of two Legendrian front planes and a vertical plane.
\item[($FV^2$)] This is a  point in $\Lambda^F_0$ at the intersection of two vertical planes and a Legendrian front plane.
\item[($VCu$)] This is a  point in $\Lambda^F_1$ where a vertical plane intersects the cusp edge of two cusping Legendrian front planes. \item[($FCu$)] This is a point in $\Lambda^F_2$ where a Legendrian front plane intersects the cusp edge of two cusping Legendrian front planes. 
\end{itemize}

\subsubsection{Computations with constructible sheaves and proof of the converse in Theorem \ref{thm:FunSh}} 

We will routinely make use of the following point which uses 
 the Microlocal Morse Lemma and that stars of strata, $s_0^*$, are balls:
\begin{observation}  \label{observe:1} When $\mathcal{F} \in \ShS$, for any point $q$ belonging to a stratum $s_0 \in \mathcal{S}$ and for all sufficiently small balls $B_\delta(q)$ centered at $q,$ the restriction is a quasi-isomorphism, 
\[
\Gamma(s_0^*; \mathcal{F}) \stackrel{\cong}{\rightarrow} \Gamma(B_{\delta}(q); \mathcal{F}).
\]
\end{observation}
Of course, the Microlocal Morse Lemma allows for $B_{\delta}(q)$ to be expanded from being a small ball to a much larger topological ball provided that this can be done so that the direction of expansion is never in $SS(F)$.  In evaluating $\Gamma(U; \mathcal{F})$ for particular $U$, it is often useful to combine the above observation with the following standard result from sheaf theory.  For the statement, recall  the  
convention that $\Gamma(U; \mathcal{F})$ denotes right derived sections on $U$.

\begin{lemma}[\v{C}ech Lemma]    Let $\mathcal{F}$ be a complex of sheaves on $M\times \R$.  
Let $\mathcal{U} = \{U_i\}_{i=1}^n$ be a collection of open sets in $M\times \R$ with $U = \cup_{i=1}^n U_i$, and consider the \v{C}ech complex $C^*(\mathcal{U}; \mathcal{F})$ with 
\[
C^p(\mathcal{U};\mathcal{F}) = \oplus_{1 \leq i_0 < \ldots < i_p \leq n} \Gamma(U_{i_0} \cap\ldots \cap U_{i_p}; \mathcal{F})
\]
 as a double complex of $\mathbb{K}$-modules where the vertical differentials come from 
 $\mathcal{F}$ (more precisely, from a complex of injective sheaves with $\mathcal{F} \stackrel{\cong}{\rightarrow}\mathcal{I}$) and the horizontal differentials are the standard differentials from the \v{C}ech complex.  Then, the sum of restriction maps $\Gamma(U; \mathcal{F}) \rightarrow C^0(\mathcal{U};\mathcal{F}) = \oplus_{i=1}^n \Gamma(U_i; \mathcal{F})$ composed with the inclusion $C^0(\mathcal{U};\mathcal{F}) \hookrightarrow Tot^*(C^*(\mathcal{U}; \mathcal{F}))$ gives a quasi-isomorphism
\[
\Gamma(U;\mathcal{F}) \stackrel{\cong}{\rightarrow} Tot^*(C^*(\mathcal{U}; \mathcal{F})).
\]
\end{lemma}

\begin{proof}[Proof of $(\Leftarrow)$ in Theorem \ref{thm:FunSh}] 
Assuming $SS(\mathcal{F})\setminus($0$-\mbox{section}) \subset \Lambda^-$, whenever $s_0 \rightarrow s_1$ is a generization map as in Definition \ref{def:fun} (1) or (2), a small ball $B_{\delta_1}(q_1)$  centered at a point $q_1 \in s_1$ near $s_0$ can be expanded to a larger (but still small) ball $B_{\delta_0}(q_0)$ with $q_0 \in s_0$ in such a way that the Microlocal Morse Lemma implies  $\Gamma(B_{\delta_0}(q_0); \mathcal{F}) \stackrel{\cong}{\rightarrow} \Gamma(B_{\delta_1}(q_1); \mathcal{F})$.  [Since $s_0 \rightarrow s_1$ is downward or $s_0$ and $s_1$ belong to the same $\Lambda^F$-stratum, the expansion from $B_{\delta_1}(q_1)$ to $B_{\delta_0}(q_0)$ never pushes the boundary of the ball downward through $\pi_{xz}(\Lambda)$.]  Combined with Observation \ref{observe:1} we get a commutative diagram of restriction maps
\[
\xymatrix{ F(s_0)=& \Gamma(s_0^*; \mathcal{F})   \ar[r]  \ar[d]^\cong  & \Gamma(s_1^*;\mathcal{F}) \ar[d]^\cong &= F(s_1) \\
 & \Gamma(B_{\delta_0}(q_0); \mathcal{F}) \ar[r]^\cong & \Gamma(B_{\delta_1}(q_1);\mathcal{F}) & 
}  
\]  
where $F = \Gamma_\mathcal{S}(\mathcal{F})$.  
Thus, we conclude that the top arrow, $F(s_0 \rightarrow s_1)$, is a quasi-isomorphism as required in Definition \ref{def:fun} (1) or (2).

To verify Definition \ref{def:fun} (3), consider strata labeled $O, NW, NE, N$ as in the definition. 
 Let $f$ be stratified Morse with a critical point at $q\in O$ that is a local minimum of $f|_O$ such that $(\nabla f)_q$ points perpendicularly downward from $O$ into the bottom quadrant and lies strictly between the downward orthogonal vectors to $NW$ and $NE$ at $q$.  (The gradient can be taken using the Euclidean metric in local coordinates.)
 Using the \v{C}ech Lemma, and the Observation \ref{observe:1} the Morse group is quasi-isomorphic to the total complex from (\ref{eq:NW}).  (See the proof of Lemma \ref{lem:A1} below for more detail on this Morse group computation.)  Since $Mo_{q,f}(\mathcal{F})$ is acyclic as $df_q \notin SS(\mathcal{F})$, we see that Definition \ref{def:fun} (3) holds.

\end{proof}

\subsubsection{Preliminary Lemmas}
\label{sec:preliminary}

The following two lemmas record some useful consequences of Definition \ref{def:fun}.

 Suppose that $s_1,s_2$ and $s_3$ are strata with $s_1 \leq s_2 \leq s_3$ and so that $s_1 \rightarrow s_3$ and $s_{2} \rightarrow s_3$ are downward maps (as in Def. \ref{def:downward}).   In this case, we say that $s_1 \rightarrow s_2$ is a {\bf generalized downward} generization map.  For example, the map from a (Cu) $1$-strata to the Legendrian $2$-strata that is the lower sheet at the cusp edge is generalized downward but not downward.

\begin{lemma}
\label{lem:Downward}
When $F \in \mathbf{Fun}^\bullet_\Lambda(\mathcal{S},\mathbb{K})$ all generalized downward maps are sent to quasi-isomorphisms by $F$.
 \end{lemma}
\begin{proof} 
$F(s_2 \rightarrow s_3) \circ F(s_1 \rightarrow s_2) = F(s_1 \rightarrow s_3)$ and since two of the three maps are quasi-isomorphisms by Definition \ref{def:fun} (1), the third is as well.
\end{proof}

Next, consider a collection of four distinct strata $s_1, s_2^a, s_2^b, s_3 \in \mathcal{S}$ having generization maps:
\begin{equation} \label{eq:square}
\xymatrix{   & s_2^a \ar[rd] &  \\ s_1 \ar[ru] \ar[rd] & & s_3 \\  & s_2^b \ar[ru] & } \quad
\end{equation} 
We say that the diagram (\ref{eq:square}) is a {\bf top dimensional square} if
\begin{enumerate}
\item  $\dim s_1 = 1$, \quad $\dim s_2^a = \dim s_2^b =2$, \quad $\dim s_3 =3$;  and
\item  in the case that the $1$-stratum $s_1$ is of type (Cu) the $3$-stratum $s_3$ is the region between (in the $z$-direction) the two Legendrian two strata $s_2^a$ and $s_2^b$.  
\end{enumerate}

\begin{lemma}  \label{lem:topdim}
Let $F \in \mathbf{Fun}^\bullet_\Lambda(\mathcal{S},\mathbb{K})$.  Then, for every top dimensional square, the total complex of 
\[
\xymatrix{   & F(s_2^a) \ar[rd] &  \\ F(s_1) \ar[ru] \ar[rd] & & F(s_3) \\  & F(s_2^b) \ar[ru] & }
\]
is acyclic.
\end{lemma}

\begin{proof}
Definition \ref{def:fun} (3) makes this acyclicity requirement in the case that $s_1, s_2^a, s_2^b, s_3$ are strata of the form $O, NW, NE, N$ as in Figure \ref{fig:Cross}.  In all other cases, there is a pair of parallel arrows (i.e., either  arrows $s_1 \rightarrow s_2^a$ and $s_2^b \rightarrow s_3$ or arrows $s_1 \rightarrow s_2^b$ and $s_2^a \rightarrow s_3$) that are both quasi-isomorphisms.  [To verify, note that in the case that $s_1$ is of type $(F^2)$ or $(Cu)$ there is always a pair of parallel maps that are both (generalized) downward. In the case that $s_1$ is of type $(V^2)$ or $(FV)$ there is a pair of parallel maps that both point away from a common vertical plane, and they are quasi-isomorphisms by Definition \ref{def:fun} (2).]
\end{proof}

\subsubsection{Proving the forward direction of Theorem \ref{thm:FunSh}}

In establishing the forward implication of Theorem \ref{thm:FunSh}, it is convenient to be able to restrict attention to covectors in ``generic directions''  when checking the vanishing of Morse groups.    
\begin{lemma}
\label{lem:StratifiedGeneric2}  Let $\mathcal{F} \in \ShS$, and let $g$ be a Riemannian metric defined near $q \in J^0M$.
Assume $q \in s_0$, and let $U$ be the set of unit vectors in $T_q(J^0M)$ that are orthogonal to $T_q s_0$ and not orthogonal to $T_q(J^0M) \cap \overline{Ts_a}$ for any stratum $s_a > s_0$, i.e.
\begin{equation}
U = ST_q(J^0M) \cap \left((T_q s_0)^\perp \setminus \left( \bigcup_{s_0 < s_a} \left( T_q(J^0M) \cap \overline{Ts_a}\right)^\perp \right) \right).
\end{equation}   
Then, the following are equivalent. 
\begin{enumerate}
\item The set of non-zero characteristic vectors for $\mathcal{F}$ in $T^*_q(J^0M)$ are contained in $\mathbb{R}_{>0}\Lambda^-.$
\item For each $\xi \in U$, either $s_0$ is a Legendrian $2$-strata and $\xi$ is the downward orthogonal vector to $s_0$ (with negative
 $z$-component), or there exists a stratified Morse $f$ defined near $q$ 
 with $(\nabla f)_q = \xi$  such that 
 $Mo_{q,f}(\mathcal{F})$ is acyclic.
\end{enumerate}
\end{lemma}

\begin{proof}
That (1) implies (2) is clear from the definitions since (using the metric $g$ to identify $T(J^0M)$ and $T^*(J^0M)$) the only case in which $U \cap \R_{>0}\Lambda^- \neq \emptyset$ is when $U$ contains the downward orthogonal vector to a Legendrian $2$-strata.  

To prove the converse, note that unit vectors not belonging to $U$ are never characteristic vectors.  This is because the singular support of $\mathcal{F}$ at $s$ must be contained in $(T s_0)^{\perp}$ since $\mathcal{F}$ is constructible with respect to $\mathcal{S}.$  Moreover, a vector
 $\xi \in (T_qs_0)^\perp \cap (T_qs_0 \cap\overline{Ts_a})^\perp$ with $s_0 < s_a$ cannot be a characteristic vector simply because there are no stratified Morse functions with $\nabla_qf = \xi$.   
Thus, if (2) holds then a unit characteristic vector $\xi$, since it must belong to $U$,  
 can only be a downward orthogonal vector to a Legendrian $2$-strata, and in this case the characteristic vector lies in $\mathbb{R}_{>0}\Lambda^-$.   Finally, note that the set of all non-zero characteristic vectors is $\R_{>0}$ times the set of unit characteristic vectors, and hence is contained in $\R_{>0} \Lambda^-$ as required.
\end{proof}

\medskip

\begin{proof}[Proof of ($\Rightarrow$) in Theorem \ref{thm:FunSh}]
Let $\mathcal{F} \in \ShS$ and assume that the associated combinatorial sheaf $F = \Gamma_{\mathcal{S}}(\mathcal{F})$ satisfies Definition \ref{def:fun} so that $F \in  \mathbf{Fun}^\bullet_\Lambda(\mathcal{S},\mathbb{K})$.   Since $SS(\mathcal{F})$ is the closure of the set of characteristic vectors for $\mathcal{F}$, in order to establish the forward direction of Theorem \ref{thm:FunSh} it is enough to show that condition (2) from Lemma \ref{lem:StratifiedGeneric2} holds at every stratum $s_0$.  This is done in the following Lemmas \ref{lem:2strata}--\ref{lem:A2A1}. 
\end{proof}

\label{sec:2and3strata}
\begin{lemma}
\label{lem:2strata}
Suppose $q \in s_0$ where $s_0$ is a 2- or 3-dimensional stratum. Then, condition (2) of Lemma \ref{lem:StratifiedGeneric2} holds.
\end{lemma}
 
\begin{proof}
Condition (2) of Lemma \ref{lem:StratifiedGeneric2} is vacuously true if $s_0$ is a 3-stratum.   

When $s_0$ is a 2-stratum, let $\xi \in U$ and, in the case that $x_0$ is a Legendrian $2$-stratum, assume $\xi$ is not the downward orthogonal vector to $s_0$. 
 Suppose that $s_0$ is bordered by $3$-strata $s_a$ and $s_b$, labeled so that $\xi$ (resp. $-\xi$) points into $s_a$ (resp. $s_b$);   in the case that $s_0$ is a Legendrian 2-stratum, $s_0$ bounds $s_a$ from below.  To complete the proof, we check the vanishing of the Morse group $Mo_{q,f}(\mathcal{F})$ for a stratified Morse function $f \in C^\infty(B_\delta(q), \R)$ defined in a small ball $B_\delta(q) \subset J^0M$ having $(\nabla f)_q = \xi$  
 and such that $f|_{s_0}$ has a local minimum at $q$.  For $0 < \epsilon \ll \delta \ll 1$, the sublevel sets $f^{-1}(-\infty, f(q) + \epsilon)$ and $f^{-1}(-\infty, f(q) - \epsilon)$ are respectively (topologically) a small ball  divided in half by $s_0$ and a small ball contained in $s_b$.  Then, using the Microlocal Morse Lemma and Observation \ref{observe:1} we see that
\begin{align*}
Mo_{q,f}(\mathcal{F}) & = \mathit{Cone}\left(\Gamma(f^{-1}(-\infty, f(q) + \epsilon) ; \mathcal{F}) \rightarrow \Gamma(f^{-1}(-\infty, f(q) - \epsilon) ; \mathcal{F}) \right)  \\ & \cong \mathit{Cone}\left(\Gamma(s_0^* ; \mathcal{F}) \rightarrow \Gamma(s_b^* ; \mathcal{F}) \right) = \mathit{Cone}( F(s_0) \rightarrow F(s_b))  \\ &\cong 0.
\end{align*}
At the last equality we used that $F(s_0) \rightarrow F(s_b)$ is either downward or satisfies Definition \ref{def:fun} (2), and hence is a quasi-isomorphism since $F \in \mathbf{Fun}^\bullet_\Lambda(\mathcal{S},\mathbb{K})$.
\end{proof}

\begin{lemma}
\label{lem:A1}
Suppose $q \in s_0$ where $s_0$ is a 1-stratum of type ($F^2$), ($FV$) or  ($V^2$).  Then, condition (2) of Lemma \ref{lem:StratifiedGeneric2} holds. 
\end{lemma}

\begin{proof}
Recall from Remark \ref{rem:orthog} and the description of the $1$-strata that there are four $2$-strata adjacent to $s_0$ and that opposite pairs of $2$-strata piece together with $s_0$ into smooth planes.    
A local model  
 of  $\mathcal{S}|_{B_{\delta}(q)}$ in $\R^3_{t_1,t_2,t_3}$ associates each of the two intersecting planes to $\{t_2=0\}$ and $\{t_3=0\}.$ 
Here, the local coordinates $(t_1, t_2,t_3)$ are unrelated to the usual coordinates $(x,z)$ on $J^0M$.
 Using the Euclidean metric in the $(t_1,t_2,t_3)$-coordinates,  
 the collection of unit vectors $U$ from Lemma \ref{lem:StratifiedGeneric2} is a circle in the $t_2t_3$-plane with $4$ points removed.  (These points are $(0, 0, \pm1)$ and $(0, \pm1, 0)$,  corresponding to the orthogonal complement of the 2-dimensional strata in $\{t_2=0\}$ and $\{t_3=0\}$.)

To verify condition (2) of Lemma \ref{lem:StratifiedGeneric2}, let $\xi \in U$.  We can assume the coordinates have been chosen so that $\xi$ points into the quadrant $\{t_2 <0\} \cap \{t_3 <0 \}$.  Notate the strata that form the closed quadrant {\it opposite} to $\xi$ as 
\[
\begin{array}{ll}
T_1 = s_0 = \{t_2 =0\} \cap \{t_3 =0\}, \quad & T_1T_2 = \{t_2>0\} \cap \{t_3 =0\}, \quad  \\  T_1T_3 = \{t_2=0\} \cap \{t_3 >0\}, \quad &  T_1T_2T_3 = \{t_2>0\} \cap \{t_3 >0\}.
\end{array}
\]
  We check acyclicity of $Mo_{q,f}(\mathcal{F})$ using a stratified Morse $f \in C^\infty(B_\delta(q), \R)$ 
  having $(\nabla f)_q = \xi$ and such that $f|_{s_0}$ has a local minimum at $q$.  For $0 < \epsilon \ll \delta \ll 1$, note that (topologically) $f^{-1}(-\infty, f(q)+\epsilon)$ is a small ball centered at $q$, while $f^{-1}(-\infty, f(q) - \epsilon) = V \cup V'$ where $V$ and $V'$ are small balls centered at points on $T_1T_2$ and $T_1T_3$ respectively and $V \cap V'$ is a ball contained in $T_1T_2T_3$.  Thus, using the Observation \ref{observe:1} and the \v{C}ech Lemma, $Mo_{q,f}(\mathcal{F})$ is quasi-isomorphic to the total complex of the diagram
  \[
  \xymatrix{  & F(T_1T_2) \ar[rd] & \\ F(T_1) \ar[ru] \ar[rd] & & F(T_1T_2T_3) \\ & F(T_1T_3) \ar[ru] & }
  \]
which is acyclic by Lemma \ref{lem:topdim} since $T_1, T_1T_2, T_1T_3$, and $T_1T_2T_3$ form a top dimensional square.

\end{proof}

\begin{lemma}
\label{lem:A2}
Suppose $q \in s_0$ where $s_0$ is a 1-stratum of type ($Cu$).  
 Then, condition (2) of Lemma \ref{lem:StratifiedGeneric2} holds.
\end{lemma}

\begin{proof}

Label strata as $T_1, T_1T_2, T_1 T_3, T_1 T_2 T_3$ where $T_1=s_0$ denotes the cusp edge; $T_1T_2$ is the lower Legendrian $2$-stratum that meets the cusp edge;  $T_1T_3$ is the upper Legendrian $2$-stratum; and $T_1T_2T_3$ is the 3-stratum ``inside"  the cusping Legendrian 2-strata, respectively. Let $O$ denote the 3-stratum ``outside" the cusping 2-strata.  The set $U$ from Lemma \ref{lem:StratifiedGeneric2} has two components, $U_1$ and $U_2$,  corresponding to unit vectors orthogonal to the cusp edge and pointing into or out of the half space containing $T_1T_2T_3$ respectively.
Again, given $\xi \in U$ we compute $Mo_{q,f}(\mathcal{F})$ using a stratified Morse $f \in C^\infty(B_\delta(q), \R)$ 
  having $(\nabla f)_q = \xi\in U$ and such that $f|_{s_0}$ has a local minimum at $q$.  

When $\xi \in U_2$, the computation of the Morse group is as in the previous Lemma. (It's acyclic.)   When $\xi \in U_1$, using Observation \ref{observe:1} we have
\[ \Gamma(f^{-1}(-\infty, f(q) + \epsilon);\mathcal{F}) \cong F(T_1) \quad \mbox{and} \quad \Gamma(f^{-1}(-\infty, f(q)-\epsilon);\mathcal{F}) \cong F(O).\]
Since $T_1 \rightarrow O$ is downward, $F(T_1) \rightarrow F(O)$ is a quasi-isomorphism, and it follows that $Mo_{q,f}(\mathcal{F})$ is acyclic as required.

\end{proof}

\begin{lemma}
\label{lem:A1A1}
Suppose $q \in s_0$ where $s_0$ is a $0$-stratum of type ($F^3$), ($F^2V$) or  ($FV^2$). 
 Then, condition (2) of Lemma \ref{lem:StratifiedGeneric2} holds.
\end{lemma}

\begin{proof}
The $0$-stratum $s_0$ lies at the intersection of three front and/or vertical planes, so we can work with a local model  
of  $\mathcal{S}|_{B_{\delta}(q)}$ in $\R^3_{t_1,t_2,t_3}$ where each plane is defined by $\{t_i = 0\}.$
 Using the Euclidean metric in $\R^3_{t_1,t_2,t_3}$, the collection of unit vectors $U$ from Lemma \ref{lem:StratifiedGeneric2} is a $2$-sphere divided into $8$ components by removing its intersections with the coordinate planes.  Given $\xi \in U$, we can assume coordinates have been chosen so that $\xi$ points into the octant with all $t_i <0$.  Label the strata that form the closed  octant {\it opposite} from $\xi$ as $s_0$, $T_i$, $T_jT_k$, $T_1T_2T_3$ with $1\leq i \leq 3$ and $1 \leq j < k \leq 3$ where $s_0$ is the origin, $T_i$ is the positive $t_i$-axis, $T_jT_k$ is the quadrant in the $t_jt_k$-plane with $t_j>0$ and $t_k >0$, and $T_1T_2T_3$ is the open octant with $t_i>0$ for $1 \leq i \leq 3$.  Moreover, we can assume our choice of coordinates is such that $T_2 T_3$ is a Legendrian 2-stratum.

We compute the Morse group $Mo_{q,f}(\mathcal{F})$ for a linear function $f \in C^\infty(B_\delta(q))$ with $(\nabla f)_q =\xi$.  
  Then, $\Gamma(f^{-1}(-\infty, f(q) + \epsilon);\mathcal{F}) \cong F(s_0)$ and $f^{-1}(-\infty, f(q) - \epsilon)$ has an open covering $U_1 \cup U_2 \cup U_3$  such that (using Observation \ref{observe:1})
\[
\Gamma(U_i;\mathcal{F}) \cong \Gamma(T_i^*; \mathcal{F}), \quad \Gamma(U_j\cap U_k;\mathcal{F}) \cong \Gamma(T_jT_k^*; \mathcal{F}), \quad \Gamma(U_1\cap U_2 \cap U_3;\mathcal{F}) \cong \Gamma(T_1T_2T_3^*; \mathcal{F})
\]
(with all quasi-isomorphisms given by restrictions).  Thus, applying the \v{C}ech Lemma we see that $Mo_{q,f}(\mathcal{F})$ is quasi-isomorphic to the total complex of 
\begin{equation}
\label{eq:2strata}
\xymatrix{ & F(T_1)   \ar[r]  \ar[rd]  & 
 F(T_1T_2) \ar[rd] & \\
 F(q) \ar[ru] \ar[r] \ar[rd] &
  F(T_2) \ar[rd] \ar[ru] &
 F(T_1T_3) \ar[r] &
 F(T_1 T_2 T_3) \\
 & F(T_3) \ar[r] \ar[ru] &
 F(T_2 T _3) \ar[ru] &
} 
\end{equation}
where the arrows are generization maps with signs as in the \v{C}ech complex.  Consider the subcomplex and  quotient complex, given respectively by
\begin{equation}
\label{eq:Fig8Subcomplex}
F(T_1) \rightarrow
 F(T_1T_2) 
\oplus F(T_1T_3)
\rightarrow
F(T_{1}T_{2}T_{3}),
\end{equation}
\begin{equation}
\label{eq:Fig8Subcomplex2}
F(q) \rightarrow
 F(T_2) 
\oplus F(T_3)
\rightarrow
F(T_2T_3).
\end{equation}
Note the subcomplex (\ref{eq:Fig8Subcomplex}) is acyclic by Lemma \ref{lem:topdim} since it is the total complex of a top dimensional square.  
So $Mo_{q,f}(\mathcal{F})$ is quasi-isomorphic to the quotient complex (\ref{eq:Fig8Subcomplex2}).

Next, consider the top-dimensional square obtained by ``shifting'' all the strata $X$ that appear in (\ref{eq:Fig8Subcomplex2}) in either the $\pm t_1$-direction where the sign is chosen so that the shifted strata sit below the Legendrian front plane containing $T_2T_3$ (with respect to the $z$-coordinate on $J^0M$.)  Denote the shifted strata as $T_1^\pm X$.  [For example, if the shift is in the positive $t_1$-direction, $T_1^\pm q = T_1, T_1^\pm T_j = T_1 T_j, T_1^\pm T_2T_3 = T_1 T_2T_3$ and we recover the original subcomplex (\ref{eq:Fig8Subcomplex}).]    
Each generization map $F(X) \rightarrow F(T_1^\pm X)$ is a quasi-isomorphism since $X \rightarrow T_1^\pm X$ is (generalized) downward.  Thus, the quotient complex (\ref{eq:Fig8Subcomplex2}) is quasi-isomorphic to the total complex of a top-dimensional square, and Lemma \ref{lem:topdim} implies it is acyclic.

\end{proof}

\begin{lemma}
\label{lem:A2A1}
Suppose $q \in s_0$ where $s_0$ is a $0$-strata of type ($FCu$) or ($VCu$).  Then, condition (2) of Lemma \ref{lem:StratifiedGeneric2} holds.
\end{lemma}

\begin{proof}
We begin by fixing notation for the various strata adjacent to $s_0$.  Let $T^\pm_1$ denote the two 1-strata sitting in the cusp edge.  
Let $T_2$ and $T_3$ denote the $1$-strata that are respectively the intersection of the lower and upper cusping sheets with the third Legendrian front plane (in the ($FCu$) case) or the vertical plane (in the ($VCu$) case).  Let $T_1^\pm T_k$ denote the unique $2$-strata bordered by $T_1^\pm$ and $T_k$ for $k =2,3$.  Let $T_2T_3$ and $\overline{T_2 T_3}$ denote the two $2$-strata within the plane (vertical or Legendrian front) which intersects the cusping sheets, where $T_2T_3$ lies between the cusping sheets in the $z$-direction.  Finally, let $T_1^\pm T_2 T_3,$ resp.  $O^\pm,$ denote the 3-stratum whose closure contains $T_1^\pm,T_2,T_3$ and which lies in the interior, resp. exterior, of the cusping sheets.   See Figure \ref{fig:TStrata}. 

The set $U$ from Lemma \ref{lem:StratifiedGeneric2} is a $2$-sphere divided into $4$ components by removing the circles orthogonal to $T_1^\pm$ and $T_2$ at $q$. (Note that $T_3$ and $T_2$ are tangent at $q$.)  Given $\xi \in U$, we compute $Mo_{q,f}(\mathcal{F})$ for a linear function $f \in C^\infty(B_\delta(q))$ with $(\nabla f)_q =\xi$.  Taking $0 < \epsilon  \ll \delta \ll 1$, the computation is done in two cases each of which corresponds to two components of $U$.

\medskip

\noindent{ \bf Case 1.}  $f^{-1}(-\infty,f(q) - \epsilon)$ intersects three of the $1$-strata.  These $1$-strata will be $T_2, T_3$, and one of $T_1^{\pm}$; reversing notation if necessary, we can assume it is $T_1^+$ which we abbreviate simply as $T_1$.  In this case, $Mo_{q,f}(\mathcal{F})$ is verified to be acyclic via an argument that is similar to the proof of Lemma \ref{lem:A1A1}:  The \v{C}ech Lemma and the Observation \ref{observe:1} show that $Mo_{q,f}(\mathcal{F})$ is quasi-isomorphic to the total complex from (\ref{eq:2strata}) and the subcomplex (\ref{eq:Fig8Subcomplex}) is an (acyclic) top-dimensional square.  In the ($FCu$) case, the acyclicness of the quotient (\ref{eq:Fig8Subcomplex2}) is established as in Lemma \ref{lem:A1A1}, since we can again shift strata into the region below the front plane  
 via generalized downward maps to obtain a top dimensional square.  In the ($VCu$) case, the same shifting argument can be applied since this time the shifting maps (in either direction) are all quasi-isomorphisms via Definition \ref{def:fun} (2).

\medskip

\noindent{ \bf Case 2.}  $f^{-1}(-\infty, f(q) - \epsilon)$ intersects only one of the $1$-strata.  The $1$-stratum will be $T_1^\pm$, and we can assume it is $T_1^+ = T_1$.  In this case, we can use an open cover $f^{-1}(-\infty, f(q) - \epsilon) = U \cup V$ where (topologically) $U$, $V$, and $U\cap V$ are small balls centered at points on $T_1$, $\overline{T_2 T_3}$, and in $O^+$.  Thus, $Mo_{q,f}(\mathcal{F})$ is quasi-isomorphic to the total complex of
\begin{equation}
\label{eq:O}
F(q) \rightarrow
 F(\overline{T_2 T_3}) 
\oplus F(T_1)
\rightarrow
F(O^+).
\end{equation}
This is acyclic since $T_1 \rightarrow O^+$ is downward and $q \rightarrow \overline{T_2 T_3}$ is downward (resp. generalized downward) in the ($VCu$) (resp. ($FCu$)) case.

\end{proof}

\begin{figure}
\labellist
\scriptsize
%\pinlabel $T_1T_2T_3$ [tl] at 136 60
\pinlabel $T_2T_3$ [tr] at 208 126
\pinlabel $T_1^+$ [b] at 35 76
\pinlabel $q$ [b] at 70 116
\endlabellist
\centerline{\includegraphics[scale=.85 ]{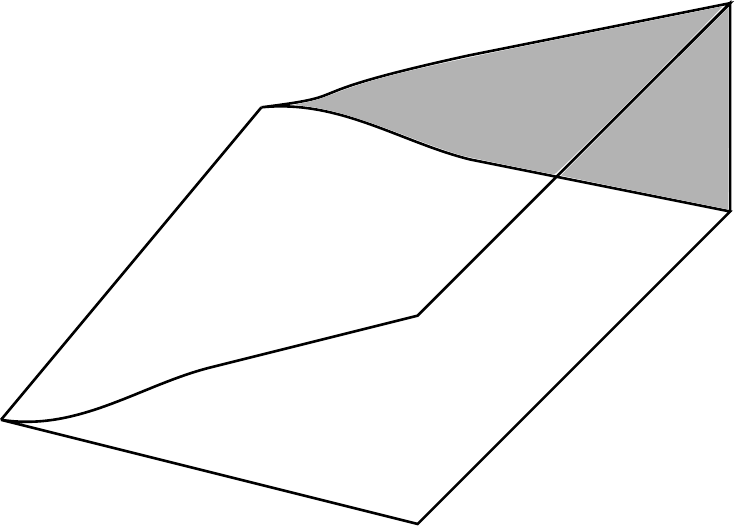}}
\caption{ 
The $(VCu)$ case.  The vertical 2-stratum $T_2T_3$ is shaded.} 
\label{fig:TStrata}
\end{figure}

\section{Simplex diagrams and generalized mapping cylinders}
\label{sec:SDandGMC}

This section is algebraic in nature.  In \ref{sec:SimplexDiagrams}, we introduce a category of {\it simplex diagrams} which are certain homotopy commutative diagrams associated to simplices.  (Throughout Section \ref{sec:SDandGMC}, the dimension of simplices may be arbitrary.)  Simplex diagrams allow for an alternate description of augmentations of the simplicial DGA that will be established in Section 5.  In constructing a combinatorial sheaf from an augmentation, we will need to be able to convert simplex diagrams into fully commutative diagrams.  This is done via a generalized mapping cylinder construction whose properties are established in Section \ref{sec:GMC}.  The reader may find it interesting to compare with Viterbo's rectification procedure in \cite{Viterbo} which serves a similar purpose.

\begin{remark}
The definitions and constructions in this section are stated for a general simplex $\Delta$.  This is in part to simplify notations.  In later sections, we will be interested in the case $\Delta = \overline{e_I}$ with $e_I \in \mathcal{E}$ where $\mathcal{E}$ is a simplicial decomposition  compatible with a Legendrian $\Lambda$.

\end{remark}

\subsection{Simplex diagrams}
\label{sec:SimplexDiagrams}

Let  $\Delta$ be an  $n$-simplex with ordered vertices $v_0, \ldots, v_n$.   
 Given $m \leq n$ and $i_0 < i_1 < \ldots < i_m$, write $F=[i_0, \ldots, i_m]$ for the $m$-dimensional face (i.e. sub-simplex) of $\Delta$ determined by the vertices $v_{i_0}, \ldots, v_{i_m}$.

\begin{definition}\label{def:simplex}
A {\bf simplex diagram} is a triple $(\Delta, V, \{a_F\})$ consisting of
\begin{itemize}
\item an $n$-simplex $\Delta$,
\item a graded $\mathbb{K}$-module $V=V^*$, and 
\item  for each face $F= [i_0, \ldots, i_m] \subset \Delta_n$, a linear map $a_F= a_{i_0 \ldots i_m}: V^* \rightarrow V^{*+1-m}$ of degree $1- \dim F$ satisfying the identity
\begin{equation} \label{eq:simplex}
\sum_{0< k <m} (-1)^{k+1} a_{i_0\ldots \widehat{i_k} \ldots i_m}+ \sum_{0\leq k \leq m} (-1)^ka_{i_k \ldots i_m} \circ a_{i_0 \ldots i_k} = 0.
\end{equation}
\end{itemize}
\end{definition}

\begin{remark} \label{rem:Simp2} 
The definition of simplex diagrams 
can be conveniently reframed using the {\it reduced} cobar differential from (\ref{eq:Dreduced}).  Assume $V$ is finite dimensional so that $V = (V^\vee)^\vee$ where $V^\vee = \oplus_{k \in \Z} (V^{-k})^\vee$ denotes the vector space dual of $V$ with grading $(V^{\vee})^k = (V^{-k})^\vee$.  Then, the data of a simplex diagram $(\Delta, V, \{a_F\})$ is equivalent to a DGA homomorphism
\[
\alpha: (\Omega C_*(\Delta), D^{\mathit{red}}_s) \rightarrow (\mathit{End}(V^\vee), 0).
\] 
Indeed, when we put $a_F = \alpha(F)^\vee$ the required equations (\ref{eq:simplex}) become equivalent to the identity
\[
\alpha \circ D^{\mathit{red}}_s =0.
\]
\end{remark}

\begin{example} \label{ex:2}  When the face $F$ has dimension $0$, $1$, or $2$, the linear map $a_F$ can be viewed as a differential, a cochain map, or a homotopy operator respectively.  In these cases, we use more suggestive notations
\[
a_{i_0} = d_{i_0}, \quad a_{i_0i_1} = f_{i_0i_1}, \quad a_{i_0i_1i_2} = K_{i_0i_1i_2}.
\]
Indeed, the identity (\ref{eq:simplex}) for $m =0, 1, 2$ then translates to 
\[
\begin{array}{cr}
 d_{i_0}^2 = 0 &  \mbox{(differential)} \\
f_{i_0i_1} \circ d_{i_0} = d_{i_1} \circ f_{i_0i_1} & \mbox{(cochain map)} \\
d_{i_2} K_{i_0i_1i_2} + K_{i_0i_1i_2}d_{i_0} = f_{i_1i_2} \circ f_{i_0i_1} -f_{i_0i_2}  & \mbox{(homotopy).}
\end{array}
\]
Thus, when $m=0,1,2$ a simplex diagram with underlying graded $\mathbb{K}$-module $V^*$ consists of an assignment of 
\begin{enumerate}
\item a degree $+1$ differential $d_{i}$ to each vertex $v_i$; 
\item a chain map $f_{i_0i_1}:(V^*, d_{i_0}) \rightarrow (V^*, d_{i_1})$ to each edge $[i_0,i_1]$;
\item and to the $2$-dimensional face $[i_0,i_1,i_2]$ a chain homotopy operator $K_{i_0i_1i_2}: (V^*, d_{i_0}) \rightarrow (V^{*-1}, d_{i_2})$ realizing a chain homotopy from $f_{i_1i_2} \circ f_{i_0i_1}$ to $f_{i_0i_2}$.
\end{enumerate}
\end{example}

\begin{remark}
For faces of dimension $\geq 3$, the map $a_{i_0\ldots i_m}$ should be interpreted as an $(m-1)$-homotopy (i.e. a higher homotopy) between two $(m-2)$-homotopies built from the faces of $[i_0, \ldots, i_m]$.  For instance, for a $3$-simplex $\Delta_3 = [0,1,2,3]$, there are two $1$-homotopies $H= K_{123}f_{01}+K_{013}$ and $H'= f_{23}K_{012}+K_{023}$ between the chain maps $f_{03}$ and $f_{23}f_{12}f_{01}$.   (Geometrically, $H$ pushs the edge $[0,3]$ over the face $[0,1,3]$ and then pushes the $[1,3]$ edge over the face $[1,2,3]$, while $H'$ pushes $[0,3]$ over the other half of the boundary of $\Delta_3$.)  The equation (\ref{eq:simplex}) becomes
\[
d_3 a_{0123} - a_{0123} d_0 = H - H',
\]
i.e., $a_{0123}$ is a higher homotopy between $H$ and $H'$.
\end{remark}

\begin{definition} \label{def:sdmorph} Let $(\Delta, V, \{a_F\})$ and $(\Delta', W, \{b_F\})$ be simplex diagrams such that $\Delta \subset \Delta'$, i.e. $\Delta$ is a face of $\Delta'$. 
A {\bf morphism} from $(\Delta, V, \{a_F\})$ to $(\Delta', W, \{b_F\})$ is a degree $0$ linear map $\varphi: V^* \rightarrow W^*$ that commutes with face maps, i.e. for any face $F \subset \Delta \subset \Delta'$, we have $\varphi \circ a_F = b_F \circ \varphi$.
 There is a resulting {\bf category of simplex diagrams}, $\mathfrak{sd}$, where morphisms compose in an obvious manner.
\end{definition}

\subsection{Generalized mapping cylinders}
\label{sec:GMC}

In the later construction of sheaves, it will be useful to be able to convert from simplex diagrams, which are only commutative up to homotopy, to fully commutative diagrams.  This will be accomplished, after a subdivision of $\Delta$, via a generalization of the mapping cylinder construction that we now discuss.

\begin{definition}  Let $\mathcal{D} = (\Delta, V, \{a_F\})$ be a simplex diagram where $\dim \Delta = n$.  The {\bf mapping cylinder} of $\mathcal{D}$ is the cochain complex
\[
\mathit{Map}_n(\mathcal{D}) = \left( \bigoplus_{F \le \Delta} V_F, D \right) \quad \mbox{where $V_F = V[\dim F]$}.
\]
The notation $V[k]$ is $V$ with grading shifted down by $k$, i.e. $(V[k])^* = V^{*+k}$.  
The sum is over all faces $F \le \Delta$, and the differential satisfies 
\[
D|_{V_E} = \sum_{E \ge F}  \alpha^E_F\,,  \quad  \alpha^E_F:V_E \rightarrow V_F
\]
where for faces $E \ge F$ such that $E = [i_0,\ldots, i_m]$, $F= [j_0, \ldots, j_l]$ (with $i_0 < \ldots < i_m$ and $j_0 < \ldots < j_l$)  we have  
\begin{equation}  \label{eq:MapDiff}
\alpha^E_{F} = \left\{ \begin{array}{clr} 
  (-1)^{m+1}a_{i_0\ldots i_{m-l}}, & (i_{m-l}, \ldots, i_m) = (j_0, \ldots, j_l), & \mbox{  for some $0\leq l \leq m$}, \\
	(-1)^{m+k+1}\mathit{id}_V, & (i_0, \ldots, \widehat{i_k}, \ldots, i_m) = (j_0, \ldots, j_l), & \mbox{  for some $1 \leq k \leq m$},\\
	0, & \mbox{else.} & \end{array} \right.
\end{equation}
More succinctly, we have
\begin{align} \label{eq:MapDiff2}
D|_{V_{i_0\cdots i_m}} = &\sum_{k=0}^m (-1)^{m+1} a_{i_0 \ldots i_k}:\left(V_{i_0\ldots i_m} \rightarrow V_{i_k \ldots i_m}\right)  \\
&+ \sum_{k=1}^m(-1)^{m+k+1} \mathit{id}:\left(V_{i_0 \ldots i_m} \rightarrow V_{i_0\ldots \widehat{i_k} \ldots i_m}\right). \notag
\end{align}

\end{definition}

\begin{example}  \label{ex:1}
Let $\mathcal{D} = (\Delta_n, V, \{a_F\})$ be a simplex diagram for the standard $n$-simplex, $\Delta_n = [0,1, \ldots, n]$.  
\begin{enumerate}
\item When $n = 0$, $\mathcal{D}$ just amounts to a cochain complex, $(V^*,d_0)$, and in this case $\mathit{Map}_0(\mathcal{D}) = (V^*,d_0)$.
\item When $n = 1$, $\mathcal{D}$ specifies a chain map $f_{01}:(V^*,d_0) \rightarrow (V^*, d_1)$, and $\mathit{Map}_1(\mathcal{D})$ is the usual mapping cylinder chain complex for $f$,
\[
V_0 \oplus V_{01} \oplus V_1 = V \oplus V[1] \oplus V \quad \mbox{with} \quad D = \left( \begin{array}{ccc} -d_0 & -1 & 0 \\ 0 & d_0 & 0 \\ 0& f & -d_1 \end{array} \right) 
\]
where $1$ indicates ${id}_V$, so $D(u,v,w) = (-d_0(u) -v, d_0(v), f(v)-d_1(w))$.
\item When $n = 2$, we have $\mathit{Map}_2(\mathcal{D})  = V_0 \oplus V_1 \oplus V_2 \oplus V_{01} \oplus V_{02} \oplus V_{12} \oplus V_{012}$ with  
\[
D = \left( \begin{array}{ccccccc}
-d_0 & 0 & 0 & -1 & -1 & 0 & 0  \\
0 & -d_1 & 0 & f_{01} & 0 & -1 & 0  \\
0 & 0 &  -d_2 & 0 & f_{02} & f_{12} & -K_{012}  \\
0 & 0 & 0 & d_0 & 0 & 0 & -1  \\
0 & 0 & 0 & 0 & d_0 & 0 & 1  \\
0 & 0 & 0 & 0 & 0 & d_1 & -f_{01}  \\
0 & 0 & 0 & 0 & 0 & 0 & -d_0  
 \end{array} \right).
\]
\end{enumerate}
\end{example}

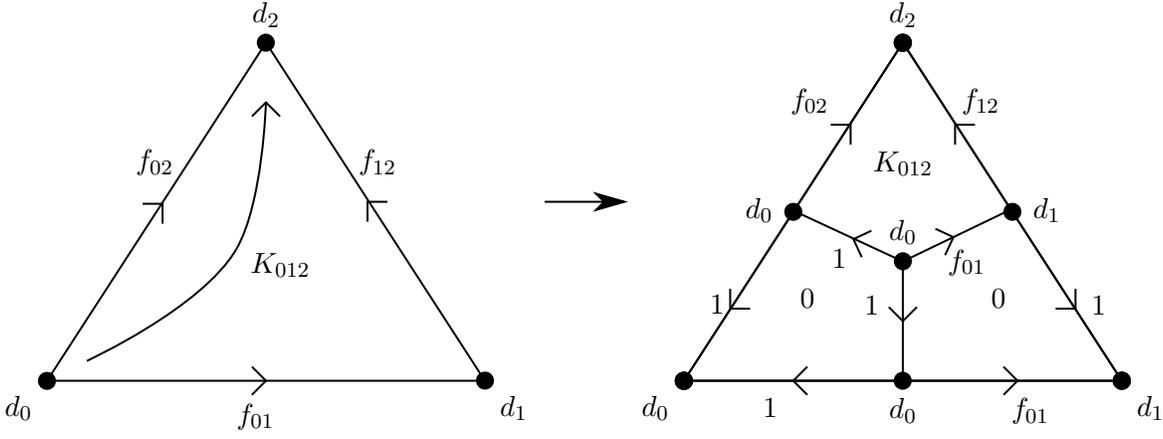
\begin{figure}
\labellist
\small
\pinlabel $d_0$ [tr] at 0 2
\pinlabel $d_1$ [tl] at 172 2
\pinlabel $d_2$ [b] at 86 136
\pinlabel $f_{01}$ [t] at 82 0
\pinlabel $f_{12}$ [b] at 128 80
\pinlabel $f_{02}$ [b] at 44 80
\pinlabel $K_{012}$ [tl] at 78 56
\pinlabel $d_0$ [tr] at 240 2
\pinlabel $d_1$ [tl] at 412 2
\pinlabel $d_2$ [b] at 326 136
\pinlabel $1$ [t] at 276 2
\pinlabel $f_{01}$ [t] at 374 2
\pinlabel $1$ [b] at 256 28
\pinlabel $f_{02}$ [br] at 300 104
\pinlabel $f_{12}$ [bl] at 346 104
\pinlabel $1$ [t] at 302 58
\pinlabel $f_{01}$ [t] at 350 58
\pinlabel $1$ [r] at 320 34
\pinlabel $1$ [l] at 394 34
\pinlabel $K_{012}$ [b] at 326 80
\pinlabel $0$ [b] at 290 30
\pinlabel $0$ [b] at 362 30
\pinlabel $d_0$ [t] at 326 2
\pinlabel $d_0$ [r] at 280 70
\pinlabel $d_1$ [l] at 372 70
\pinlabel $d_0$ [b] at 326 54

\quad

\endlabellist
\centerline{\includegraphics[scale=1]{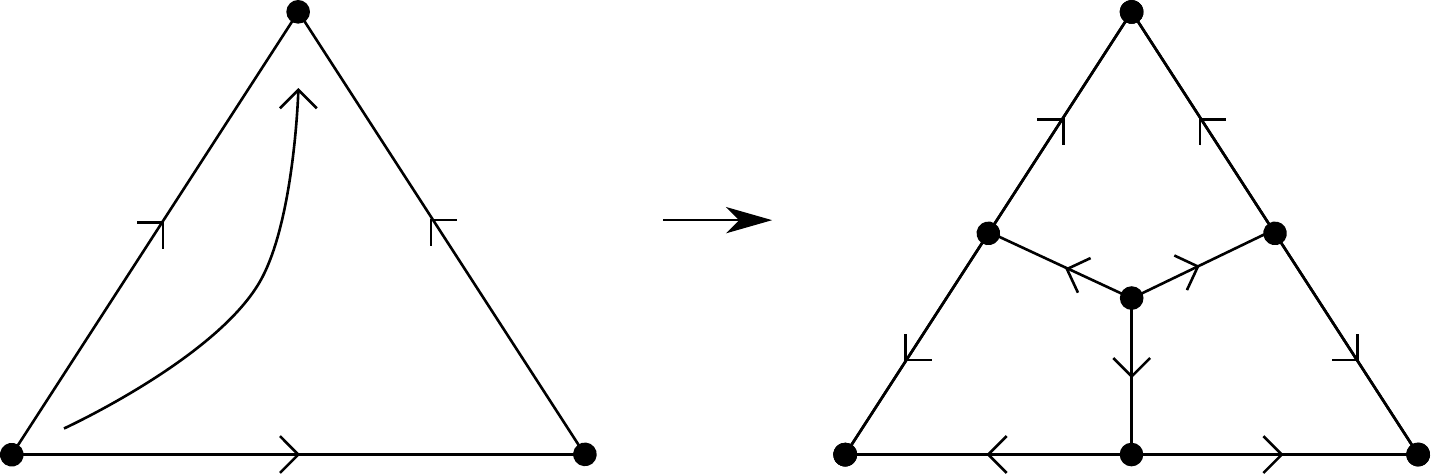}}

\quad

\caption{The face maps from a simplex diagram (left) with $\dim \Delta =2$ correspond (up to sign) to the components $\alpha^E_{F}$ as indicated. A $1$ indicates the identity map.}
\label{fig:MapDiff}
\end{figure}

\begin{remark}  There is also a geometric description of the mapping cylinder differential.  
Figure \ref{fig:MapDiff} demonstrates how components of the mapping cylinder differential arise from ``sliding'' the simplex diagram maps $\{a_F\}$ to the cells of the pair-subdivision of $\Delta$ when $n=2.$  (See Section \ref{sec:handle}.) 
 The component $\alpha^E_F:V_E \rightarrow V_F$ of $D$ 
 then appears (up to sign) as the labeling 
 of the cell $(F,E)$.
\end{remark}

\begin{proposition}
The mapping cylinder $\mathit{Map}_n(\mathcal{D})$ has $D^2 = 0.$
\end{proposition}

\begin{proof}
Let $\langle f V_i, V_j \rangle$ denote the $j$-th component of the linear map $f: \oplus_k V_k \rightarrow \oplus_k V_k$ restricted to the factor $V_i.$  From (\ref{eq:MapDiff2}) we see that there are only three cases where the components of $D^2$ can have non-zero terms.

\medskip

\noindent{\bf Case 1:}  For $ 0 \leq l \leq m$,
\begin{eqnarray*}
\langle D^2 V_{i_0 \ldots i_m}, V_{i_l \ldots i_m} \rangle & =& 
\sum_{k=0}^l (-1)^{m+1}  \langle D V_{i_k \ldots i_m}, V_{i_l \ldots i_m} \rangle \circ a_{i_0 \ldots i_k}
+ \sum_{k=1}^{l-1} (-1)^{m+k+1}\langle D V_{i_0 \ldots \hat{i}_k \ldots  i_m}, V_{i_l \ldots i_m} \rangle \circ   id   \\
& =& \sum_{k=0}^l (-1)^{m+1+m-k+1} a_{i_k \ldots i_l} \circ a_{i_0 \ldots i_k} 
+ \sum_{k=1}^{l-1} (-1)^{m+k+1+m-1+1} a_{i_0 \ldots \hat{i}_k \ldots i_l}
\end{eqnarray*}
which vanishes by (\ref{eq:simplex}).

\medskip

\noindent{\bf Case 2:} For $0 \leq l < k \leq m$,
\begin{align*}
\langle D^2 V_{i_0 \ldots i_m}, V_{i_l \ldots \hat{i}_k \ldots i_m} \rangle & = (-1)^{m+1} \langle D V_{i_l\ldots i_m}, V_{i_l\ldots \hat{i}_k \ldots i_m} \rangle \circ a_{i_0\ldots i_l} + (-1)^{m+k+1} \langle D V_{i_0\ldots \hat{i}_k \ldots i_m}, V_{i_l\ldots \hat{i}_k \ldots i_m} \rangle \circ \mathit{id} \\
 &= (-1)^{m+1+(m-l)+(k-l)+1} a_{i_0\ldots i_l} + (-1)^{m+k+1+m} a_{i_0\ldots i_l} = 0.
\end{align*}

\medskip

\noindent{\bf Case 3:} For $0 <k <l \leq m$,
\begin{align*}
\langle D^2 V_{i_0 \ldots i_m}, V_{i_0 \ldots \hat{i}_k \ldots \hat{i}_l \ldots i_m} \rangle & = (-1)^{m+k+1} \langle D V_{i_0\ldots\hat{i}_k \ldots i_m}, V_{i_0\ldots \hat{i}_k \ldots \hat{i}_l \ldots i_m} \rangle \circ \mathit{id}  \\
& \quad \quad + (-1)^{m+l+1} \langle D V_{i_0\ldots\hat{i}_l \ldots i_m}, V_{i_0\ldots \hat{i}_k \ldots \hat{i}_l \ldots i_m} \rangle \circ \mathit{id} \\
 &= (-1)^{m+k+1+(m-1)+(l-1)+1} \mathit{id} + (-1)^{m+l+1+(m-1)+k+1} \mathit{id} = 0.
\end{align*}

\end{proof}

\begin{proposition} \label{prop:MappingFun}  
The mapping cylinder construction gives a functor $\mathit{Map}:\mathfrak{sd} \rightarrow \mathit{Ch}(\mbox{$\mathbb{K}$-mod})$ from the category of simplex diagrams to the category of cochain complexes of $\mathbb{K}$-modules with cochain maps.
\end{proposition}
\begin{proof}  When $\varphi: \mathcal{D} \rightarrow \mathcal{D}'$ is a morphism of simplex diagrams where $\mathcal{D} = (\Delta, V, \{a_F\})$ and $\mathcal{D}' = (\Delta', W, \{b_F\})$, since $\Delta \subset \Delta'$ (see Definition \ref{def:sdmorph}) we can apply $\varphi$ component by component and then compose with the inclusion to obtain a chain map
\[
\mathit{Map}_m(\mathcal{D}) = \oplus_{F \leq \Delta} V_F \stackrel{\varphi}{\rightarrow} \oplus_{F \leq \Delta} W_F \hookrightarrow \oplus_{F \leq \Delta'} W_F = \mathit{Map}_n(\mathcal{D}'). 
\]
Clearly, compositions and identity morphisms are preserved.
\end{proof}

\begin{proposition}  \label{prop:vertex}  
Let $\mathcal{D}=(\Delta, V, \{a_F\})$ be a simplex diagram with the property that face maps associated to edges are quasi-isomorphisms, i.e. for any $F = [i_0,i_1] \subset \Delta$ with $\dim F =1$ the cochain map
\[
a_{i_0i_1}:(V^*, a_{i_0}) \rightarrow (V^*, a_{i_1})
\]
is a quasi-isomorphism.  Moreover, let $v_j \subset \Delta$ be a vertex and $\mathcal{D}|_{v_j}=(v_j, V, \{a_{v_j}\})$ be the simplex diagram obtained from restricting $\mathcal{D}$ to $v_j$. 
Then, the map of simplex diagrams $\mathcal{D}|_{v_j} \rightarrow \mathcal{D}$ arising from $v_j \subset \Delta$ and $\mathit{id}:V^* \rightarrow V^*$  
induces a quasi-isomorphism of mapping cylinders, 
\[
\mathit{Map}_0(\mathcal{D}|_{v_j}) \stackrel{\cong}{\hookrightarrow} \mathit{Map}_n(\mathcal{D}).
\]
\end{proposition}
\begin{proof}
This follows since the quotient $Q = \mathit{Map}_n(\mathcal{D})/\mathit{Map}_0(\mathcal{D}|_{v_j})$ 
 is acyclic.  To verify this statement, let $X$ denote the set of faces of $\Delta$ other than the vertex $v_j$, so that as a vector space,
\[
Q= \oplus_{F \in X} V_F.
\]
Write $X = X_0 \sqcup X_1$ where $X_0$ (resp. $X_1$) consists of those faces in $X$ that do not have (resp. do have) $v_j$ as a vertex.  We have a bijection $\varphi:X_0 \rightarrow X_1$ where $\varphi(F)$ is the face whose vertices consist of $v_j$ together with the vertices of $F$.  There is a filtration of $Q$ by subcomplexes
\[
\{0\} = Q_{-1} \subset Q_{0} \subset \cdots \subset Q_{n-1} = Q
\]
where 
\[
Q_k = \bigoplus_{\stackrel{F \in X_{0}}{\dim F \leq k}} 
\left( V_{F} \oplus V_{\varphi(F)}\right).
\]
The corresponding spectral sequence vanishes at the $E_1$ page since we have a direct sum of complexes
\[
\frac{Q_k}{Q_{k-1}} = \bigoplus_{\stackrel{F \in X_{0}}{\dim F = k}}
 \mathit{Cone}(\alpha^{\varphi(F)}_F: V_{\varphi(F)}[-1] \rightarrow V_{F})
\]
(where for any face $F =[i_0, \ldots, i_m]$,  $V_F$ is a complex with differential $(-1)^{\dim F+1} a_{i_0}$), and the cones are all acyclic since each $\alpha^{\varphi(F)}_F$ is a quasi-isomorpism.  [Examining (\ref{eq:MapDiff}), since $\dim \varphi(F) = \dim F +1$, each $\alpha^{\varphi(F)}_F$ either has the form $\pm a_{i_0i_1}$ or $\pm \mathit{id}_V$.]

\end{proof}

\begin{corollary} \label{cor:vertex}  Let $\mathcal{D}=(\Delta, V, \{a_F\})$ and $\mathcal{D}'= (\Delta', W, \{b_F\})$ be simplex diagrams with $\Delta \subset \Delta'$ where $\dim \Delta = m$ and $\dim \Delta'=n$, and let $v_i \subset \Delta$ be a vertex, and assume $\mathcal{D}$ and $\mathcal{D}'$ are such that all face maps associated to edges are quasi-isomorphism.  Suppose we have a morphism of simplex diagrams $\varphi:\mathcal{D} \rightarrow \mathcal{D}'$ such that the $\varphi:(V^*, a_{v_i}) \rightarrow (W^*, b_{v_i})$ is a quasi-isomorphism of chain complexes.  Then, $\varphi$ induces a quasi-isomorphisms
\[
\mathit{Map}_m(\mathcal{D}) \stackrel{\cong}{\rightarrow} \mathit{Map}_n(\mathcal{D}').
\] 
\end{corollary}
\begin{proof}  We have a commutative diagram of simplex diagrams
\[
\xymatrix{ 
\mathcal{D} \ar[r]^\varphi &  \mathcal{D}' \\
\mathcal{D}|_{v_i} \ar[r]^\varphi \ar[u] & \mathcal{D}'|_{v_i} \ar[u] 
}. 
\]
By Proposition \ref{prop:vertex}, when $\mathit{Map}$ is applied the left, right and bottom arrows all become quasi-isomorphisms.  Thus, the top arrow must also become a quasi-isomorphism.
\end{proof}

\section{Chain homotopy diagrams and augmentations} \label{sec:CHD}

In \cite{RuSu3} augmentations to $\Z/2$ of the cellular DGA of a Legendrian surface $\Lambda$ with compatible polygonal decomposition $\mathcal{E}$ were shown to be in bijection with certain algebraic structures called chain homotopy diagrams (abbrv. CHDs).  In Section \ref{sec:vspace}, using the language of simplex diagrams, we review the definition of CHD and extend it to allow for a general coefficient field $\mathbb{K}$ and Legendrians (with mild front singularities) of arbitrary dimensions.  In Section \ref{sec:CHDaug}, we establish the bijection between CHDs and augmentations of the simplicial DGA from Section \ref{sec:SimplicialReformulation}.

Throughout this section, $\dim(\Lambda)=n$ does not need to be 2.

\subsection{Vector spaces associated to $\mathcal{E}$}   \label{sec:vspace}
We return now to the setting of Section \ref{sec:SimplicialAlgebra} where $(\Lambda, \mu, \mathcal{E})$ is a closed Legendrian $\Lambda\subset J^1M$ (of arbitrary dimension) 
  equipped with a $\Z$-valued 
 Maslov potential, $\mu$, and a compatible simplicial decomposition, $\mathcal{E}$, of $M$. 
  As before, a cell in $\mathcal{E}$ that is the interior of the simplex with vertices $e_{i_0}, \ldots, e_{i_m}$ having $i_0<\ldots < i_m$ is denoted $e_I =e_{i_0\ldots i_m}$ where $I = \{i_0, \ldots, i_m\}$.  
Given $e_I \in \mathcal{E}$, recall that $\Lambda(e_I) = \{ S^I_i\}$ is the set of sheets of $\Lambda$ above $e_I$, and  by definition  these are the components of $\pi_{x}^{-1}(e_I) \cap (\Lambda \setminus \Lambda_{\mathit{cusp}})$ where $\Lambda_{\mathit{cusp}}$ is the set of cusp points of $\Lambda$.   It is important that the sheets in $\Lambda(e_I)$ are considered as subsets of $\Lambda$, not of $\pi_{xz}(\Lambda)$, eg. a triple point of $\pi_{xz}(\Lambda)$ above a vertex $e_{i_0}$ corresponds to three sheets in $\Lambda(e_{i_0})$.

To each cell $e_I \in \mathcal{E}$, we associate a graded vector space 
\[
V(e_I) = \mbox{Span}_{\mathbb{K}}\Lambda(e_I) = \mbox{Span}_{\mathbb{K}} \{S^I_i\}. 
\]
spanned by the sheets of $\Lambda$ above $e_I$.  The $\Z$-grading on $V(e_I)$ is given by
\[
|S_i^I| = -\mu(S_i^I).
\]
As in Property \ref{ob:closure}, if $e_I \leq e_J$, then each sheet in $\Lambda(e_I)$ belongs to the closure of a unique sheet in $\Lambda(e_J)$. This gives an inclusion
\begin{equation} \label{eq:iotaDef}
\iota=\iota(e_I \rightarrow e_J): \Lambda(e_I) \rightarrow \Lambda(e_J), \quad  \iota(S^{I}_{i}) = S^{J}_{j} \quad \Leftrightarrow \quad S^I_i \subset \overline{S^{J}_j},
\end{equation}
and leads to a direct sum decomposition
\begin{equation} \label{eq:directsum}
V(e_J) = \iota\big(V(e_I)\big) \oplus V^{I,J}_{\mathit{cusp}}
\end{equation}
where $V^{I,J}_{\mathit{cusp}}$ is the (possibly empty) span of those sheets in $\Lambda(e_J)$ that meet in pairs at cusp components above $e_I$.  Denote the projection to the first component as
\begin{equation} \label{eq:projection}
p= p(e_J \rightarrow e_I): V(e_J) \rightarrow V(e_I).
\end{equation}

Each of the collections of sheets  $\Lambda(e_I)$ is partially ordered by comparing $z$-coordinates (pointwise) in descending order.  More precisely, we declare
\[
  S^I_i \prec S^I_j \quad \mbox{iff}  \quad  z(S_i) > z(S_j).
\] 
As in Property \ref{ob:crossing}, %only 
sheets that are not comparable must cross one another in $\pi_{xz}(\Lambda)$ so that their front projections coincide above $e_I$.
 A linear map $c:V(e_I) \rightarrow V(e_I)$
  is {\bf lower triangular} (resp. {\bf strictly lower triangular}) if 
\[
\forall\, S^I_i \in \Lambda(e_I), \quad
c(S^I_i) \in \mbox{Span}\{S^I_j\,|\, S^I_i \preceq S^I_j \},  \quad \left(\mbox{resp. } c(S^I_i) \in \mbox{Span}\{S^I_j\,|\, S^I_i \prec S^I_j \} \right).
\]

\subsubsection{Boundary extensions}  
Suppose that $I \subset J$ so that $\overline{e_I}$ is a face of $\overline{e_J}$, and that we are given a map $c(e_I):V(e_I) \rightarrow V(e_I).$
We define the {\bf boundary extension} of $c_I$ from $V(e_I)$ to $V(e_J)$ to be the  
 linear map $c_J(e_I):V(e_J) \rightarrow V(e_J)$ obtained by using (\ref{eq:directsum}) to view $V(e_J) \cong V(e_I) \oplus V^{I,J}_{\mathit{cusp}}$ and putting 
\begin{equation} \label{eq:boundary}
c_J(e_I)\big\vert_{V(e_I)} = c(e_I), \quad c_J(e_I)\big\vert_{V^{I,J}_{\mathit{cusp}}} = \left\{ \begin{array}{cr} d_0, & \mbox{if} \,\, \dim e_I = 0, \\ \mathit{id}, &  \mbox{if}\,\, \dim e_I = 1,  \\
0, & \mbox{if}\,\, \dim e_I \geq 2, \end{array} \right.
\end{equation}
where the map $d_0:V^{I,J}_{\mathit{cusp}} \rightarrow V^{I,J}_{\mathit{cusp}}$ is defined to have $d_0S_a = S_b$, and $d_0 S_b=0$  whenever $S_a,S_b \in \Lambda(e_J)$ meet one another at a cusp edge above $e_I$ and satisfy $S_a \prec S_b$ so that  $S_a$ (resp. $S_b$) is the upper (resp. lower) sheet of the cusp edge.

\begin{lemma} \label{lem:boundary}
When $e_F \leq e_{I} \leq e_{J}$ and $c(e_F):V(e_F) \rightarrow V(e_F)$, the boundary extensions satisfy  
\[
c_I(e_F) \circ p(e_{J} \rightarrow e_{I}) = p(e_{J} \rightarrow e_{I}) \circ c_J(e_F).
\]
\end{lemma}
\begin{proof}
This is immediate from the definitions.
\end{proof}

\begin{remark} \label{remark:boundary}
\begin{enumerate}
\item  When $e_I = e_J$, we have $c_I(e_I) = c(e_I)$.
\item The boundary extension of a strictly lower triangular map with $\dim e_I \neq 1$ is strictly lower triangular, since it is the direct sum of two strictly lower triangular maps.  The boundary extension of an invertible lower triangular map with $\dim e_I = 1$ is  invertible lower triangular. 
\end{enumerate}
\end{remark}

\subsection{Chain Homotopy Diagrams and augmentations of the simplicial DGA}  \label{sec:CHDaug}

\begin{definition} \label{def:CHD}
A {\bf Chain Homotopy Diagram (CHD)} for $(\Lambda, \mu, \mathcal{E})$, denoted $\mathcal{C} = \{c(e_I)\}$,  is an assignment of a map $c(e_I):V(e_I) \rightarrow V(e_I)$ for each simplex $e_I \in \mathcal{E}$ satisfying:

\begin{itemize} 
\item Each $c(e_I)$ has degree $1-\dim e_I$.
\item For each $e_I$ with $\dim e_I \neq 1$ (resp. with $\dim e_I =1$)  $c(e_I)$ is strictly lower triangular (resp. $c(e_I) - \mathit{id}$ is strictly lower triangular). 

\item For each $e_I$, the collection 
\[
\left(\overline{e_I}, V(e_I), \{c_I(e_F)\}_{F \subset I} \right)
\]
is a simplex diagram (where $c_I(e_F)$ denotes the boundary extension of $c(e_F)$ from $V(e_F)$ to $V(e_I)$).
\end{itemize}
\end{definition}

\begin{remark}
Note that when $\dim e_F =1$, $c(e_F)$ is invertible since it is the sum of a strictly lower triangular map and the identity map.  As in Remark \ref{remark:boundary}, all of the boundary extensions, $c_I(e_F)$, will also be invertible lower triangular.  In particular, Proposition \ref{prop:vertex} and Corollary \ref{cor:vertex} apply to the simplex diagrams $\left(\overline{e_I}, V(e_I), \{c_I(e_F)\}_{F \subset I} \right)$.
\end{remark}

\begin{example}
In view of Example \ref{ex:2},
 when $\Lambda$ is $2$-dimensional a CHD amounts to an assignment of  
\begin{itemize}
\item for each $0$-cell a strictly lower triangular differential $d_{i_0}: V(e_{i_0}) \rightarrow V(e_{i_0})$;
\item for each $1$-cell a chain isomorphism $f_{i_0i_1}:\left(V(e_{i_0i_1}), \widehat{d}_{0}\right) \rightarrow \left(V(e_{i_0i_1}), \widehat{d}_1\right)$ such that $f_{i_0i_1}-\mathit{id}$ is strictly lower triangular, where $\widehat{d}_0$ and $\widehat{d}_1$  are the boundary extensions of the differentials from the   initial and terminal vertices $e_{i_0}$ and $e_{i_1}$ of $e_{i_0i_1}$;
\item for each $2$-cell a strictly lower triangular homotopy operator $K=K_{i_0i_1i_2}:V(e_{i_0i_1i_2}) \rightarrow V(e_{i_0i_1i_2})$ satisfying
\[
\widehat{f}_{12} \circ \widehat{f}_{01}- \widehat{f}_{02} = \widehat{d}_{2}\circ K + K \circ \widehat{d}_{0}
\] 
where the $\widehat{f}_{ij}$ and $\widehat{d}_{i}$ are boundary extensions of the chain maps and differentials associated to the edges and vertices of $e_{i_0i_1i_2}$ as indicated in Figure \ref{fig:MapDiff}.
\end{itemize}
In fact, in \cite{RuSu3}  CHDs for Legendrian surfaces are defined in this manner as triples $\mathcal{C} = (\{d_\alpha\}, \{f_\beta\}, \{K_\gamma\})$ consisting of collections of differentials, chain maps, and chain homotopies associated to the vertices, edges, and $2$-dimensional faces of a compatible polygonal decomposition for $\Lambda$. 
\end{example}

\begin{remark}  In \cite{RuSu3}, the maps that constitute a CHD are required to be upper triangular rather than lower triangular.  In addition, the grading on the $V(e_I)$ is given by $\mu(S^I_i)$ on generators rather than $-\mu(S^I_i)$ as defined here.  Thus, the two definitions are equivalent after taking duals.  
\end{remark}

\begin{definition} An
 {\bf augmentation} of a unital DGA $(\mathcal{A}, \partial)$ to a field $\mathbb{K}$ is a unital DGA homomorphism
\[
\epsilon:(\mathcal{A}, \partial) \rightarrow (\mathbb{K},0)
\]
where $\mathbb{K}$ is concentrated in degree $0$.  
\end{definition}

The connection between CHDs and augmentations of the simplicial DGA is through the following result which generalizes Proposition 3.7 from \cite{RuSu3}.

\begin{proposition} \label{prop:CHDbiject}
For any $(\Lambda, \mu, \mathcal{E})$, there is a bijection between augmentations from the simplicial DGA to $\mathbb{K}$ and 
 CHDs for $(\Lambda, \mathcal{E})$ over $\mathbb{K}$.  
\end{proposition}

\begin{proof}
To an augmentation $\epsilon: (\mathcal{A}, \partial) \rightarrow \mathbb{K}$ we associate a CHD $\mathcal{C}_\epsilon = \{c(e_I)\}$ by defining $c(e_I): V(e_I) \rightarrow V(e_I)$ so that with respect to the ordered basis $\{S_1^I, \ldots, S^I_{N_I}\}$ for $V(e_I)$ we have:
\begin{equation} \label{eq:CHDbiject}
\mbox{Matrix of $c(e_I)$} \quad = \quad \left[\epsilon \circ M_I \circ \varphi^{-1}(e_I)\right]^{T}
\end{equation}
where $M_I: \Omega C_*(\overline{e_I}) \rightarrow \mathit{Mat}(n_I, \mathcal{A})$ and $\varphi: \Omega C_*(\overline{e_I}) \rightarrow \Omega C_*(\overline{e_I})$ are the algebra homomorphisms from Section \ref{sec:Cobar}.  (See (\ref{eq:Dmatrix2}) and (\ref{eq:varphimap}).)  In particular, $\varphi^{-1}(e_F) = e_F +1$ if $\dim(e_F) =1$ and $\varphi^{-1}(e_F) = e_F$ if $\dim(e_F) \neq 1$.  Note that, because the only non-zero entries of $M_I(e_I)$ are the generators $m^I_{i,j}$ with $S^I_i \prec S^I_j$, it is the case that $c(e_I)$ is strictly lower triangular when $\dim e_I \neq 1$ and is the sum of a strictly lower triangular map with the identity map when $\dim e_I = 1$.  Moreover, using that $\epsilon$ vanishes on generators of non-zero degree, it can be checked that $\deg c(e_I) = 1 - \dim e_I$.  
 The construction can be reversed, so (\ref{eq:CHDbiject}) gives a bijection between graded algebra homomorphisms $\epsilon: \mathcal{A} \rightarrow \mathbb{K}$ and collections of maps $\{c(e_I)\}$ satisfying the first two requirements of Definition \ref{def:CHD}.  To complete the proof we show that the augmentation equation $\epsilon \circ \partial = 0$ holds if and only if each $(\overline{e_I}, V(e_I), \{c_I(e_F)\}_{F \subset I})$ is a simplex diagram.  
 
As a preliminary, note that a comparison of the constructions of the matrices $M_I(e_F)$ and the boundary extensions $c_I(e_F)$ from $M_F(e_F)$ and $c(e_F)$ respectively shows that 
\begin{equation} \label{eq:eFeI}
\forall \, e_F \subset \overline{e_I}, \quad  \mbox{Matrix of $c_I(e_F)$} \quad = \quad \left[\epsilon \circ M_I \circ \varphi^{-1}(e_F)\right]^{T}.
\end{equation}
We will make use of the alternate characterization of simplex diagrams from Remark \ref{rem:Simp2} which states that $(\overline{e_I}, V(e_I), \{c_I(e_F)\}_{F \subset I})$
will be a simplex diagram if and only if 
\[
\alpha_I \circ D^{\mathit{red}}_s = 0
\]
where $\alpha_I: \Omega C_*(\overline{e_I}) \rightarrow \mathit{End}( V(e_I)^\vee)$ is the graded algebra homomorphism satisfying $\alpha_I(e_F) = c_I(e_F)^\vee$.  Using the basis dual to $\{S^I_i\}$ to identify $\mathit{End}( V(e_I)^\vee) \cong \mathit{Mat}(n, \mathbb{K})$ and applying (\ref{eq:eFeI}), for any $e_F \subset \overline{e_I}$ we have
\[
 \alpha_I(e_F) = \mbox{Matrix of } c_I(e_F)^{\vee}  
 = \left(\mbox{Matrix of } c_I(e_F) \right)^{T} 
  = \epsilon \circ M_I \circ \varphi^{-1}(e_F).
\]
Thus, $\alpha_I = \epsilon \circ M_I \circ \varphi^{-1}$.  Making use of the identities (\ref{eq:Dreduced}) and (\ref{eq:Claim1})
 we compute
\begin{equation}
\alpha_I \circ D^{\mathit{red}}_s  = \epsilon \circ M_I \circ D_s \circ \varphi^{-1}  \label{eq:alpha0} = \Theta_I \cdot (\epsilon \circ \partial  \circ M_I  \circ \varphi^{-1})  
\end{equation}
where we used that $\epsilon (\Theta_I) = \Theta_I$ holds since $\epsilon( \pm 1) = \pm 1$.  

Assuming that $\epsilon \circ \partial =0$, the identity (\ref{eq:alpha0}) shows that $\alpha_I \circ D^{\mathit{red}}_s =0$ so that each $(\overline{e_I}, V(e_I), \{c_{I}(e_F)\})$ is indeed a simplex diagram.  Conversely, if each $(\overline{e_I}, V(e_I), \{c_{I}(e_F)\})$ is a simplex diagram then $\alpha_I \circ D^\mathit{red}_s=0$, so applying both sides of (\ref{eq:alpha0}) to  $\varphi(e_I)$ and multiplying on the left by $\Theta_I$ shows that 
\[
\epsilon \circ \partial ( M_I(e_I)) = 0.
\]
As every generator of $\mathcal{A}$ appears as an entry of $M_I(e_I)$ for some $e_I$, we see that $\epsilon \circ \partial = 0$ as required.

\end{proof}

\section{Construction of sheaves from CHDs}
\label{sec:SheafConstruction}

In this section, we will associate a combinatorial sheaf
 $F(\mathcal{C}) \in \mathbf{Fun}^\bullet_{\Lambda}(\mathcal{S}, \mathbb{K})$ to a CHD $\mathcal{C}$ of a Legendrian surface $\Lambda$ with compatible simplicial decomposition $\mathcal{E}$.  Together with the bijection between augmentations and CHDs from Proposition \ref{prop:CHDbiject} and the equivalence of categories  from Corollary \ref{cor:Equiv},  $\mathbf{Fun}^\bullet_{\Lambda}(\mathcal{S}, \mathbb{K}) \cong \ShLambda$, this completes the construction of sheaves from augmentations as stated in Theorem \ref{thm:main}.   In outline, $F(\mathcal{C})$ is defined as follows:  Using the stratification, $\mathcal{S}$, of $M \times \R$ %associated to $\mathcal{E}$ 
 from Section \ref{sec:handle}, 
 we begin in Section \ref{sec:PreliminaryModules} by associating to each stratum $s \in \mathcal{S}$ a $\mathbb{K}$-module $X(s)= Y(s) \oplus Z(s),$ where $Y(s)$ is generated by the sheets below $s,$ and $Z(s)$ is $\{0\}$ except in certain exceptional cases when $s$ is near a cusp edge.
 In Section \ref{sec:Maps}, $X$ is extended to a functor from  $P(\mathcal{S})$, the poset category of $\mathcal{S}$, to graded $\mathbb{K}$-modules.
In Section \ref{sec:MapG} we use the chain homotopy diagram $\mathcal{C}$ to upgrade $X$ to a functor $G(\mathcal{C})$ from $P(\mathcal{S})$ to the category of simplex diagrams $\mathfrak{sd}$ defined in Section \ref{sec:SimplexDiagrams}.  Finally, we compose $G(\mathcal{C})$ with the mapping cylinder functor $Map$ from Proposition \ref{prop:MappingFun} to produce the required functor $F(\mathcal{C})$ from $P(\mathcal{S})$ to $\mathit{Ch}(\mathbb{K}\mbox{-mod})$ with  Proposition \ref{prop:FisFun} establishing that $F(\mathcal{C})$ indeed belongs to $\mathbf{Fun}_{\Lambda}^\bullet(\mathcal{S}, \mathbb{K}).$

Throughout this section, we assume that $\dim(\Lambda) = 2.$

\subsection{Preliminary modules assigned to strata of $\mathcal{S}$}
\label{sec:PreliminaryModules}

Recall the constructions of Section \ref{sec:handle}.  Starting from a simplicial decomposition $\mathcal{E}$ of $M$ that is compatible with a Legendrian surface $\Lambda \subset J^1M$ with mild front singularities, we formed a handle decomposition $H=\{h(\bfe)\}$ which is a polygonal decomposition of $M$ whose cells $h(\bfe)$ are labeled by ordered pairs $\bfe = (e_I, e_J)$ of simplices $e_I, e_J \in \mathcal{E}$ with $e_I \leq e_J$.  
 See Figures \ref{fig:PairSub} and \ref{fig:HLabel}.
Next, the stratification $\mathcal{S}$ of $M\times \R$ was formed 
 having as strata the connected components of the intersections $\left(h(\bfe) \times \R \right)\cap  \Lambda^F_k$ where $M \times \R = \sqcup_{k=-1}^2 \Lambda^F_{k} $ is the $\Lambda^F$-stratification of $M \times \R$ from (\ref{eq:xzstrat}).

\begin{notation} We write $\mathcal{S}(\mathbf{e}) \subset \mathcal{S}$ for those strata $s \in \mathcal{S}$ with $\pi_x(s) \subset h(\mathbf{e})$.  
\end{notation}

Recall that for $e_I \in \mathcal{E}$, $\Lambda(e_I)$ denotes the set of sheets of $\Lambda$ above $e_I$, and by definition these are the components of $\pi_x^{-1}(e_I) \cap \Lambda$ that are disjoint from $\Lambda_{\mathit{cusp}}$.  We extend this terminology and notation to cells of $H$ as follows.   Given $h(\mathbf{e}) \in H$, let $\Lambda(\mathbf{e})$ denote the set of {\bf non-cusp sheets of $\Lambda$ above $h(\mathbf{e})$} which we define to be the set of those components of $\Lambda \cap \pi_x^{-1}(h(\mathbf{e}))$ that are disjoint from $\Lambda_{\mathit{cusp}}$.  Note that sheets in $\Lambda(\mathbf{e})$ project diffeomorphically to $h(\mathbf{e})$ under $\pi_x$.  Moreover, for $\bfe = (e_I, e_J)$, there is a bijection
\begin{equation} \label{eq:eJ}
\Lambda(\bfe) \cong \Lambda(e_J)
\end{equation}
where a sheet $T$ above $\bfe$ is identified with the unique sheet $S$ above $e_J$ for which $T \cap S \neq \emptyset$.  

Next, suppose that $C \subset \Lambda \cap \pi_{x}^{-1}(h(\mathbf{e}))$ is a {\bf cusp component} of $\Lambda$ above $h(\mathbf{e})$, i.e. $C$ {\it does} intersect $\Lambda_{\mathit{cusp}}$.  Such a component can be written in the form $C = C_0 \sqcup C_u \sqcup C_l $ where $C_0 = C \cap \Lambda_{\mathit{cusp}}$ and under the front projection $C_u$ and $C_l$ map respectively to the upper and lower branches of $\pi_{xz}(\Lambda)$ that meet at the cusp edge $\pi_{xz}(C_{0})$.  Notice that if $\dim h(\mathbf{e}) = 0$, then there are no cusp components above $h(\mathbf{e})$.  If $\dim h(\mathbf{e}) = 1$ or $2$, then $\pi_x(C_u) = \pi_x(C_l)$ is a proper subset of $h(\mathbf{e})$.

\begin{definition}
We say that $s \in \mathcal{S}(\bfe)$ is an {\bf exceptional stratum} for a cusp component $C$ of $\Lambda$ above $h(\mathbf{e})$ when, at any $x$ coordinate belonging to $s,$ $s$ lies strictly above $\pi_{xz}(C_l)$ and weakly below $\pi_{xz}(C_u)$.  That is, for any points of the form $(x,z_l) \in \pi_{xz}(C_l)$, $(x,z) \in s$, and $(x,z_u) \in \pi_{xz}(C_u)$, we have $z_l < z \leq z_u$.  
\end{definition}

For example, above a $1$-cell that is bisected by the base projection of a cusp edge, there are two exceptional strata as indicated in Figure \ref{fig:cusp}. 

\begin{figure}
\labellist
\small
\pinlabel $h(\bfe')$ [t] at 64 2
\pinlabel $h(\bfe)$ [t] at 124 -2
\pinlabel $T_1$ [l] at 128 206
\pinlabel $T_2$ [l] at 128 168
\pinlabel $T_3$ [l] at 128 124
\pinlabel $T_4$ [l] at 128 100
\pinlabel $T_5$ [l] at 128 62
\pinlabel $T'_1$ [b] at 28 206
\pinlabel $T'_2$ [b] at 28 102
\pinlabel $T'_3$ [b] at 28 64
\pinlabel $s'$ [l] at 158 142
\pinlabel $s$ [b] at 150 196

\endlabellist
\centerline{\includegraphics[scale=1]{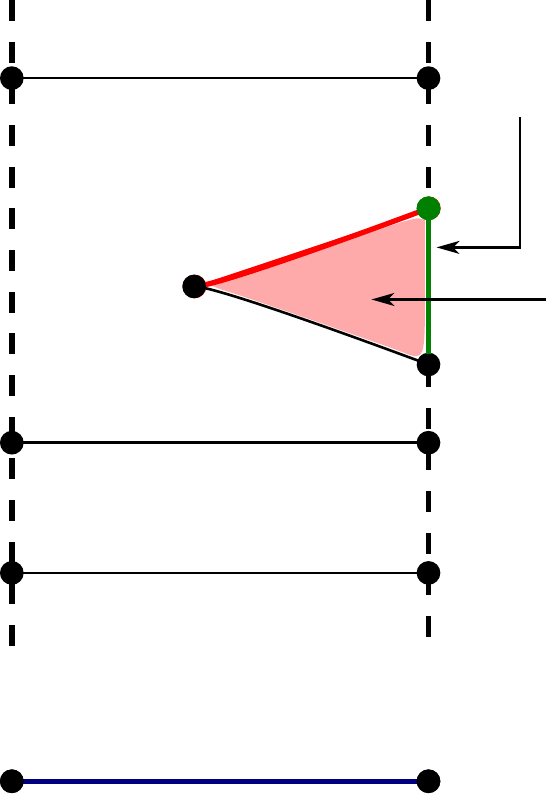}}

\quad

\caption{ An example of a $1$-cell $h(\bfe')$ and a $0$-cell $h(\bfe)$ where $\Lambda$ has sheets labeled as $\Lambda(\bfe') = \{T_1', T_2', T_3'\}$ and $\Lambda(\bfe) = \{T_i\,|\, 1 \leq i \leq 5\}$.  Above $h(\bfe')$, $\Lambda$ has a single cusp component, $C= C_0 \sqcup C_u\sqcup C_l$, for which there are two exceptional strata in $\mathcal{S}(\bfe')$ that appear shaded in red.  (They are the $2$-dimensional stratum $s'$ and the $1$-stratum  that appears at its upper boundary.)  For the labeled strata $s'$ and $s$ the preliminary modules are $X(s') = Y(s') \oplus Z(s') = \mathbb{K}\{T_2',T_3'\} \oplus \mathbb{K}\{v_C\}$ and $X(s) = Y(s) \oplus Z(s) = \mathbb{K}\{T_3, T_4, T_5\} \oplus \{0\}$.  The sheets $T_3,T_4,T_5$ belong respectively to $\overline{C_l}, \overline{T'_2}, \overline{T'_3}$, so the preliminary generization map is $X(s \rightarrow s'):  \quad T_3 \mapsto v_C, \quad T_4 \mapsto T_2', \quad T_5 \mapsto T_3'$.  For the strata pictured in Figure \ref{fig:cusp}, the map $k$ from Definition \ref{def:maps} is only non-zero on generization maps from a green stratum  to a red stratum.
} 

\label{fig:cusp}
\end{figure}
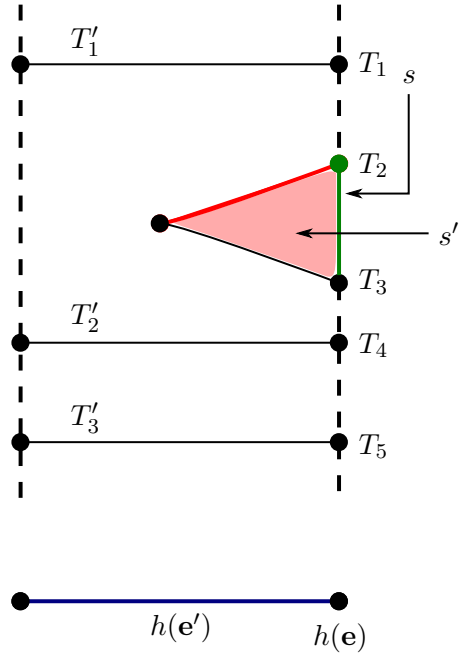

\begin{definition} Given a stratum $s \in \mathcal{S}(\mathbf{e})$ define the {\bf preliminary module}
\[
X(s) = Y(s) \oplus Z(s)
\]
where $Y(s)$ is the span of those (non-cusp) sheets in $\Lambda(\mathbf{e})$ that are strictly below $s$ (at any $x \in \pi_x(s)$), i.e.
\[
Y(s) = \mbox{Span}_\mathbb{K} \left\{ T \in \Lambda(\bfe) \,|\, z_T < z_s \, \mbox{whenever}\, (x,z_T) \in \pi_{xz}(T),\, (x,z_s) \in s \right\} 
\]
and
\[
Z(s) = \mbox{Span}_\mathbb{K}\left\{ v_C \,\middle|\, \begin{array}{l} \mbox{$s$ is an exceptional stratum} \\ \mbox{for the cusp component $C$}, \end{array}\right\}.
\]
A grading on $X(s)$ is defined using the Maslov potential $\mu$ by requiring that sheets $T \in Y(s)$ have $|T| = -\mu(T)$, and that the $v_C \in Z(s)$ satisfy $|v_C| = -\mu(C_l)$.  
\end{definition}

\begin{observation}
\begin{enumerate}
\item Notice that the bijection (\ref{eq:eJ}) allows us to view $Y(s) \subset V(e_J)$.   
\item  
 The subspace $Z(s)$ has $\dim Z(s) \leq 1$.  [A single strata is never exceptional for more than $1$ cusp component.] 
\item
Unlike $Y(s)$ which contains all the non-cusp sheets below $s$,  
 the generators of $Z(s)$ only correspond to those cusp components for which $s$ is between $\pi_x(C_l)$ and $\pi_x(C_u)$.  Intuitively, once both the lower and upper branches of a cusp appear below $s,$ they can be viewed as  canceling one another from appearing in $X(s)$.
 
 \end{enumerate}
\end{observation}

Examples of the preliminary modules $X(s)$ appear in Figures \ref{fig:cusp} and \ref{fig:CrossingEx}.

\begin{figure}
\labellist
\small
\pinlabel $h(\bfe)$ [l] at 248 56
\pinlabel $T_1$  at 74 340
\pinlabel $T_2$  at 74 268
\pinlabel $T_3$  at 74 202
\pinlabel $T_4$  at 74 140
\pinlabel $\{0\}$ [r] at 0 100
\pinlabel $\mathbb{K}\{T_4\}$ [l] at 240 176
\pinlabel $\mathbb{K}\{T_3,T_4\}$ [r] at 0 228
\pinlabel $\mathbb{K}\{T_2,T_4\}$ [l] at 240 236
\pinlabel $\mathbb{K}\{T_2,T_3,T_4\}$ [l] at 240 304
\pinlabel $\mathbb{K}\{T_1,T_2,T_3,T_4\}$ [r] at 0 378
\endlabellist
\centerline{\includegraphics[scale=.8]{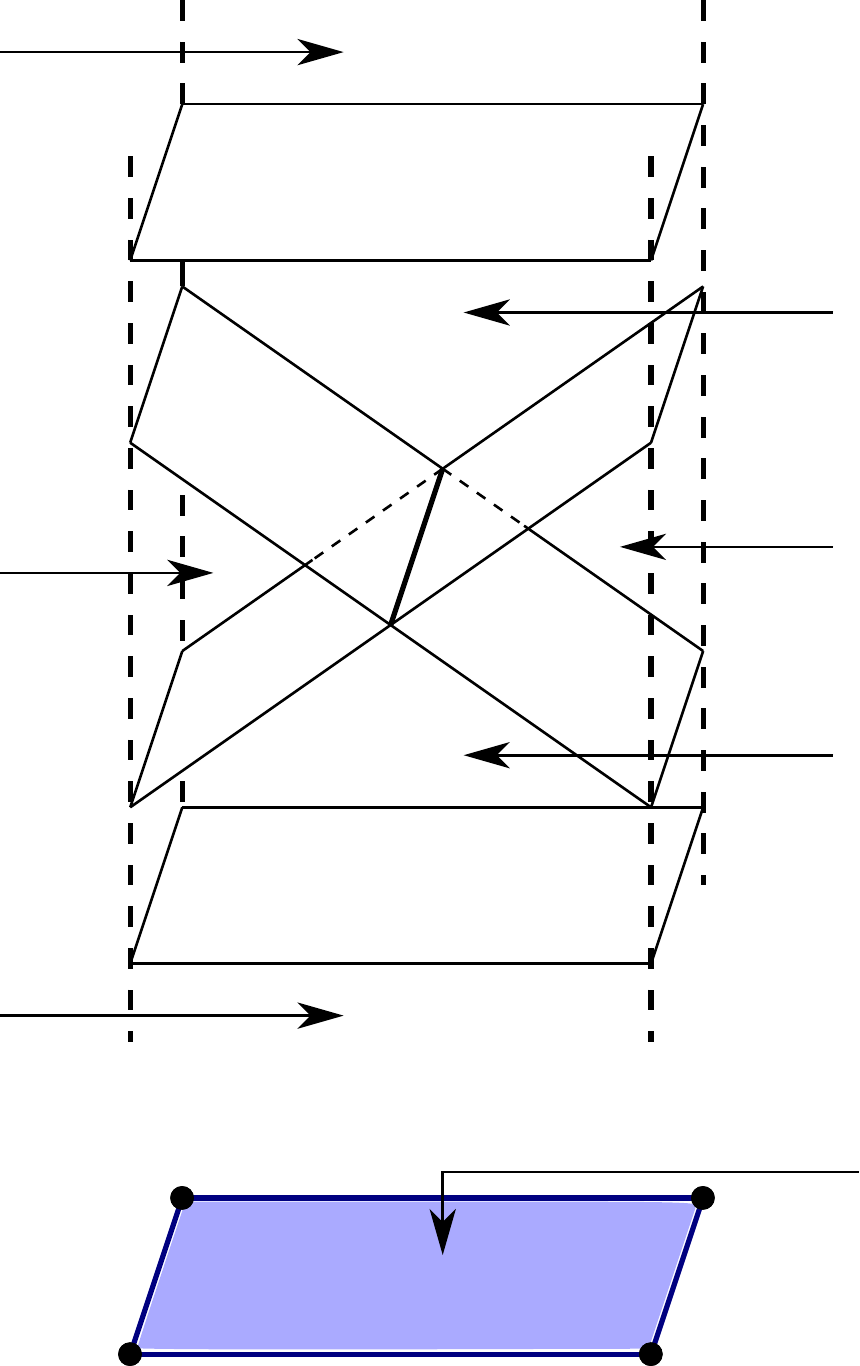}}

\quad

\caption{An example of a $2$-cell, $h(\bfe)$, where $\Lambda(\bfe)$ has $4$ sheets, $T_1, T_2, T_3, T_4$, such that the images of $T_2$ and $T_3$ intersect along a crossing arc in the front projection.   For each of the 3-dimensional strata in $\mathcal{S}(\bfe)$ the vector space $X(s) = Y(s)$ is indicated.  (In all cases $Z(s) = 0$.)  }
\label{fig:CrossingEx}
\end{figure}
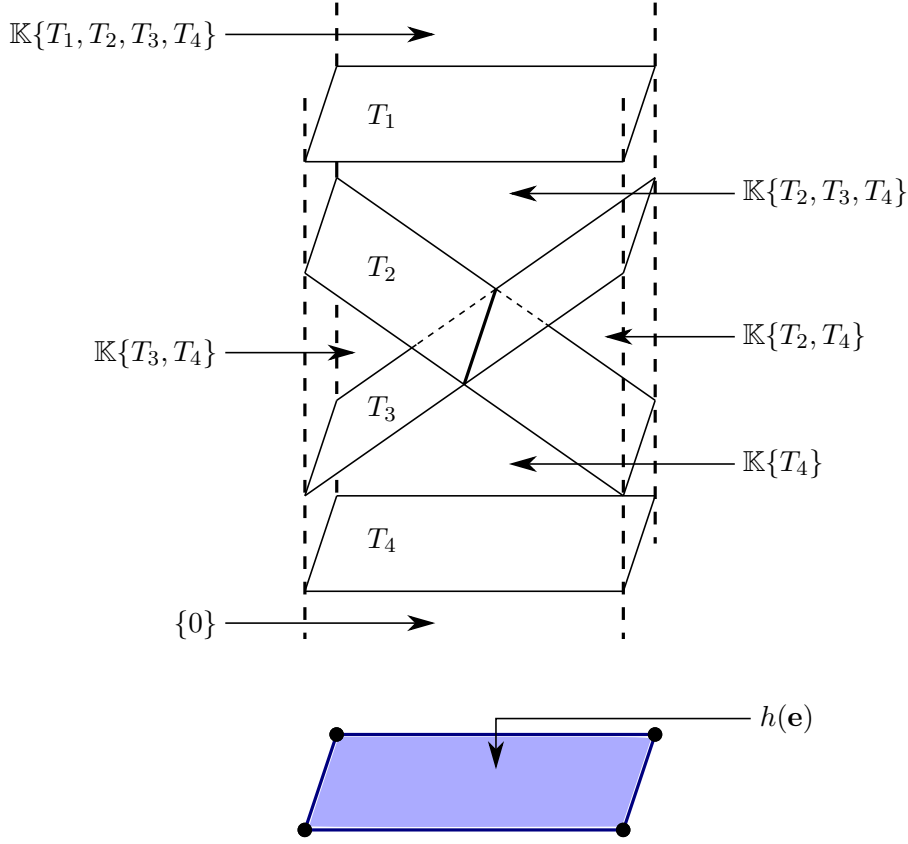

\subsection{Preliminary generization maps}  \label{sec:Maps}
 Let $s \in \mathcal{S}(\bfe)$, $s' \in \mathcal{S}(\bfe')$ have $s \leq s'$, i.e. $s \subset \overline{s'}$, where $\bfe= (e_I, e_J)$   and $\bfe' =(e_{I'}, e_{J'})$.   
 Note that $s \leq s'$ implies $h(\bfe) \leq h(\bfe')$, and in general 
\[
h(\bfe) \leq h(\bfe') \quad \Leftrightarrow \quad e_I \leq e_{I'} \mbox{ and } e_J \geq e_{J'}.
\]
(See Figure \ref{fig:HLabel}.)
\begin{definition} \label{def:maps} The {\bf preliminary generization map} $X(s \rightarrow s'):X(s) \rightarrow X(s')$ is
\[
X(s \rightarrow s') = (p, k+ \ell): Y(s) \oplus Z(s) \rightarrow Y(s') \oplus Z(s')
\] 
with component maps $p:Y(s) \rightarrow Y(s')$,  $k:Y(s) \rightarrow Z(s')$ and $\ell:Z(s) \rightarrow Z(s')$ defined as follows:
\begin{itemize}
\item The projection map $p(e_J \rightarrow e_{J'}): V(e_J) \rightarrow V(e_{J'})$  from (\ref{eq:projection}) leads to a projection map $p: Y(s) \rightarrow Y(s')$ (where sheets  in $\Lambda(\bfe)$ that belong to the closure of a cusp component above $h(\bfe')$ are mapped to 0).  

[To see $p(e_J \rightarrow e_{J'})(Y(s)) \subset Y(s')$, suppose that $T \in \Lambda(\bfe)$ has  $T \in Y(s) $ and $p(e_J\rightarrow e_{J'})(T) = T' \in \Lambda(\bfe')$.  This means that $T \subset \overline{T'}$.  Suppose for contradiction that $T' \notin Y(s')$.  Then, there exists $(x',z') \in s'$ and $(x',z_T') \in T'$ such that $z' \leq z_T'$.  Since $T'$, viewed as a union of strata in $\mathcal{S}(\bfe')$, separates $h(\bfe') \times \R_z$ into two pieces we see that $z' \leq z_T'$ actually holds for all $(x',z') \in s'$ and $(x',z_T') \in T'$.  Taking sequences of the form $(x_k',z_k') \in s'$ and $(x_k', (z_T')_k) \in T'$ converging to $(x,z) \in s$
 and $(x,z_T) \in T$, we see that $z \leq z_T$.  This contradicts $T \in Y(s)$.]   
\item The map $k:Y(s) \rightarrow Z(s')$ is defined so that for each cusp component $C = C_0 \sqcup C_{u} \sqcup C_l$ above $h(\bfe')$ such that $v_C \in Z(s')$, we have $k(T) = v_C$ if $T$ belongs to the closure of the lower cusp sheet, $\overline{C_{l}}$.  Moreover, $k$  is $0$ on all other generators.  Put another way, 
\[
k(T) = \sum_C \delta(T,C,s') v_C
\]
where $\delta(T,C,s')=1$ if $T \subset \overline{C_l}$ and $s'$ is an exceptional stratum for $C$ and $\delta(T,C,s')=0$ otherwise.
\item Any cusp component $C$ above $h(\bfe)$ belongs to the closure of a  unique cusp component $C'$ above $h(\bfe')$, and the map $\ell: Z(s) \rightarrow Z(s')$ has $\ell(v_C) = v_{C'}$ if $s'$ is exceptional for $C'$ and $\ell(v_C) = 0$ otherwise.  That is,
\[
\ell(v_C) = \sum_{C'} \delta(C, C',s') v_{C'}
\]
where $\delta(C,C',s') = 1$ if $C \subset \overline{C'}$ and $s'$ is an exceptional stratum for $C'$ and $\delta(C,C',s') = 0$ otherwise. 
\end{itemize}
\end{definition}

A sample computation of a preliminary generization map, $X(s \rightarrow s')$, appears in Figure \ref{fig:cusp}.  

\begin{proposition}
\label{prop:XFunctor}
The assignment 
\[
s \in \mathcal{S} \mapsto X(s), \quad \left( s \rightarrow s'\right) \mapsto \left(X(s\rightarrow s'): X(s) \rightarrow X(s')\right)
\]
defines a functor $X: P(\mathcal{S})\rightarrow \mathbb{K}-\mathit{mod}^{Gr}$ (from the poset category of $\mathcal{S}$ to the category of graded $\mathbb{K}$-modules.)
\end{proposition}

\begin{proof}   That $X(s \rightarrow s) = \mathit{id}_{X(s)}$ is clear from the definition.  
Consider strata $s \leq s' \leq s''$ above $h(\bfe) \leq h(\bfe') \leq h(\bfe'')$ with $\bfe = (e_I,e_J)$, $\bfe'=(e_{I'},e_{J'})$, $\bfe''=(e_{I''},e_{J''})$ so that $e_I \leq e_{I'} \leq e_{I''}$ and $e_J \geq e_{J'} \geq e_{J''}$.  To check that $X(s \rightarrow s'') = X(s' \rightarrow s'') \circ X(s\rightarrow s')$, we verify that the component functions satisfy 
\begin{align}
p(s \rightarrow s'') &= p(s' \rightarrow s'') \circ p(s \rightarrow s'),  \label{eq:p} \\
  k(s \rightarrow s'') & = k(s' \rightarrow s'') \circ p(s \rightarrow s') + \ell(s' \rightarrow s'') \circ k(s \rightarrow s'), \quad \quad \mbox{and}  \label{eq:k} \\
 \ell(s\rightarrow s'') &= \ell(s'\rightarrow s'') \circ \ell(s\rightarrow s'). \label{eq:l}
\end{align}
Note that (with notation as in (\ref{eq:directsum})) $V_{\mathit{cusp}}^{J'',J} = V_{\mathit{cusp}}^{J'', J'} \oplus V_{\mathit{cusp}}^{J', J}$.  Therefore, the projections factor as $p(e_{J} \rightarrow e_{J''}) = p(e_{J'} \rightarrow e_{J''}) \circ p(e_{J} \rightarrow e_{J'})$ and (\ref{eq:p}) follows.

To verify (\ref{eq:k}), let $T \in Y(s)$.  When applied to $T$, both sides of (\ref{eq:k}) are $0$ unless $T\subset \overline{C''_l}$ where $\overline{C''_l}$ denotes the closure of the lower sheet of a cusp component $C''$ above $h(\bfe'')$.  (Otherwise the terms on the RHS become $0$ when the $k$ factor is applied.)  Assuming $T \subset \overline{C''_l}$, the LHS of (\ref{eq:k}) applied to $T$ is 
\[
\delta(T, C'',s'') v_{C''}.
\]
In evaluating the RHS we consider cases.

\medskip

\noindent{\bf Case 1.}  $C' \subset \overline{C''}$ for some cusp component $C'$ above $h(\bfe')$.

\medskip

Then, the RHS of (\ref{eq:k}) becomes
\[
(k \circ p+\ell \circ k)(T) = k(0) + \ell(\delta(T,C',s')v_{C'}) = \delta(C',C'',s'')\delta(T, C',s') v_{C'},  
\]
so we should check that $\delta(T, C'',s'') = \delta(C',C'',s'')\delta(T, C',s')$.
 Since $T \subset \overline{C''_l}$ and $C' \subset \overline{C''}$, we have $\delta(T,C'',s'') = \delta(C',C'',s'')$ and we may assume the common value is $1$ so that $s''$ is exceptional for $C''$.  We then need to show that $\delta(T, C',s')=1$, and  (since $T \subset \overline{C'_l}$ follows as  $T \subset \overline{C''_l} \cap \overline{(h(\bfe') \times \R_z)} = \overline{C'_l}$) this amounts to checking that 
  $s'$ is also exceptional for $C'$. 

Now, because $s''$ is exceptional for $C''$, for any $x \in \pi_x(s'')$   
  and $(x,z''_l) \in C''_l, (x,z'') \in s'', (x,z''_u) \in C''_u$ we have  $z''_l < z'' \leq z''_u$.  Since $s' \subset \overline{s''}$ and $C' \subset \overline{C''}$ it follows that for any $(x,z'_l) \in C'_l, (x,z') \in s', (x,z'_u) \in C'_u$ we have  $z'_l \leq z' \leq z'_u$.  Assume for contradiction that $z'_l = z'$, then this will hold for all $x \in \pi_{x}(s')$, and since $s \subset \overline{s'}$ and $T \subset \overline{C'_l}$ we get that there exist $(x,z) \in s$ and $(x,z_l) \in T$ such that $z_l = z$.  This provides the contradiction $T \notin Y(s)$.  

\medskip

\noindent{\bf Case 2.}  $\overline{C''}$ does not contain the closure of a cusp component above $h(\bfe')$.

\medskip

Then, there is a sheet $T' \in \Lambda(\bfe')$ such that $T \subset \overline{T'} \subset \overline{C''_l}$  
so the RHS of (\ref{eq:k}) becomes
\[
(k \circ p+\ell \circ k)(T) = k(p(T)) = k(T') = \delta(T', C'',s'') v_{C''}.
\]
To verify (\ref{eq:k}), we note that in this case $\delta(T,C'',s'') = \delta(T', C'',s'')$.

 \medskip

Finally, we check (\ref{eq:l}).  For $C \in Z(s)$, we have unique cusp components $C'$ and $C''$ above $h(\bfe')$ and $h(\bfe'')$ such that $C \subset \overline{C'} \subset \overline{C''}$.  For verifying that $\delta(C, C'', s'') = \delta(C,C', s') \delta(C',C'', s'')$, we can assume $s''$ exceptional for $C''$ (otherwise both sides are $0$) and check that $s'$ is also exceptional for $C'$.  If not, then $s' \subset C'_l$ and this would imply that $s \subset C_l$ contradicting that $s$ is an exceptional stratum for $C$.

\end{proof}

The following will be useful later in checking that the combinatorial sheaf $F(\mathcal{C})$ associated to the CHD $\mathcal{C}$ belongs  to $\mathbf{Fun}^\bullet_{\Lambda}(\mathcal{S}, \mathbb{K})$.  
\begin{lemma} \label{lem:d}
If $s \rightarrow s'$ is a downward generization map (as in Definition \ref{def:downward}), then $X(s) = X(s')$ and $X(s \rightarrow s') = \mathit{id}$.
\end{lemma}
\begin{proof}  By definition, there is some $\bfe$ such that $s,s' \in \mathcal{S}(\bfe)$, and $s'$ is adjacent to $s$ from below.  Since the front projection of each sheet $T\in \Lambda(\bfe)$ separates $h(\bfe) \times \R_z$ into two components, each of which are unions of strata from $\mathcal{S}(\bfe)$, we have that $T$ is strictly below $s$ if and only if it is strictly below $s'$.  Thus, we have $Y(s) = Y(s')$. In addition, for a cusp component, $C$, of $\Lambda$ above $h(\bfe)$ it is clear that $s$ is an exceptional stratum for $C$ if and only if $s'$ is, so $Z(s) = Z(s')$.  That $X(s \rightarrow s')= \mathit{id}$ follows since the definitions show that $p = \mathit{id}_{Y(s)}$, $k =0$, and $\ell = \mathit{id}_{Z(s)}$.  
\end{proof}

\subsection{Defining the combinatorial sheaf}  
\label{sec:MapG}

We are now ready to construct from a CHD $\mathcal{C} = \{c(e_I)\}$ for $(\Lambda, \mu, \mathcal{E})$ a combinatorial sheaf $F(\mathcal{C}) \in \mathbf{Fun}^\bullet_{\Lambda}(\mathcal{S}, \mathbb{K})$ of the form $F(\mathcal{C}) = \mathit{Map} \circ G(\mathcal{C})$ where $G(\mathcal{C})$ is a functor  from $P(\mathcal{S})$ to the category of simplex diagrams and $\mathit{Map}$ is the mapping cylinder functor from Proposition \ref{prop:MappingFun}.

We start with the definition of $G(\mathcal{C})$ on objects. Let $s \in \mathcal{S}(\bfe)$ with $\bfe = (e_I, e_J)$ and define a simplex diagram 
\[
G(\mathcal{C})(s) = (\overline{e_I}, X(s), \{a_F\})
\]
with underlying simplex $\Delta = \overline{e_I}$ and graded $\mathbb{K}$-module $X(s)$. 
For $F \subset I \subset J$, 
the map $a_F:X(s) \rightarrow X(s)$ is a direct sum $a_F = \alpha_F \oplus \beta_F$ of maps $\alpha_F:Y(s) \rightarrow Y(s)$ and $\beta_F:Z(s) \rightarrow Z(s)$ defined as follows:
\begin{enumerate}
\item   To define $\alpha_F$ recall that the maps $c(e_F):V(e_F) \rightarrow V(e_F)$ from the CHD $\mathcal{C}$ have boundary extensions $c_J(e_F): V(e_J) \rightarrow V(e_J)$ that are lower triangular with respect to the partial ordering, $\prec$, of sheets of $\Lambda(e_J)$ by descending $z$-coordinate. 
Note that when we view $Y(s) \subset V(e_J)$ (using (\ref{eq:eJ})), if $T \in Y(s)$ has $T \prec T'$, then $T' \in Y(s)$ (since $z(T) > z(T')$ above $e_J$ implies that $z(T) > z(T')$ above $h(\mathbf{e})$).   
 Thus,  each $c_J(e_F)$ restricts to an endomorphism of the subspace  $Y(s) \subset V(e_J)$, and we can define 
\[
\alpha_F = c_J(e_F)|_{Y(s)}.
\]
\item Define
\[
\beta_F = \left\{ \begin{array}{cr} 0, &  \dim e_F \neq 1,  \\ \mathit{id}_{Z(s)}, & \dim e_F =1. \end{array} \right.
\]
\end{enumerate}
\begin{lemma}  As defined $G(\mathcal{C})(s) = (\overline{e_I}, X(s), \{a_F\})$ is a simplex diagram.
\end{lemma}
\begin{proof}
The components $\alpha_F$ and $\beta_F$ individually satisfy the equations from (\ref{eq:simplex}).  For the $\alpha_F$ this follows from Definition \ref{def:CHD} and for $\beta_F$ this is a direct calculation.  [For $\beta_F$, the two sums in (\ref{eq:simplex}) are both zero unless $\dim e_F = 2$, and in the latter case the first sum equals $-\mathit{id}_{Z(s)}$ while the second sum is $\mathit{id}_{Z(s)}$.]
\end{proof}

To complete the definition of $G(\mathcal{C})$, for $s\in \mathcal{S}(\bfe)$ and $s' \in \mathcal{S}(\bfe')$ with $s \leq s'$ we define a map of simplex diagrams $G(\mathcal{C})(s \rightarrow s'): G(\mathcal{C})(s) \rightarrow G(\mathcal{C})(s')$ to arise from the inclusion $\overline{e_I} \subset \overline{e_{I'}}$ together with the  map $X(s \rightarrow s'):X(s) \rightarrow X(s')$. 

\begin{lemma} As defined, $G(\mathcal{C})(s \rightarrow s')$ is a morphism of simplex diagrams.
\end{lemma}
\begin{proof}
Note that since $h(\bfe) \leq h(\bfe')$ implies $e_{I} \leq e_{I'}$, the underlying simplex $\overline{e_I}$ 
 of $G(\mathcal{C})(s)$ is indeed a face of that of $G(\mathcal{C})(s')$ as required in Definition \ref{def:sdmorph}.
 For any $F \subset I$ we need to check 
\[
X(s\rightarrow s') \circ \left(c_J(e_F)|_{Y(s)} \oplus \beta_F\right) \,= \, \left(c_{J'}(e_F)|_{Y(s')} \oplus \beta_F\right) \circ X(s \rightarrow s').
\]
This amounts to verifying that for any  sheet  $T \in \Lambda(e_J)$ with $T \in Y(s)$  and any $v_C \in Z(s)$
\begin{align}
p(e_J \rightarrow e_{J'}) \circ c_J(e_F)(T) & = c_{J'}(e_F) \circ p(e_J \rightarrow e_{J'})(T) \\
k(s \rightarrow s') \circ c_J(e_F)(T) & = \beta_F\circ k(s \rightarrow s')( T)  \label{eq:k2} \\
\ell(s \rightarrow s') \circ \beta_F (v_C) & = \beta_F \circ \ell(s \rightarrow s') (v_C).
\end{align}
The first identity is Lemma \ref{lem:boundary}, while the last identity is trivial since $\beta_F$ is $0$ or $\mathit{id}$.

To prove (\ref{eq:k2}), since $e_{J} \geq e_{J'}$ we can write $V(e_J) = V(e_{J'}) \oplus V^{J',J}_{\mathit{cusp}}$ and consider cases.

\medskip

\noindent {\bf Case 1.}  $T \in V(e_{J'})$.  In this case, both sides of (\ref{eq:k2}) are $0$.  [This is because $k(s \rightarrow s')$    is $0$ on the subspace $V(e_{J'})$ of $V(e_{J})$ and $c_J(e_F)$ preserves this subspace.] 

\medskip

\noindent{\bf Case 2.}  $T \in V^{J',J}_{\mathit{cusp}}$.  If $\dim e_F \geq 1$, then the $V^{J',J}_{\mathit{cusp}}$ component of $c_J(e_F)$  agrees with $\beta_F$ and moreover (being either $0$ of $\mathit{id}$, see (\ref{eq:boundary})) commutes with $k(s \rightarrow s')$ so that the identity follows.  Assuming that $\dim e_F =0$ the RHS of (\ref{eq:k2}) is 0 and the $V^{J',J}_{\mathit{cusp}}$ component of $c_J(e_F)$ is $d_0$ (as in (\ref{eq:boundary})).  Thus, $c_J(e_F)(T) = 0$ unless there is a cusp component $C = C_0 \sqcup C_u \sqcup C_l$ above $h(\bfe')$ such that $T \subset \overline{C_u}$.  In this case, $c_J(e_F)(T)= T_l \subset \overline{C_l}$, 
  but  $k(s \rightarrow s') (T_l) =0$ because  $s'$ cannot be an exceptional stratum for $C$.  [If $s'$ were exceptional for $C$, then for any $(x',z') \in s'$ and $(x',z'_u) \in C_u$ we have $z' \leq z'_u$.  Taking a sequence with $(x'_n,z'_n) \rightarrow (x,z) \in s$ and $(x_n',(z'_u)_n) \rightarrow (x,z_u) \in T$ we see that $z \leq z_u$, contradicting that $T \in Y(s)$.] 

\end{proof} 

\begin{proposition}
The construction $G(\mathcal{C}):P(\mathcal{S}) \rightarrow \mathfrak{sd}$ is a functor from $P(\mathcal{S})$ to the category of simplex diagrams.
\end{proposition}
\begin{proof}
That compositions and identities are preserved follows since $X$ is a functor by Proposition \ref{prop:XFunctor}.
\end{proof}

With $G(\mathcal{C})$ defined we now set the combinatorial sheaf $F(\mathcal{C}) = \mathit{Map} \circ G(\mathcal{C})$.

\begin{proposition}
\label{prop:FisFun}
We have $F(\mathcal{C}) \in \mathbf{Fun}_{\Lambda}^\bullet(\mathcal{S}, \mathbb{K})$.
\end{proposition}

\begin{proof}
We check that $F(\mathcal{C})$ satisfies the conditions (1)-(3) from Definition \ref{def:fun}.

\medskip

\noindent{\bf Checking (1) of Definition \ref{def:fun}:}  Already for $G(\mathcal{C})$ all downward generization maps are the identity.  [This follows from Lemma \ref{lem:d}, and the fact that when $s \rightarrow s'$ is downward $s$ and $s'$ will belong to $\mathcal{S}(\bfe)$ for a common $\bfe$ so that $e_I = e_{I'}$.]

\medskip

\noindent{\bf Checking (2) of Definition \ref{def:fun}:}  For condition (2), suppose $s \rightarrow s'$ is a generization map with $s \in \mathcal{S}(\bfe)$, $s' \in \mathcal{S}(\bfe')$, such that $s, s' \in \Lambda^F_k$ for some $-1 \leq k \leq 2$ (where $\sqcup_{k=-1}^2\Lambda^F_k$ is the $\Lambda^F$-stratification of $M \times \R$; see (\ref{eq:xzstrat}).) 
  By Corollary \ref{cor:vertex}, it suffices to check that for some $i_0 \in I \subset I'$, 
$X(s \rightarrow s')$ induces a quasi-isomorphism from $\left(X(s), d^s_{i_0}\right)$ to $\left(X(s'), d^{s'}_{i_0}\right)$ where $d^s_{i_0}$ and $d^{s'}_{i_0}$ are the maps associated to the vertex $e_{i_0}$ in the simplex diagrams $G(\mathcal{C})(s)$ and $G(\mathcal{C})(s')$.  These maps are obtained from the CHD $\mathcal{C} = \{c(e_I)\}$ as 
\[
d^s_{i_{0}}=c_J(e_{i_0})|_{Y(s)} \oplus 0_{Z(s)}  \quad \mbox{and} \quad d^{s'}_{i_{0}}=c_{J'}(e_{i_0})|_{Y(s')} \oplus 0_{Z(s')}.
\]   

As a preliminary we make some observations about the cusp components.  The cusp components of $\Lambda$ above $h(\bfe')$ can be divided into three disjoint types:
\begin{enumerate}
\item[(A)] Those whose closure does not intersect $\pi_{x}^{-1}(h(\bfe))$.
\item[(B)] Those whose closure (viewed as a subset of $\Lambda$)
 intersects $\pi^{-1}_x(h(\bfe))$ in a cusp component of $\Lambda$ above $h(\bfe)$.  We enumerate these cusp components above $h(\bfe')$ as $\{D_1', \ldots, D_m'\}$ and the corresponding cusp components above $h(\bfe)$ as $\{D_1, \ldots, D_m\}$ so that $D_i \subset \overline{D_i'}$.  
\item[(C)] Those whose closure intersects $\pi^{-1}_x(h(\bfe))$ in two non-cusp sheets of $\Lambda$.  We denote the cusp components as $\{C_1', \ldots, C_n'\}$ and the corresponding pairs of sheets as $\{T^1_u, T^1_l, \ldots, T^n_u, T^n_l\}$, so that $T^i_u \subset \overline{(C_i')_u}$ and $T^i_l \subset \overline{(C_i')_l}$ where $(C_i')_u$ and $(C_i')_l$ denote the upper and lower branches of $C_i'$.   
\end{enumerate}

\medskip

The following properties of the cusp components are verified using the hypothesis that $s,s' \in \Lambda_k$.  

\begin{enumerate}
\item[(i)] The stratum $s'$ is never exceptional for cusp components from (A).  [Since $s'$ is adjacent to $s\in \mathcal{S}(\bfe)$ it cannot sit between the upper and lower sheets of a cusp component that does not intersect $\pi_x^{-1}(h(\bfe))$.]

\item[(ii)] Any cusp component above $h(\bfe)$ belongs to $\{D_1, \ldots, D_m\}$. 
 [By construction of the handle decomposition $H$, projections of cusp edges do not intersect $0$-cells and only intersect $1$-cells in their interior.  As a consequence, if a cusp edge appears above $h(\bfe)$ then it also appears above any $h(\bfe')$ with $h(\bfe) \leq h(\bfe')$.] 
 
\item[(iii)] For any $1 \leq i \leq m$, we have that $s$ is an exceptional stratum for $D_i$ if and only if $s'$ is an exceptional stratum for $D'_i$.  [This uses that $s,s' \in \Lambda_k$.]  

\item[(iv)]  For any $1 \leq i \leq n$, $s'$ is exceptional for $C_i'$ if and only if $T^i_l \in Y(s)$ and $T^i_u \notin Y(s)$.  [Since $s,s' \in \Lambda_k$,   $s$ and $s'$ are either both strictly above $\overline{(C'_i)_u}$ or both weakly below $\overline{(C'_i)_u}$.  A similar remark applies to $\overline{(C'_i)_l}$.  Thus, $s'$ is exceptional for $C_i'$ if and only if $s$ is weakly below $\overline{(C'_i)_u}$ and strictly above $\overline{(C'_i)_l}$.  The latter two conditions are equivalent to having $T^i_u \notin Y(s)$ and $T^i_l \in Y(s)$.]
\end{enumerate}

We can then write 
\[
Y(s) = Y_1 \oplus Y_2 \oplus Y_3  
\]
where $Y_1$ is spanned by those sheets in $Y(s)$ that belong to the closure of non-cusp sheets above $h(\bfe')$; $Y_2$ is spanned by the union of subsets $\{T^i_u, T^i_l\}$ such that $T^i_u \in Y(s)$ and $T^i_l \in Y(s)$; and  
\[
Y_3 = \mbox{Span} \{T^i_l\,|\, T^i_l \in Y(s) \, \mbox{and} \, T^i_u \notin Y(s)\}.
\]
Moreover, (using (i)) we can write
\[
Z(s') = Z_D \oplus Z_C
\]
where $Z_D$ (resp. $Z_C$) is spanned by $v_{D'_i}$ (resp. $v_{C'_i}$) such that $s'$ is exceptional for $D'_i$ (resp. for $C_i'$).

Now from the definition of  $X(s \rightarrow s')=(p, k +\ell)$ (see Definition \ref{def:maps}) we note the following:
\begin{itemize}
\item  $X(s \rightarrow s')$  maps $Y_1\stackrel{\cong}{\rightarrow} Y(s')$.  [The hypothesis that $s,s' \in \Lambda^F$ is used here to see that the sheets of $\Lambda(\bfe')$ below $s'$ are precisely those sheets whose closure contains the sheets of $\Lambda(\bfe)$ below $s$.]  
\item $X(s\rightarrow s')$ maps $Z(s) \stackrel{\cong}{\rightarrow} Z_D$.  [This uses (ii) and (iii).]
\item $X(s\rightarrow s')$ maps $Y_3 \stackrel{\cong}{\rightarrow} Z_C$.  [This uses (iv).]
\item $\ker X(s\rightarrow s') = Y_2$.  
[This uses (iv) and the previous three observations.]
\end{itemize}
Thus, $X(s\rightarrow s')$ is surjective with kernel $Y_2$.  Moreover, $Y_2$ is an acyclic subcomplex. [The definition of the boundary extensions $c_I(e_F)$ shows that $d^s_{i_0} T^i_u = T^i_l$ for all $T^i_u,T^i_l \in Y_2$.]   It follows that $X(s \rightarrow s')$ is a quasi-isomorphism.

\medskip

\noindent {\bf Checking (3) of Definition \ref{def:fun}:}  Suppose that $O, NE, NW, N \in \mathcal{S}(\bfe)$ are strata as described in (3) of Definition \ref{def:fun}.  We have $\dim h(\bfe) = 2$, and by construction $\bfe = (e_I, e_I)$ where $e_I \in \mathcal{E}$ has dimension $0$ or $1$.  I.e., $h(\bfe)$ is the interior of either a $0$-handle or $1$-handle of $H$ corresponding to either a vertex or edge of $\mathcal{E}$.  See Figure \ref{fig:HLabel}.

\medskip

\noindent{\bf Case 1:}  $\dim e_I = 0$.  

\medskip

\noindent {\bf Subcase A:}  Both $NE$ and $NW$ belong to non-cusp sheets above $h(\bfe)$.  Number the sheets of $\Lambda(\bfe) \cong \Lambda(e_I)$ as $\{S_1, \ldots, S_k, S_{k+1}, \ldots, S_n\}$ so that in the front projection $S_k$ and $S_{k+1}$ intersect  along $O$;  $z(S_k)> z(S_{k+1})$ where $NW$ exists; and $z(S_{k+1})> z(S_k)$ where $NE$ exists.  Remaining sheets are numbered with $S_1,\ldots, S_{k-1}$ appearing above $O$ and $S_{k+2}, \ldots, S_n$ appearing below $O$. 
 For any of the strata $s \in \{O, NE, NW, N\}$ we have $Z(s)=0$, and so the commutative diagram (\ref{eq:NW}) obtained from applying $F(\mathcal{C})$ to these strata becomes
\[
\xymatrix{ 
 & \mbox{Span}_\mathbb{K} \{S_k,S_{k+1}, S_{k+2}, \ldots, S_{n}\} & \\
\mbox{Span}_\mathbb{K} \{S_{k+1}, S_{k+2}, \ldots, S_{n}\} \ar[ru] & & \mbox{Span}_\mathbb{K} \{S_k, S_{k+2}, \ldots, S_{n}\} \ar[lu] \\
& \mbox{Span}_\mathbb{K} \{S_{k+2}, \ldots, S_{n}\} \ar[ru]  \ar[lu] &
} 
\]
where all arrows are inclusions and all differentials are  induced by $d=c(e_I):V(e_I) \rightarrow V(e_I)$ where $\{c(e_I)\} = \mathcal{C}$ is the CHD.
    Quotienting by the sequence of acyclic sub-complexes, $\mathit{Tot}^\cdot\left(F(NE) \rightarrow \mbox{Im}\left( F(NE)\rightarrow F(N) \right)\right)$, then $\mathit{Tot}^\cdot\left( F(O) \rightarrow \mbox{Im}\left(F(O) \rightarrow F(NW)\right)\right)$ leaves the acyclic complex $\mathit{Tot}^\cdot\left( \mbox{Span}\{S_{k+1}\} \rightarrow \mbox{Span}\{S_{k+1}\}\right)$ (with the $S_{k+1}$ in $F(NW)$ and $F(N)$ respectively).

\medskip

\noindent {\bf Subcase B:} One of $NE$ and $NW$ belongs to a cusp component $C$ of $\Lambda$ above $h(\bfe)$.  (Note that this case only occurs when $e_I \in \mathcal{E}$ sits below a cusp-sheet intersection point of the front projection; see Figure \ref{fig:FrontSing}.) 
  We may assume $NE$ belongs to a cusp component.  Then, the maps $F(NE) \rightarrow F(N)$ and $F(O) \rightarrow F(NW)$ restrict to isomorphisms 
 of sub-complexes $Y(NE) \stackrel{\cong}{\rightarrow} Y(N)$ and $Y(O) \stackrel{\cong}{\rightarrow} Y(NW)$.  Quotienting the total complex by the corresponding sequence of acyclic sub-complexes leaves just the $Z(s)$ components and they have the form
\[
\xymatrix{ 
 & \mbox{Span}_\mathbb{K} \{v_C\} & \\
\mbox{Span}_\mathbb{K} \{v_C\} \ar[ru] & &  0 \ar[lu] \\
& 0 \ar[ru]  \ar[lu] &
}  \quad \raisebox{-1.5cm}{\mbox{or}} \quad
\xymatrix{ 
 & 0 & \\
0 \ar[ru] & &  \mbox{Span}_\mathbb{K} \{v_C\} \ar[lu] \\
& \mbox{Span}_\mathbb{K} \{v_C\} \ar[ru]  \ar[lu] &
}
\]
depending on whether $NE$ belongs to the lower or upper branch of $C$.

\medskip

\noindent{\bf Case 2:}  $\dim e_I = 1$.  In this case, the $1$-cell $e_I=e_{i_0i_1}$ must sit below a single crossing arc of $\pi_{xz}(\Lambda)$.  For each $s \in \{O, NE, NW, N\}$, the preliminary vector space is $X(s) = Y(s)$ and we can label sheets as in Case 1A.  Then, $F(s)$ is the mapping cylinder of the simplex diagram on $e_I$ corresponding to the chain map $f:(X(s), d_-) \rightarrow (X(s), d_+)$ where $f,d_-, d_+$ are induced by the maps $c(e_I), c_I(e_{i_0})$ and $c_I(e_{i_1})$ from the CHD $\mathcal{C}$.  One can then establish acyclicity of the total complex by arguing as in Case 1A where the final acyclic complex is now 
\[
\mathit{Tot}^\cdot\left( \begin{array}{c} \mbox{Span}\{S_{k+1}\} \oplus  \mbox{Span}\{S_{k+1}\}[-1] \oplus  \mbox{Span}\{S_{k+1}\} \quad \quad \quad \\  \quad \quad \quad \rightarrow\mbox{Span}\{S_{k+1}\} \oplus  \mbox{Span}\{S_{k+1}\}[-1] \oplus  \mbox{Span}\{S_{k+1}\} \end{array} \right).
\]
Alternatively, one can observe that the total complex of the $F(s)$ is itself the mapping cylinder of a simplex diagram corresponding to the map of total complexes of the $(X(s), d_-)$ and $(X(s), d_+)$ induced by $f$.  Moreover, for this simplex diagram, the complex associated to a vertex is exactly the total complex from Case 1A, and hence Proposition \ref{prop:vertex} shows that  the mapping cylinder of this simplex diagram is acyclic.

\end{proof}

Together with Proposition \ref{prop:CHDbiject} and Corollary \ref{cor:Equiv}, this completes the construction of sheaves from  augmentations as stated in Theorem \ref{thm:main}.  Explicitly, we have map from augmentations of the simplicial DGA to $\mathbf{Fun}^\bullet_\Lambda(\mathcal{S}, \mathbb{K})$ sending $\epsilon:\mathcal{A}(\Lambda, \mathcal{E})\rightarrow \mathbb{K}$ to $F(\mathcal{C}_\epsilon)$ where $\mathcal{C}_\epsilon$ is the CHD corresponding to $\epsilon$ under the bijection from Proposition \ref{prop:CHDbiject}.  Since $\Gamma_{\mathcal{S}}: \mathbf{Sh}^\bullet_{\Lambda}(\mathcal{S}, \mathbb{K}) \rightarrow \mathbf{Fun}^\bullet_\Lambda(\mathcal{S}, \mathbb{K})$ is a quasi-equivalence, there exists a complex of sheafs $\mathcal{F} \in  \mathbf{Sh}^\bullet_{\Lambda}(\mathcal{S}, \mathbb{K})$, unique up to isomorphism in $\mathbf{Sh}^\bullet_{\Lambda}(\mathcal{S}, \mathbb{K})$, such that $\Gamma_{\mathcal{S}}(\mathcal{F})$ is isomorphic to $F(\mathcal{C}_\epsilon)$ in $\mathbf{Fun}^\bullet_\Lambda(\mathcal{S}, \mathbb{K})$.

\subsection{Closing remarks}
As usual, cf. \cite{STZ, NRSSZ}, the sheaves constructed from augmentations in Theorem 1.1 have microlocal rank $1$, i.e. at co-vectors $(q,\xi) \in SS(\mathcal{F})$ the Morse group, $Mo_{q,f}(\mathcal{F})$, from (\ref{eq:MorseGroup}) has rank $1$.  [For the corresponding combinatorial sheaves, at smooth points of $\pi_{xz}(\Lambda)$, the Morse group is the cone of the {\it upward} generization map.]  However, a modification of the construction from Theorem \ref{thm:main} can be used to produce sheaves with micro-local rank $n>1$ from $n$-dimensional representations of the simplicial DGA.  Specifically, one should carry out the constructions of Sections \ref{sec:CHD} and \ref{sec:SheafConstruction} but, in the definition of the vector spaces $V(e_I)$ and $X(s) = Y(s) \oplus Z(s)$ replace the generators associated to sheets $S_i$ of $\Lambda$ (or, in the case of the $Z(s)$, with cusp components $C$) with $n$-dimensional subspaces $W_{S_i}$ (or $W_{C}$).  

At least in the case of $\Z/2$-coefficients and $M= \R^2$, it is reasonable to conjecture that the map from augmentation to sheaves constructed may be essentially surjective onto the sub-category $\mathcal{C}_1(\Lambda, \mathbb{K}) \subset \ShLambda$ of micro-local rank $1$ sheaves with acyclic stalks at $\infty$.  However, in the case of higher rank representations, this is likely to not be the case, since the constructed sheaves will have trivial microlocal monodromy, cf. \cite{STZ, KS}.  To obtain sheaves with non-trivial micro-local monodromy one may be able to expand the constructions of Section \ref{sec:CHD} and \ref{sec:SheafConstruction}  
 to allow the subspaces $W_{S_i}$ to arise from a non-trivial (combinatorial) local system of vector spaces on $\Lambda$.  In terms of DGA representations, this would require a version of the simplicial DGA with (fully non-commutative) $\Z[\pi_{1}(\Lambda)]$-coefficients; see eg. \cite{CDGG} for discussion of the Legendrian contact homology DGA with $\Z[\pi_1(\Lambda)]$-coefficients.

\end{document}